\documentclass[oneside,english]{amsart}
\usepackage[a4paper, left=2.7cm, right=2.7cm, top=3.5cm, bottom=5cm]{geometry}
\usepackage[dvipsnames]{xcolor}
\usepackage{setspace} 
\usepackage{amsmath,amsfonts,amssymb,wasysym} 
\usepackage{caption}
\usepackage{todonotes}
\usepackage{stmaryrd}
\usepackage[]{epsf,epsfig,latexsym}
\usepackage{enumerate}
\usepackage{mathrsfs}
\usepackage[inline]{enumitem}
\usepackage{comment}
\usepackage{mathtools}
\usepackage{braket}		
\usepackage[normalem]{ulem}		
\usepackage{nicefrac}
\usepackage{cancel}
\usepackage{xcolor}
\usepackage[pdfencoding=auto, psdextra, bookmarksdepth=subsection,hyperindex=true,pdfborder={0 0 0}]{hyperref}
\usepackage{doi}
\usepackage{cleveref}
\crefname{subsection}{subsection}{subsections}
\usepackage{tikz}
\usetikzlibrary{shapes,arrows,automata,matrix,fit,patterns,external}
\usepackage{pgfplots} 
\usepackage{tikz-cd}
\usepackage{fdsymbol}
\usepackage{thmtools}
\graphicspath{{./img_pdf/}}
\theoremstyle{plain}
\newtheorem{theorem}{Theorem}[section]
\newtheorem{lemma}[theorem]{Lemma}
\newtheorem{corollary}[theorem]{Corollary}
\newtheorem{proposition}[theorem]{Proposition}

\newtheorem{claim}[theorem]{Claim}
\newtheorem{remark}[theorem]{Remark}
\newtheorem{example}[theorem]{Example}

\theoremstyle{definition}
\newtheorem{definition}[theorem]{Definition}
\theoremstyle{plain}
\newtheorem{thmx}{Theorem}

\newcommand{\vertexlabels}[4]{
    \draw (0,0) -- (0.5,0) node [anchor=west] {\textbf{$#1$}};
    \draw (0,0) -- (0,0.5) node [anchor=south] {\textbf{$#2$}};
    \draw (0,0) -- (-0.5,0) node [anchor=east] {\textbf{$#3$}};
    \draw (0,0) -- (0,-0.5) node [anchor=north] {\textbf{$#4$}};

    \filldraw (0,0) circle (1.5pt);
}
\newcommand{\vertexlabeldegreethree}[3]{
    \draw (0,0) -- (0.5,0) node [anchor=west] {\textbf{$#1$}};
    \draw (0,0) -- (0,-0.5) node [anchor=north] {\textbf{$#2$}};
    \draw (0,0) -- (-0.5,0) node [anchor=east] {\textbf{$#3$}};

    \filldraw (0,0) circle (1.5pt);
}

\newcommand{\upborder}[1]{
    \draw (0,0) -- (0.5,0) node [anchor=west] {\textbf{$#1$}};
    \draw [->] (0,0) -- (0,0.5);
    \draw [-<] (0,0) -- (0,-0.5);

    \filldraw (0,0) circle (1.5pt);
}

\newcommand{\centerborder}[1]{
    \draw (0,0) -- (0.5,0) node [anchor=west] {\textbf{$#1$}};
    \draw [->] (0,0) -- (0,0.5);
    \draw [->] (0,0) -- (0,-0.5);

    \filldraw (0,0) circle (1.5pt);
}

\newcommand{\downborder}[1]{
    \draw (0,0) -- (0.5,0) node [anchor=west] {\textbf{$#1$}};
    \draw [-<] (0,0) -- (0,0.5);
    \draw [->] (0,0) -- (0,-0.5);

    \filldraw (0,0) circle (1.5pt);
}

\newcommand{\hypersquare}[3]{
\pgfmathsetmacro{\k}{1/pow(#1,#2-1)}
\pgfmathsetmacro{\sy}{0}
\foreach \i in {1,...,#2}{
\pgfmathsetmacro{\j}{pow(#1,\i-1)}
        \begin{scope}[shift={(0,(\sy*\k) cm}]
        \draw[step={\j*\k},gray, thick] (0,0) grid ({pow(#1,#2-1)*#3*\k},\j*\k);
        \end{scope}
        \pgfmathsetmacro{\sy}{\sy + \j}
        } 
}

\newcommand{\ZZ}{\mathbb{Z}}			
\newcommand{\NN}{\mathbb{N}}			
\newcommand{\RR}{\mathbb{R}}			
\newcommand{\QQ}{\mathbb{Q}}			
\newcommand{\CC}{\mathbb{C}}            
\newcommand{\HH}{\mathbb{H}^2}            
\DeclareMathOperator{\IM}{Im}           
\DeclareMathOperator{\PSO}{PSO}
\DeclareMathOperator{\PSL}{PSL}
\DeclarePairedDelimiter\norm{\lVert}{\rVert}	
\makeatletter
\newcommand{\supp}{
    \operatorname{\mathrm{supp}}%
}

\newcommand{\define}[1]{\textbf{#1}}

\title{%
    Medvedev degrees of SFTs on cocompact Fuchsian groups
}

\author{Sebasti\'an Barbieri, Nicanor Carrasco-Vargas, and Paul Toussaint}

\newcommand{\Addresses}{{
        \bigskip
        
        \hskip-\parindent   S.~Barbieri, \textsc{Departamento de Matem\'{a}ticas, Universidad de Chile, Santiago, Chile.}\par\nopagebreak
        \textit{E-mail address}: \texttt{sbarbieri@uchile.cl}
        
        \medskip
        
        \hskip-\parindent   N.~Carrasco-Vargas, \textsc{Faculty of Mathematics and Computer Science, Jagiellonian University, Krakow, Poland.}\par\nopagebreak
        \textit{E-mail address}: \texttt{nicanor.vargas@uj.edu.pl}
        
        \medskip
        
        \hskip-\parindent   P.~Toussaint, \textsc{Université Claude Bernard Lyon 1, Institut Camille Jordan, Lyon, France.}\par\nopagebreak
        \textit{E-mail address}: \texttt{toussaint@math.univ-lyon1.fr}
}}
\date{}

\pgfplotsset{compat=1.18}
\begin{document}
    
\begin{abstract}

We prove that for any cocompact Fuchsian group the class of Medvedev degrees attained by subshifts of finite type is the class of $\Pi_1^0$ degrees. This result relies on the construction of a rigid hierarchical structure in a graph model of the hyperbolic plane, and on the fact that every such group admits a representation in $\operatorname{PSL}_2(\RR)$ such that the matrices in the image have coefficients that are computable real numbers.

\medskip
        
        
        \textit{Secondary:} 37B02, 
                03D78, 
                20F10. 
\end{abstract}

\maketitle


\section{Introduction}

Given two subsets $X$ and $Y$ of the Cantor space $\{0,1\}^{\NN}$, we say that $X$ is Medvedev reducible to $Y$ if there exists a partial computable function that maps elements of $Y$ into elements of $X$. The equivalence classes induced by this relation are called \define{Medvedev degrees}, and can be thought of as a measure of how algorithmically complex it is to compute one of the elements of a set.

In the context of symbolic dynamics, given a finite set $A$ with $|A|\geq 2$ and an infinite and finitely generated group $G$ with decidable word problem, one can identify in a computably meaningful way the space $A^G$ with $\{0,1\}^{\NN}$. Hence we can define the Medvedev degree of a subshift $X\subset A^G$. It turns out that the Medvedev degree is an invariant of topological conjugacy that behaves in some ways similarly to topological entropy (see~\cite{barbieri2024medvedev}). For instance, it does not increase under topological factor maps, even in nonamenable groups. It is an open problem to determine which groups admit subshifts of finite type (SFTs) with nonzero Medvedev degree, and more generally, to understand how the class of Medvedev degrees realized by SFTs changes as a function of the underlying group. 

Given a finitely generated group $G$ with decidable word problem, a natural restriction on the Medvedev degrees of both $G$-SFTs and, more generally, effective $G$-subshifts (those that can be defined by a recursively enumerable set of forbidden patterns) is that they must be $\Pi_1^0$ degrees, that is, equivalence classes that contain a representative which is an effectively closed set ($\Pi_1^0$ set). For $G=\ZZ$, Miller~\cite{Miller_2012_twonotesonSubshifts} showed that effective $\ZZ$-subshifts in fact do attain all $\Pi_1^0$ Medvedev degrees. It is also worth noting that $\ZZ$-SFTs always contain finite orbits and thus can only attain the minimal degree which consists of all sets that contain a computable element. This is in stark contrast to the higher-dimensional case: a result due to Simpson~\cite{Simpson_2014_Medvedev} shows that for $d\geq 2$, $\ZZ^d$-SFTs also attain every possible $\Pi_1^0$ Medvedev degree.

For arbitrary groups, the first two authors~\cite{barbieri2024medvedev} recently extended the classification of Simpson to several classes of groups with decidable word problem, such as all virtually polycyclic groups which are not virtually cyclic, direct products of infinite finitely generated groups, and finitely generated branch groups. In fact, it is worth noting that we are not aware of any recursively presented group $G$ in which the class of Medvedev degrees attained by $G$-SFTs is neither reduced to the trivial degree nor or the whole class of $\Pi_1^0$ degrees. See~\cite[Conjecture 6.2]{barbieri2024medvedev}.

An interesting class of groups for which the classification was still unknown is that of the hyperbolic surface groups, that is, the fundamental groups of closed orientable surfaces of genus at least $2$. In~\cite{barbieri2024medvedev} it was shown that some SFTs on those groups can attain non-zero Medvedev degrees, but the class of said degrees was not characterized. Our main result is that this class is indeed the class of $\Pi_1^0$ degrees for a large class of groups which are quasi-isometric to the hyperbolic plane.

\begin{thmx}\label{mainthm:degrees}
    For every cocompact Fuchsian group $G$, the class of $G$-SFTs attains all $\Pi_1^0$ Medvedev degrees.
\end{thmx}

The proof of~\Cref{mainthm:degrees} relies on three preliminary results. In~\cite[Corollary 4.24]{barbieri2024medvedev} it was shown that if two finitely presented groups are quasi-isometric through a computable map, then they have the same space of Medvedev degrees of SFTs; moreover, if two groups admit computable quasi-isometries to a common computable metric space, such as the upper half-plane model $\HH = \{z \in \CC: \operatorname{Im}(z)>0\}$, then they are computably quasi-isometric~\cite[Lemma 4.26]{barbieri2024medvedev}. The first preliminary result that we need is the statement that all cocompact Fuchsian groups are computably quasi-isometric. By the (computable) Schwartz–Milnor lemma, it suffices to verify the following property. 
\begin{thmx}\label{mainthm:computable_rep}
    Every cocompact Fuchsian group admits a representation in $\PSL_2(\RR)$ with matrices whose coefficients are computable real numbers.
\end{thmx}
This statement follows from a known theorem \cite{waterman_fuchsian_1985}, which asserts that the same coefficients can be taken algebraic. For the sake of completeness, we present a distinct argument. Our proof for \Cref{mainthm:computable_rep} is new, and we believe that the argument can be adapted to more general classes of groups. In~\Cref{sec:fuchsian} we will provide a brief introduction to Fuchsian groups and prove~\Cref{mainthm:computable_rep}. The main technical tool is a famous theorem of Weil on the topology of the space of representations of discrete groups by connected Lie groups. 

Given a cocompact Fuchsian group $G$ and a representation as in~\Cref{mainthm:computable_rep}, a computable quasi-isometry to the upper half-plane model is given by the orbit map given by $g \mapsto g\cdot i \in \HH$ induced by the matrices in the representation, where the action is the one inherited from $\PSL_2(\RR)$ by Möbius transformations. Therefore,~\Cref{mainthm:computable_rep} and the results cited above reduce the proof of~\Cref{mainthm:degrees} to computing the Medvedev degrees of SFTs on a single cocompact Fuchsian group.

Nonetheless, our strategy is not to construct all $\Pi_1^0$ degrees on a cocompact Fuchsian group. Instead, we do it on an algebraic structure which codifies a family of graphs that are quasi-isometric to the hyperbolic plane, but where the local geometry is more convenient to build hierarchical structures. Given an integer $k \geq 2$, we consider the graph $H_k = (V_k,E_k)$ whose vertex set is $V_k=\ZZ^2$ and whose edges are \[ E_k = \bigl\{ \{(n,m), (n+1,m)\} : n,m \in \ZZ \bigl\} \cup \bigl\{ \{(n,m), (kn ,m-1)\} : n,m \in \ZZ \bigr\}. \]
 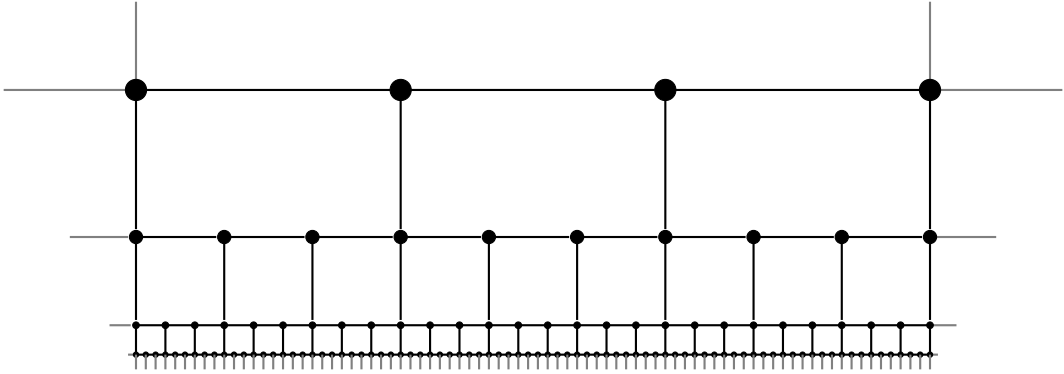
\begin{figure}[ht!]
    \centering
    \begin{tikzpicture}[scale = 3.5]
   \draw[thick, gray] (0,1.04) -- (0,1+1/3);
   \draw[thick, gray] (3,1.04) -- (3,1+1/3);
   \draw[thick, gray] (-0.04,1) -- (-0.5,1);
   \draw[thick, gray] (-0.03,4/9) -- (-0.25,4/9);
   \draw[thick, gray] (-0.02,1/9) -- (-0.1,1/9);
   \draw[thick, gray] (0,0) -- (-0.030,0);
   \draw[thick, gray] (3,1) -- (3.5,1);
   \draw[thick, gray] (3,4/9) -- (3.25,4/9);
   \draw[thick, gray] (3,1/9) -- (3.1,1/9);
   \draw[thick, gray] (3,0) -- (3.03,0);
    \foreach \i in {0,...,2}{
        \draw[thick] (\i,1) -- (\i+0.96,1);
    }
    \foreach \i in {0,...,8}{
        \draw[thick] (\i/3,4/9) -- (\i/3+1/3-0.03,4/9);
    }
    \foreach \i in {0,...,26}{
        \draw[thick] (\i/9,1/9) -- (\i/9+1/9-0.01,1/9);
    }
    \foreach \i in {0,...,3}{
        \draw[fill = black] (\i,1) circle (0.04);
        \draw[thick] (\i,1) -- (\i,4/9+0.03);
    }
    \foreach \i in {0,...,9}{
        \draw[fill = black] (\i/3,4/9) circle (0.025);
        \draw[thick] (\i/3,4/9) -- (\i/3,1/9+0.02);
    }
    \foreach \i in {0,...,27}{
        \draw[fill = black] (\i/9,1/9) circle (0.0125);
         \draw[thick] (\i/9,1/9) -- (\i/9,0.01);
    }
    
    \draw[thick] (0,0) -- (3,0);
    \foreach \i in {0,...,81}{
        \draw[fill = black] (\i/27,0) circle (0.01);
        \draw[thick, gray] (\i/27,0) -- (\i/27,-1/18); 
    }
    \end{tikzpicture}
    \caption{A finite portion of the graph $H_3$.}
    \label{fig:grafoH3}
\end{figure}
A portion of $H_3$ is shown in~\Cref{fig:grafoH3}. We shall consider a structure $\mathcal{H}_k$ which codifies a compact family of rooted graphs that locally look like $H_k$, equipped with a partial action of a free monoid which acts changing the root by moving along labeled edges. The advantage of these structures, which we call \define{blueprints}, is that they can be codified by a subshift of finite type on any finitely presented group which is quasi-isometric to any of their graphs. In particular, using~\cite[Theorem 5.15]{BarBit2026}, we are able to show the following.

\begin{thmx}\label{mainthm:to_blueprint}
    Let $G$ be a finitely presented group which is computably quasi-isometric to $\HH$, and let $k\geq 2$. The classes of $G$-SFTs and $\mathcal H_k$-SFTs attain the same class of Medvedev degrees. 
\end{thmx}

A brief introduction to Medvedev degrees, subshifts, and blueprints shall be given in~\Cref{sec:QI}. In that section we shall also introduce the blueprints $\mathcal{H}_k$ and show that they satisfy the hypotheses needed to apply~\cite[Theorem 5.15]{BarBit2026}. 

With~\Cref{mainthm:to_blueprint} we have reduced our goal to the construction of SFTs with arbitrary $\Pi_1^0$ Medvedev degree on $\mathcal{H}_k$ for some $k \geq 2$. We work with $k=3$ for convenience, and prove that $\mathcal H_3$-SFTs indeed attain all $\Pi_1^0$ Medvedev degrees.%
\begin{thmx}\label{mainthm:grafo_medved}
    The space of Medvedev degrees of $\mathcal{H}_3$-SFTs is the set of all $\Pi_1^0$ Medvedev degrees.
\end{thmx}

Putting~\Cref{mainthm:grafo_medved} along with~\Cref{mainthm:to_blueprint} and (the consequence of)~\Cref{mainthm:computable_rep}, we obtain the proof of~\Cref{mainthm:degrees}.

The proof of~\Cref{mainthm:grafo_medved} is the most technical part of this article, and spans~\Cref{sec:sofic,sec:medvedev,sec:appendix}. Its general lines are very similar to those in Simpson's proof about Medvedev degrees of $\ZZ^d$-SFTs~\cite{Simpson_2014_Medvedev}. That is, we construct an SFT on $\mathcal{H}_3$ which contains a hierarchical collection of grids, which we use to implement space-time diagrams of a Turing machine. Using that structure, we show that we can prescribe the Medvedev degree of the $\mathcal H_3$-SFT by synchronizing information along all grids and by encoding the Turing machine suitably. A nontrivial difficulty, which leads to some technical considerations, is ensuring that the Medvedev degree of the SFT only depends on this machine, and is not affected by any of the other elements of the construction.

The basic ingredient one needs for this argument is an adaptation of Robinson's tiling of the plane~\cite{robinson_undecidability_1971} to the hyperbolic setting. Such an adaptation was first obtained by Margenstern in~\cite{Margenstern_2008_domino_hyp} (implemented as heptagonal Wang tiles of $\HH$), and later simplified by Goodman-Strauss in~\cite{goodman-strauss_hierarchical_2010} (implemented as Wang tiles on the faces of $H_5$, and more generally $H_k$ with $k\equiv 1\mod 4$). For the sake of completeness, and in order to verify some technical requirements related to Medvedev degrees and to the blueprint formalism, we give a complete presentation of a $\mathcal H_3$-SFT which can be considered as a hyperbolic version of Robinson's tiling (in the sense of a $\ZZ^2$-SFT). 

The construction of the aforementioned $\mathcal H_3$-SFT is long and tedious, and we defer it to~\Cref{sec:appendix}. In order to show~\Cref{mainthm:grafo_medved}, we instead work with one of its topological factors, which only carries the essential hierarchical structure. This subshift, denoted $\mathfrak X_1$, is defined and studied in~\Cref{sec:sofic}.~\Cref{mainthm:grafo_medved} is then proved in~\Cref{sec:medvedev}, under the assumption that $\mathfrak{X}_1$ is a factor of an SFT with zero Medvedev degree and satisfying a further technical condition. This assumption is finally proved in~\Cref{sec:appendix}.




\subsection*{Acknowledgments} We would like to thank M. Sablik for explaining to us the method of putting bits on the borders of Robinson's SFT, which we use to prove \Cref{mainthm:grafo_medved}. S. Barbieri was supported by ANID FONDECYT regular 1240085, AMSUD240026 and ECOS230003. N. Carrasco-Vargas was supported by a grant from the Priority Research Area SciMat under the Strategic Programme Excellence Initiative at Jagiellonian University, partially supported by the Simons Foundation grant (award no. SFI-MPS-T-Institutes-00010825), and from State Treasury funds as part of a task commissioned by the Minister of Science and Higher Education under the project “Organization of the Simons Semesters at the Banach Center - New Energies in 2026-2028” (agreement no. MNiSW/2025/DAP/491). P. Toussaint was supported by ANID-Subdirecci\'on de Capital Humano/Mag\'ister Nacional/2025-22251030 and ECOS230003.    
\subsection*{AI statement}
Kimi AI and ChatGPT were used to search for typos and small mistakes. The writing of this text and all of the proofs herein are human-generated.
\section{Computable representations of Fuchsian groups}\label{sec:fuchsian}

In this section we show~\Cref{mainthm:computable_rep}, that is, that every cocompact finitely generated Fuchsian group admits a representation in $\operatorname{PSL}_2(\RR)$ whose coefficients are computable. We shall assume the reader has some familiarity with basic notions of computability and provide a brief introduction to Fuchsian groups. We refer the reader interested in a more extensive review of Fuchsian groups to~\cite{katok1992fuchsian},~\cite[Chapter 1]{shimura1971introduction} and~\cite[Chapter 9]{EinsiedlerWard_ergodictheory_book}.

\subsection{Fuchsian groups}

We denote by $\operatorname{SL}_2(\RR)$ the special linear group of $2 \times 2$ matrices with determinant $1$, and endow it with its usual topology. Let $\PSL_2(\RR) = \operatorname{SL}_2(\RR) / \{\pm I_2\}$, where $I_2$ is the $2 \times 2$ identity matrix. The quotient topology is well defined on $\PSL_2(\RR)$ and makes it a connected locally compact topological group. Moreover, there is a natural smooth structure on $\PSL_2(\RR)$ that makes it a Lie group. See \cite[Chapters 7 and 21]{lee2003introduction} for an introduction to Lie groups.

Let $\HH = \{ z \in \CC : \IM (z) > 0 \}$ be the upper half-plane model of the hyperbolic plane. The group $\PSL_2(\RR)$ acts from the left on $\HH$ by Möbius transformations, namely,
\begin{equation*}
    \text{if } g = \pm\begin{pmatrix}
    a & b \\
    c & d
\end{pmatrix} \in \PSL_2(\RR) \text{ then } g\cdot z = \frac{az+b}{cz+d} \text{ for all }z\in\HH.
\end{equation*}

Given $\Gamma \leqslant \PSL_2(\RR)$, we denote by $\Gamma \backslash \HH = \{\Gamma x : x \in \HH \}$ the space of orbits under this action with the quotient topology. $\Gamma$ is said to be \textit{cocompact} if $\Gamma \backslash \HH$ is compact.

\begin{definition}
    A \define{Fuchsian group} is a subgroup $\Gamma \leqslant \PSL_2(\RR)$ such that the induced topology on $\Gamma$ is the discrete topology.
\end{definition}

Each subgroup $\Gamma\leqslant \PSL_2(\RR)$ also has a natural left action on $\PSL_2(\RR)$ by left multiplication whose orbit space we denote by $\Gamma \backslash \PSL_2(\RR)$. The following proposition is well-known.

\begin{proposition}
    \label{prop:cocompact}
    A Fuchsian group $\Gamma$ is cocompact if and only if $\Gamma \backslash \PSL_2(\RR)$ is compact.
\end{proposition}
\begin{proof}
    Let $K = \{g \in \PSL_2(\RR) : g\cdot i = i \}$ be the stabilizer subgroup of $i$. A straightforward computation shows that $K$ is the projective special orthogonal group, that is, $K = \PSO_2(\RR) = \operatorname{SO}_2(\RR) / \{ \pm I_2 \}$, where
    \begin{equation*}
        \operatorname{SO}_2(\RR) = \left\{ \begin{pmatrix}
            \cos \theta & -\sin \theta \\
            \sin \theta & \cos \theta
        \end{pmatrix} : \theta \in \RR \right\}.
    \end{equation*}
    It follows that we have a natural bijection $\varphi \colon \PSL_2(\RR)/\PSO_2(\RR) \to \HH$ given by $\varphi(g\PSO_2(\RR)) = g\cdot i$. Moreover, this bijection is a homeomorphism (see~\cite[Theorem 1.1]{shimura1971introduction}). Put $S = \PSL_2(\RR)/\PSO_2(\RR)$. As before, $\Gamma$ acts on $S$ by left multiplication and we can form the double quotient \[\Gamma \backslash S = \Gamma \backslash \left( \PSL_2(\RR) / \PSO_2(\RR)\right).\] The homeomorphism $\varphi$ preserves the respective actions of $\Gamma$, so $\Gamma \backslash S$ is homeomorphic to $\Gamma \backslash \HH$. Since $\PSO_2(\RR)$ is compact, it follows that $\Gamma \backslash \PSL_2(\RR)$ is compact if and only if $\Gamma \backslash S$ is compact (see \cite[Proposition 1.9]{shimura1971introduction}) and thus the result follows.
\end{proof}

\subsection{Computable representations}

\begin{definition}
    A real number $x$ is \define{computable} if there exists a computable map $f\colon \NN \to \QQ$ with the property that for each $n \in \NN$ we have $|f(n)-x|\leq 2^{-n}$.
\end{definition} 
In less technical terms, a real number $x$ is computable if and only if there is an algorithm that on input $n$ computes a rational number which approximates $x$ with an error of at most $2^{-n}$.

Next we will show that every cocompact Fuchsian group $\Gamma$ can be represented by matrices whose coefficients are computable real numbers. We shall need the following result of A. Weil.

\begin{theorem}[Weil~\cite{weil_discrete_1960}]\label{thm:Weil}
    Let $G$ be a connected Lie group and $\Gamma$ a discrete group. Denote
    \begin{equation*}
        \mathcal{R} = \{\varphi\colon \Gamma \to G \mid \varphi \text{ is a group homomorphism} \}
    \end{equation*}
    and give it the subspace topology inherited from the product topology on $G^\Gamma$. Let $\mathcal{R}_0$ be the subset of all injective group homomorphisms $\varphi\colon \Gamma \to G$ such that $\varphi(\Gamma)$ is discrete in $G$ with compact quotient space $\varphi(\Gamma)\backslash G$. Then $\mathcal{R}_0$ is an open subset of $\mathcal{R}$.
\end{theorem}

It is well known (for instance, by the Schwarz-Milnor lemma) that all cocompact Fuchsian groups are finitely generated (in fact, finitely presented, as they are hyperbolic). Applying~\Cref{thm:Weil} to our context, we have the following corollary.

\begin{corollary}
    \label{coro:weil}
    Let $\Gamma$ be a cocompact Fuchsian group generated by $\gamma_1,\dots,\gamma_n$. There exists $\delta > 0$ such that, if $\varphi\colon \Gamma \to \PSL_2(\RR)$ is a group homomorphism with $\norm{\gamma_i - \varphi(\gamma_i)} < \delta$ for every $i=1,\dots,n$, then $\varphi(\Gamma)$ is a cocompact Fuchsian group isomorphic to $\Gamma$.
\end{corollary}

\begin{proof}
    Let $G = \PSL_2(\RR)$ in the previous theorem. By proposition \ref{prop:cocompact}, the set $\mathcal{R}_0$ coincides with the subset of all $\varphi\colon \Gamma \to \PSL_2(\RR)$ such that $\varphi$ is an injective group homomorphism and $\varphi(\Gamma)$ is a cocompact Fuchsian group. Since $\Gamma$ is generated by $\gamma_1,\dots,\gamma_n\in G$, we may identify $\mathcal{R}$ with a subspace of $(\PSL_2(\RR))^n$. Therefore, we may choose $\delta>0$ such that the $\delta$-neighborhood of $(\gamma_1,\dots,\gamma_n)$ is contained in $\mathcal{R}_0$.
\end{proof}

Now we are ready to prove~\Cref{mainthm:computable_rep}. Namely, we prove that every cocompact Fuchsian group admits a representation with computable real numbers.


\begin{proof}[Proof of~\Cref{mainthm:computable_rep}]
    Let $\Gamma$ be a cocompact Fuchsian group generated by $\gamma_1, \dots, \gamma_n\in \PSL_2(\RR)$. By adjoining the finite number of coefficients of each $\gamma_i$ to $\QQ$, we may regard $\Gamma$ as a subgroup of $\PSL_2(K)$, where $K = \QQ(\alpha_1,\dots,\alpha_k,a_1,\dots,a_l)$, $\alpha_1, \dots, \alpha_k$ are algebraically independent over $\QQ$ and $a_1,\dots,a_l$ are algebraic over $L = \QQ(\alpha_1,\dots,\alpha_k)$. Then, $L \subset K$ is a finitely generated algebraic field extension, so it is a finite field extension. By the primitive element theorem (see for example~\cite{dummit2004abstract}), there exists a single element $a \in K$ such that $K = L(a)$. In other words, $K = \QQ(\alpha_1,\dots,\alpha_k,a)$, where $a$ is algebraic over $L$.

    Let $p \in L[x]$ be the minimal polynomial of $a$ over $L$. $L$ is isomorphic to $\QQ(x_1,\dots,x_k)$, the field of rational functions on $k$ variables, so we may regard $p$ as a rational function $p(x_1,\dots,x_k,x)$ with rational coefficients. Since $L$ is of characteristic 0 and $p$ is irreducible in $L[x]$, $\frac{\partial p}{\partial x}(\alpha_1,\dots,\alpha_k,a) \neq 0$. Hence, by the implicit function theorem, there is an open set $U \subseteq \RR^k$ which contains the point $(\alpha_1,\dots,\alpha_k)$, and a continuous function $f \colon U \to \RR$ such that $f(\alpha_1,\dots,\alpha_k)=a$ and $p(x_1,\dots,x_k, f(x_1,\dots,x_k)) = 0$ for every $(x_1,\dots,x_k) \in U$. Now, if $\vec{\beta} = (\beta_1,\dots,\beta_k)$ is an algebraically independent $k$-tuple which is in $U$, then we have an isomorphism of fields $K \cong K' : = \QQ(\beta_1,\dots,\beta_k, f(\vec{\beta}))$. To see this, start with the isomorphism $h \colon L \to L' := \QQ(\beta_1,\dots,\beta_k)$ which sends $\alpha_i$ to $\beta_i$ for each $i=1,\dots,k$ and denote $b = f(\vec{\beta})$. Then, $h$ naturally extends to an isomorphism $L[x] \to L'[x]$, and thus we have an induced isomorphism of the quotients $L[x]/\langle p \rangle \cong L'[x]/\langle h(p) \rangle$. From elementary field theory, we know that $L(a) \cong L[x]/\langle p \rangle$ and $L'(b) \cong L'[x]/\langle q \rangle$, where $q \in L'[x]$ is the minimal polynomial of $b$, but observe that by construction of $b$, $h(p)$ is the minimal polynomial of $b$ and thus, $L(a) \cong L'(b)$. In this way, for any algebraically independent $k$-tuple $\vec{\beta} = (\beta_1,\dots,\beta_k)$ sufficiently close to $(\alpha_1,\dots,\alpha_k)$, we have an isomorphism of fields $\varphi \colon \QQ(\alpha_1,\dots,\alpha_k, a) \to \QQ(\beta_1,\dots,\beta_k,f(\vec{\beta}))$ where $f(\vec{\beta})$ is algebraic over $\QQ(\beta_1,\dots,\beta_k)$, which naturally induces an isomorphism of groups $\psi \colon \Gamma \to \Gamma' \subset \PSL_2(\QQ(\beta_1,\dots,\beta_k,f(\vec{\beta})))$.

    In order to assure that $\Gamma'$ is a cocompact Fuchsian group we use corollary \ref{coro:weil}. Let $c \in K$. As $\varphi(c)$ is a linear combination over $\QQ(\beta_1,\dots,\beta_k)$ of powers of $b$, it follows that the value $\varphi(c)$ changes continuously with respect to $(\beta_1,\dots,\beta_k)$, which means that by taking $(\beta_1,\dots,\beta_k)$ sufficiently close to $(\alpha_1,\dots,\alpha_k)$ we can make sure that $c$ and $\varphi(c)$ are arbitrarily close. Now, by doing this with every coefficient of each $\gamma_i$, we can make $\gamma_i$ and $\psi(\gamma_i)$ arbitrarily close for all $i=1,\dots,n$. Therefore, by corollary \ref{coro:weil}, $\Gamma'$ is a cocompact Fuchsian group isomorphic to $\Gamma$ for $(\beta_1,\dots,\beta_k)$ sufficiently close to $(\alpha_1,\dots,\alpha_k)$.

    Finally, we just need to make $\beta_1,\dots,\beta_k$ computable numbers. Let $p_i$ be the $i$-th prime number. By the Lindemann-Weierstrass theorem (see~\cite[Theorem 1.4]{Bakerbook_transcendental_1990}), $(e^{\sqrt{p_1}}, \dots, e^{\sqrt{p_k}})$ are algebraically independent. 
    Then, for each $i = 1,\dots,k$, take $\beta_i = q_i e^{\sqrt{p_i}}$ with $q_i\in \QQ\setminus \{0\}$ a non-zero rational number such that $q_ie^{\sqrt{p_i}}$ is sufficiently close to $\alpha_i$. These numbers are all computable. Furthermore, as $f(\vec{\beta})$ is a real root of a rational equation with computable coefficients, it is computable as well. We conclude that with this choice $K'$ is a subfield of the field of computable real numbers and thus that the coefficients of each element of $\Gamma'$ are computable numbers.
\end{proof}

\subsection{Computable quasi-isometries}

The next goal is to use~\Cref{mainthm:computable_rep} to show that every cocompact Fuchsian group admits a computable quasi-isometry into the hyperbolic plane. We briefly recall the notion of quasi-isometry and the Schwarz-Milnor lemma.

\begin{definition}
    Let $(X,d_X)$ and $(Y,d_Y)$ be metric spaces. A map $f\colon X\to Y$ is called a quasi-isometry if there exist positive real constants $A$, $B$ and $C$ such that
 \begin{itemize}
\item $f$ is a \define{quasi-isometric embedding}: for every $x,y\in X$ one has
\[\frac{1}{A}d_X(x,y)-B\leq d_Y(f(x),f(y))\leq Ad_X(x,y)+B.\]
\item $f(X)$ is \define{relatively dense} in $Y$: for every $z\in Y$ there exists $x\in X$ such that 
\[d_{Y}(z,f(x))\leq C.\]
\end{itemize}
We say that $(X,d_X)$ and $(Y,d_Y)$ are \define{quasi-isometric} if such a map exists.
\end{definition}

Quasi-isometries induce an equivalence relation on the class of metric spaces that captures their large-scale geometry. In the case of a finitely generated group $G$, every finite symmetric set of generators $S$ induces a metric space through the Cayley graph $\operatorname{Cay}(G,S)$, that is, the undirected graph whose vertices are $G$ and the edges are of the form $(g,gs)$ for $g\in G$ and $s\in S$. It is well-known that all Cayley graphs of a given finitely generated group are quasi-isometric through the identity map, hence we may speak plainly about quasi-isometries from a group without making explicit reference to a set of generators.


\begin{lemma}[Schwarz-Milnor, see Proposition I.8.19 of~\cite{BridsonHaefiger_book}]
 Let $G$ be a group acting on a proper length metric space $(X,d)$ by isometries such that the action is properly discontinuous and cocompact. Then $G$ is finitely generated and for every $x \in X$ the orbit map $f\colon G \to X$ given by $f(g)= g\cdot x$ is a quasi-isometry. 
\end{lemma}

The canonical action of a cocompact Fuchsian group on the hyperbolic plane $\HH$ satisfies the hypotheses of the Schwarz-Milnor lemma; thus, as we mentioned before, these groups are finitely generated and quasi-isometric to $\HH$. Next we will show that this quasi-isometry can be made computable using~\Cref{mainthm:computable_rep}.

In order to speak of computable quasi-isometries we first need to fix computable structures. Whenever $G$ is an infinite and finitely generated group with decidable word problem, there always exists a canonical computable structure on $G$, that is, a bijection $\nu \colon \NN \to G$ which makes the group operations $(n,m) \mapsto \nu^{-1}(\nu(n)\nu(m))$ and $n \mapsto \nu^{-1}(\nu(n)^{-1})$ computable. Furthermore, it is unique up to computable isomorphism (see for instance the preliminaries of~\cite{BarCarRoj_2025}). This means that we can treat elements of $G$ as integers, and thus we may speak without ambiguity of computability on $G$. Since cocompact Fuchsian groups are word-hyperbolic, they are finitely generated and they have decidable word problem (see~\cite{BridsonHaefiger_book}), and hence they admit a computable structure.

For the hyperbolic plane $\HH$, we regard it as a computable metric space, where the metric is the natural geodesic distance $d_{\HH}$ and the countable dense subset is $S := \{a+bi \in \HH : a,b \in \QQ \}$, the set of rational points.

\begin{definition}\label{def:computable_QI}
    Let $G$ be a finitely generated group with decidable word problem. A map $f\colon G \to \HH$ is computable if there exists a computable map $h\colon \NN \times G \to S$ such that for each $n \in \NN$ and $g \in G$ we have \[d_{\HH}\bigl(h(n,g), f(g)\bigr) \leq 2^{-n}.\]

    If additionally $f$ is a quasi-isometry, we say that $f$ is a computable quasi-isometry.
\end{definition}

\begin{corollary}\label{cor:FuchsianQItoH}
    Every cocompact Fuchsian group admits a computable quasi-isometry to $\HH$.
\end{corollary}

\begin{proof}
    By~\Cref{mainthm:computable_rep}, without loss of generality we may write $\Gamma = \langle \gamma_1,\dots,\gamma_n\rangle$ where each $\gamma_i\in \PSL_2(\RR)$ has computable real numbers as coefficients. The action of $\Gamma$ on $\HH$ satisfies the hypothesis of the Schwarz-Milnor lemma and thus the orbit map based at any $z \in \HH$ induces a quasi-isometry $f\colon \Gamma \to \HH$. 
    
    Since each $\gamma_i$ is computable, and Möbius transformations with computable coefficients are themselves computable maps, it follows that a computable quasi-isometry is given by the orbit map of any computable point of $\HH$.
\end{proof}

\section{Medvedev degrees and quasi-isometries}\label{sec:QI}

We shall give a brief introduction to Medvedev degrees. An interested reader can find a more extensive and detailed introduction in~\cite{barbieri2024medvedev}.

Given a set $X\subset \NN^{\NN}$, a map $f\colon X \to \NN^{\NN}$ is called \define{computable} if there exists an oracle Turing machine which on oracle $x \in X$ and input $n\in \NN$ computes $f(x)(n)$. Intuitively, this means that there is an algorithm which, given access to the values $x(k)$ for any $k \in \NN$, can compute any coordinate of the image $f(x)$.

Given two subsets $X,Y$ of $\NN^{\NN}$, we say that $X$ is Medvedev reducible to $Y$, and write $X \preceq_{\mathfrak{m}} Y$, if there is a computable map $f\colon Y \to \NN^{\NN}$ that satisfies $f(Y)\subset X$. If both $X$ and $Y$ are Medvedev reducible to each other, we say they are Medvedev equivalent.  

\begin{definition}
    A \define{Medvedev degree} is an equivalence class of Medvedev equivalent subsets of $\NN^{\NN}$. For $X\subset \NN^{\NN}$ we denote its degree by $m(X)$.
\end{definition}

The collection $\mathfrak{M}$ of Medvedev degrees is a distributive lattice with the order $\preceq_{\mathfrak{m}}$. The minimum of this lattice is denoted by $0_{\mathfrak{M}}$, and consists of all sets that contain at least one computable element.

A set $X\subset \NN^{\NN}$ is \define{effectively closed} or $\Pi_1^0$ if there exists a recursively enumerable $L\subset \NN^*$ such that $X = \NN^{\NN}\setminus \bigcup_{w \in L}[w]$, where $[w]$ denotes the cylinder set of all infinite sequences that begin with the word $w$. Equivalently, a set $X$ is $\Pi_1^0$ if and only if there is an oracle Turing machine which halts on empty input with oracle $x$ precisely when $x \in \NN^{\NN}\setminus X$. The Medvedev degree of a $\Pi_1^0$ set $X\subset \{0,1\}^{\NN}$ is called a \define{$\Pi_1^0$-degree}.

In what follows, we shall explain how to assign Medvedev degrees to shift spaces on groups and on more general structures.

\subsection{Shift spaces on groups}

Let $A$ be a finite set. Without loss of generality, we may assume $A=\{0,1,\dots, |A|-1\}\subset \NN$. For a group $G$, the full $G$-shift is the space $A^G$ endowed with the right shift action $ A^G \curvearrowleft G$ given by \[(x\cdot g)(h)=x(gh) \mbox{ for all } x\in A^G \mbox{ and } g,h\in G.\]

We refer to elements $a\in A$ as \define{symbols} and to maps $x\in A^G$ as configurations. A \define{shift space} or \define{subshift} is a subset of configurations $X\subset A^G$ which is closed in the prodiscrete topology on $A^G$ and which is invariant under the right shift $G$-action. 

\begin{remark}
    It is more common in the literature to define subshifts using the left shift action given by $(g\cdot x)(h)=x(g^{-1}h)$ for $g,h\in G$ and $x\in A^G$. In this work we rather use the right action in order to have a more natural correspondence with a partial monoid action we shall define in the next subsection. We note that this choice is inconsequential, as there is a natural correspondence between subshifts defined through either of these two actions.
\end{remark}

Given a finitely generated group $G$ with decidable word problem, we can take a bijection $\nu \colon \NN \to G$ as explained above and thus identify the spaces $\NN^G$ with $\NN^{\NN}$ through the computable homeomorphism $\delta \colon \NN^{\NN} \to \NN^G$ given by \[ \delta(x)(g) = x(\nu^{-1}(g)) \mbox{ for all } g \in G.\]
Thus naturally we define the Medvedev degree of $X\subset \NN^G$ as $m(X) = m(\delta^{-1}(X))$ and similarly, we may speak about effectively closed subshifts. We remark that both of these notions do not depend upon the specific choice of $\nu$, see~\cite{barbieri2024medvedev}, and thus we can freely speak about effectively closed subshifts and their Medvedev degrees without specifying $\nu$.

\begin{remark}
    In fact, it is possible to meaningfully define Medvedev degrees of subshifts in arbitrary finitely generated groups by passing to their pullbacks on finite-rank free groups, see~\cite{barbieri2024medvedev}. As we shall only work with groups that have decidable word problem, we will not go into such generality.
\end{remark}

In what follows, we will be interested in a particular class of subshifts which encode sets of configurations that avoid finitely many forbidden patterns. Given a finite subset $F\subset G$, a \define{pattern} is a map $p\colon F \to A$. The \define{cylinder set} associated to a pattern $p$ is given by the set of configurations whose restriction to $F$ coincides with $p$, that is \[ [p] = \{x\in A^G : x|_F = p \}. \]

\begin{definition}
    A \define{$G$-subshift of finite type} ($G$-SFT) is a subset $X\subset A^G$ for which there exists a finite set $\mathcal{F}$ of patterns such that $x \in X$ if and only if for every $g\in G$ and $p \in \mathcal{F}$ we have that $x\cdot g \notin [p]$.
\end{definition}

In simpler words, a $G$-SFT is the set of maps $x\colon G \to A$ for which a finite list of forbidden patterns does not occur up to translation. It can be shown that if $G$ is a finitely generated group with decidable word problem (more generally, a recursively presented group), then every $G$-SFT is effectively closed. In particular, the Medvedev degree of every $G$-SFT is a $\Pi_1^0$-degree, see~\cite[Observation 3.5]{barbieri2024medvedev}. We denote the space of all Medvedev degrees of $G$-SFTs by $\mathfrak{M}_{\texttt{SFT}}(G)$.

\subsection{Blueprints}

Next we will introduce the notion of blueprints given in~\cite{BarBit2026}. Here we will provide the definitions succinctly and refer the reader to the above reference for a more gentle introduction.

\begin{definition}
    A \define{blueprint} is a 5-tuple $\Gamma = (M,S,i,t,R)$ where $M$ is a set of states, $S$ is a set of generators, $R\subset S^*\times S^*$ is a set of relations, and $i\colon S \to M$ and $t\colon S \to \mathcal{P}(M)\setminus \{ \varnothing\}$ are the initial and terminal maps.
\end{definition}

If $M,S$ and $R$ are finite, we say that the blueprint is \define{finitely presented}. We shall use blueprints to describe collections of countable rooted graphs, where the vertices are labeled with a state in $M$, and the directed edges are labeled with a generator $s\in S$ and must begin in the initial state $i(s)$ and end in a terminal state in $t(s)$. Finally, the set of relations $R$ provides a uniform description of the cycles for the whole family of graphs.

Denote the empty word in $S^*$ by $\varepsilon$. A words $u=u_1\dots u_n \in S^*$ is called $\Gamma$-consistent if either $w=\varepsilon$ or we have $i(u_{i+1})\in t(u_i)$ for all $i$. Two $\Gamma$-consistent words $u,v \in S^*$ are called $\Gamma$-similar if there exist $w,z \in S^*$ and $(x,y)\in R$ such that $u=wxz$ and $v = wyz$. We call $\Gamma$-equivalence the equivalence relation in $S^*$ generated by $\Gamma$-similarity.

For a map $f \colon S^*\to M \cup \{\varnothing\}$, its support is $\supp(f) = S^*\setminus f^{-1}(\varnothing)$. Given a blueprint $\Gamma$, we say that a map $\varphi \colon S^* \to M \cup \{\varnothing\}$ is $\Gamma$-consistent if $\varepsilon \in \supp(\varphi)$ and for each $w\in S^*$ and $s \in S$ we have that $\varphi(ws) \in t(s)$ whenever $i(s) = \varphi(w)$, and $\varphi(ws)=\varnothing$ whenever $i(s) \neq \varphi(w)$.

A $\Gamma$-consistent map $\varphi$ is called a $\Gamma$-\define{model} if for every $u,v \in \supp(\varphi)$ which are $\Gamma$-equivalent we have $\varphi(u)=\varphi(v)$. We denote the space of $\Gamma$-models of a blueprint $\Gamma$ by $\mathcal{M}(\Gamma)$. For a $\Gamma$-model $\varphi\in \mathcal{M}(\Gamma)$ and $w\in \supp(\varphi)$, denote by \[\underline{w} = \{u \in \supp(\varphi) : u \mbox{ is $\Gamma$-equivalent to }w\}.\]

The rooted graph $\mathcal{G}(\Gamma,\varphi)$ associated to a $\Gamma$-model $\varphi$ is the one given by the vertex set $V_{\mathcal{G}} = \{ \underline{w} : w \in \supp(\varphi)\} $ with root $\underline{\varepsilon}$ and edges \[ E_{\mathcal{G}} = \{ (\underline{w},\underline{ws}) : w \in \supp(\varphi), s \in S, i(s)=\varphi(w)\}.  \]

\begin{example}\label{ex:group_blueprint}
    Let $G=\langle S \mid R\rangle$ be a finitely presented group (assume that $S$ is symmetric and $R$ contains the trivial relations $ss^{-1}$ for $s\in S$). If we take $M = \{0\}$, $R'= \{(r,\varepsilon) : r \in R\}$ and for each $s\in S$ we set $i(s)=0$ and $t(s)=\{0\}$, then $\Gamma=(M,S,i,t,R')$ is a finitely presented blueprint, the only $\Gamma$-model is the constant map $\varphi \colon S^* \to M$ and its associated graph $\mathcal{G}(\Gamma,\varphi)$ is the right Cayley graph of $G$ with respect to $S$.
\end{example}

For our purposes, besides the fact that we can represent finitely presented groups through finitely presented blueprints, we will only use one specific family of blueprints that encodes shifted versions of the graphs $H_k$ described in the introduction.

\begin{definition}
     For $k \geq 2$, the \define{$k$-ary hyperbolic blueprint} $\mathcal{H}_k = (M_k,S_k,i,t,R_k)$ is given by
     \begin{enumerate}
         \item $M_k = \ZZ/k\ZZ$.
         \item $S_k = \{ a_0,\dots,a_{k-1}, \overline{a}_0,\dots, \overline{a}_{k-1}, b_0,\dots,b_{k-1}, \overline{b}\}$.
         \item For every $j \in \ZZ/k\ZZ$, we set $i(a_j)=i(\overline{a}_j)=i(b_j)=j$. We also set $i(\overline{b})=0$.
         \item For every $j \in \ZZ/k\ZZ$, we set $t(a_j)= \{j+1\}, t(\overline{a}_j)= \{j-1\},    t(b_j)=\{0\}, t(\overline{b})=\{0,\dots,k-1\}$.
         \item $R_k = \bigcup_{j \in \ZZ/k\ZZ}\{ (a_jb_{j+1}, b_ja_0\dots a_{k-1}), (a_j\overline{a}_{j+1}, \varepsilon), (\overline{a}_ja_{j-1}, \varepsilon),(b_j\overline{b},\varepsilon),(\overline{b}b_j,\varepsilon) \}$.
     \end{enumerate}
\end{definition}

In the definition above, the generator $a_j$ represents horizontal movement to the right from state $j$, while $\overline{a}_j$ represents movement to the left. The generator $b_j$ represents vertical movement downwards from state $j$, while $\overline{b}$ represents movement upwards and can only start from state $0$. Notice that the only source of non-determinism lies in the choice of terminal state for edges labeled by $\overline{b}$. The relations $R_k$ describe the basic cycles $(a_jb_{j+1}, b_ja_0\dots a_{k-1})$ for every $j\in \ZZ/k\ZZ$, along with the trivial relations that describe inverses. See~\Cref{fig:coloredH3}.

\begin{figure}[ht!]
    \centering
    \begin{tikzpicture}[scale = 5]
   \draw[thick, Stealth-] (0,1.04) -- (0,1+1/3);
   \draw[black!50, -Stealth] (0.04,1.04) -- (0.04,1+1/3);
   \node at (-0.05, 1+1/6) {$b_0$};
   \node[black!50] at (0.09, 1+1/6) {$\overline{b}$};
   \draw[thick, Stealth-] (-0.04,1) -- (-0.5,1);
   \draw[black!50, -Stealth] (-0.04,0.96) -- (-0.5,0.96);
   \node at (-0.25,1.05) {$a_2$};
   \node[black!50] at (-0.25,0.91) {$\overline{a}_0$};
   
   \draw[thick, arrows = {Stealth[scale=0.8]-}] (-0.03,4/9) -- (-0.25,4/9);
   \draw[black!50, arrows = {-Stealth[scale=0.8]}] (-0.03,4/9-0.03) -- (-0.25,4/9-0.03);
   \node[scale = 0.8] at (-0.125,4/9+0.04) {$a_2$};
   \node[scale = 0.8, black!50] at (-0.125,4/9-0.07) {$\overline{a}_0$};
   \draw[thick, -Stealth] (1,1) -- (1.5,1);
   \draw[black!50, Stealth-] (1.05,0.95) -- (1.5,0.95);
   \node at (1.25,1.05) {$a_1$};
   \node[black!50] at (1.25,0.91) {$\overline{a}_2$};

   \draw[thick, arrows = {-Stealth[scale=0.8]}] (1,4/9) -- (1.25,4/9);
   \draw[black!50, arrows = {Stealth[scale=0.8]-}] (1.04,4/9-0.03) -- (1.25,4/9-0.03);
   \node[scale = 0.8] at (1.125,4/9+0.04) {$a_0$};
   \node[scale = 0.8, black!50] at (1.125,4/9-0.07) {$\overline{a}_1$};
   
    \foreach \i in {0}{
        \draw[thick, -Stealth] (\i,1) -- (\i+0.96,1);
        \node at (\i+0.5,1.05) {$a_0$};
        \draw[black!50, -Stealth] (\i+1,0.96) -- (\i+0.05,0.96);
        \node[black!50] at (\i+0.5,0.91) {$\overline{a}_1$};
    }
    \foreach \i in {0,...,2}{
        \draw[thick, arrows = {-Stealth[scale=0.8]}] (\i/3,4/9) -- (\i/3+1/3-0.03,4/9);
        \draw[black!50, arrows = {Stealth[scale=0.8]-}] (\i/3+0.04,4/9-0.03) -- (\i/3+1/3-0.03,4/9-0.03);
        \node[scale = 0.8] at (\i/3+1/6,4/9+0.04) {$a_{\i}$};
    }
     \node[black!50, scale = 0.8] at (1/6,4/9-0.07) {$\overline{a}_{1}$};
     \node[black!50, scale = 0.8] at (3/6,4/9-0.07) {$\overline{a}_{2}$};
     \node[black!50, scale = 0.8] at (5/6,4/9-0.07) {$\overline{a}_{0}$};
    \foreach \i in {0,1}{
        \draw[thick, arrows = {-Stealth[scale=0.8]}] (\i,1) -- (\i,4/9+0.05);
        \draw[black!50, arrows = {Stealth[scale=0.8]-}] (\i+0.03,1-0.05) -- (\i+0.03,4/9+0.05);
        \node at (\i-0.05,6.5/9) {$b_{\i}$};
        \node[black!50] at (\i+0.08,6.5/9) {$\overline{b}$};
        \draw[thick, fill=black!10] (\i,1) circle (0.05);
        \node at (\i,1) {$\i$};
    }
    \foreach \i in {0,...,3}{
        \draw[thick, arrows = {-Stealth[scale=0.6]}] (\i/3,4/9) -- (\i/3,1/9+0.1);
        \draw[black!50, arrows = {Stealth[scale=0.6]-}] (\i/3+0.03,4/9-0.04) -- (\i/3+0.03,1/9+0.1);
        \draw[thick, fill=black!10] (\i/3,4/9) circle (0.04);
        \node at (\i/3,4/9) {$\i$};
    }
    \node[scale = 0.8] at (-0.04,3/9) {$b_0$};
    \node[scale = 0.8] at (1/3-0.04,3/9) {$b_1$};
    \node[scale = 0.8] at (2/3-0.04,3/9) {$b_2$};
    \node[scale = 0.8] at (1-0.04,3/9) {$b_0$};
    \node[black!50, scale = 0.8] at (0.07,3/9) {$\overline{b}$};
    \node[black!50, scale = 0.8] at (1/3+0.07,3/9) {$\overline{b}$};
    \node[black!50, scale = 0.8] at (2/3+0.07,3/9) {$\overline{b}$};
    \node[black!50, scale = 0.8] at (1+0.07,3/9) {$\overline{b}$};
    
    \draw[thick, fill=black!10] (1,4/9) circle (0.04);
        \node at (1,4/9) {$0$};
    \end{tikzpicture}
    \caption{A portion of a graph associated to the blueprint $\mathcal{H}_3$.}
    \label{fig:coloredH3}
\end{figure}
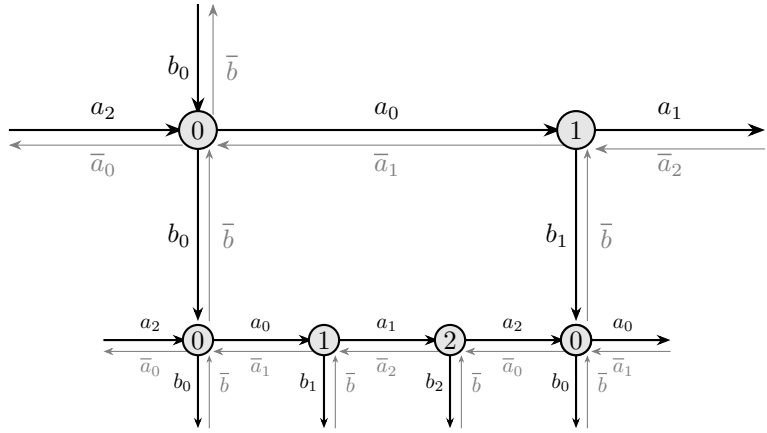

\begin{remark}\label{rem:H-1}
    The definition of $\mathcal{H}_k$ still makes sense for $k =1$. In this case, there is a single state and thus a single model. The unique graph associated to this model coincides with the canonical Cayley graph of $\ZZ^2$.
\end{remark}

\begin{example}\label{ex:maingraph}Consider the model $\varphi_{\circ}\in \mathcal{M}(\mathcal{H}_k)$ with $\varphi_{\circ}(\varepsilon)=0$ and which satisfies $\varphi_{\circ}(\overline{b}^n)=0$ for every $n \geq 1$. It can be shown using the relations that $\varphi_{\circ}$ is completely determined by this choice and that $\mathcal{G}(\mathcal{H}_k,\varphi_{\circ})$ is isomorphic to the graph $H_k$ rooted at $(0,0)$ after replacing its undirected edges by directed ones in both directions. It is easy to see that $\varphi_{\circ}$ is computable. 
\end{example} 

  Next we define subshifts of finite type in a blueprint $\Gamma = (M,S,i,t,R)$. As before, we provide the definition briefly and refer the interested reader to~\cite{BarBit2026} for further examples. Let $A$ be a finite set. 

  \begin{definition}
      Let $\Gamma$ be a blueprint. A \define{$\Gamma$-subshift} on alphabet $A$ is a set $X=X[\Gamma,\mathcal{F}]$, where $\mathcal{F}\subset \bigcup_{W\Subset S^*}((M\times A)\cup \{\varnothing\})^{W}$ is a set of forbidden patterns, and $X$ consists of all pairs of the form $(\varphi,x) \in (M\cup \{\varnothing\})^{S^*}\times (A\cup \{\varnothing\})^{S^*}$ such that
      \begin{enumerate}
          \item $\varphi$ is a $\Gamma$-model and $\supp(\varphi)=\supp(x)$.
          \item For each pair of $\Gamma$-equivalent $u,v \in \supp(\varphi)$ we have $x(u)=x(v)$.
          \item For every $u \in \supp(\varphi)$, $W\Subset S^*$ and $p \in \mathcal{F}\cap ((M\times A)\cup \{\varnothing\})^{W}$ there exists $w \in W$ such that $(\varphi(uw),x(uw))\neq p(w)$.
      \end{enumerate}
  \end{definition}

  In the above definition, the first two conditions essentially say that $x$ is consistent with the model $\varphi$, so that in the graph induced by $\varphi$, the map $x$ encodes a coloring of the vertices. The last condition just says that these colorings avoid a collection of forbidden patterns. We note that each $\Gamma$-subshift can be endowed with the partial right monoid action $X \curvearrowleft S^*$ given by \[ ((\varphi,x) \cdot w) (u) = (\varphi(wu),x(wu)) \mbox{ for every } w \in \supp(\varphi) \mbox{ and  } u \in S^*.   \]

  We say that a $\Gamma$-subshift $X$ is of \define{finite type} ($\Gamma$-SFT) if there exists a finite $\mathcal{F}\subset \bigcup_{W\Subset S^*}((M\times A)\cup \{\varnothing\})^{W}$ such that $X=X[\Gamma,\mathcal{F}]$. In other words, $X$ is of finite type if it can be defined by a finite collection of forbidden patterns.

  Identifying $S^*$ with $\NN$ through a shortlex order and $(M \times A) \cup \{\varnothing\}$ as a finite subset of $\NN$, we may interpret a $\Gamma$-subshift of finite type as a subset of $\NN^{\NN}$ and thus assign it a Medvedev degree. We denote by $\mathfrak{M}_{\texttt{SFT}}(\Gamma)$ the space of Medvedev degrees of a blueprint $\Gamma$.

  \begin{example}
      (Hard square shift) Take the blueprint $\mathcal{H}_k$ and let $A=\{0,1\}$. Consider $W = \{\varepsilon\} \cup S_k$ and take $\mathcal{F}$ as the set of all maps $p \colon W\to (M_k \times A) \cup \{\varnothing\}$ for which there exists $s \in S_k$ such that $p(\varepsilon)=(m,1)$ and $p(s)=(m',1)$ for some $m,m'\in M_k$. The corresponding $\mathcal{H}_k$-subshift of finite type $X[\mathcal{H}_k,\mathcal{F}]$ corresponds to all possible colorings of the graphs $\mathcal{G}(\mathcal{H}_k,\varphi)$ for $\varphi \in \mathcal{M}(\mathcal{H}_k)$ with the property that adjacent pairs of vertices cannot both have the symbol $1$. This $\mathcal{H}_k$-subshift has zero Medvedev degree, as the element $(\varphi_\circ,x_\circ)\in X[\mathcal{H}_k,\mathcal{F}]$, where $x_\circ(w)=0$ for every $w\in \supp(\varphi_\circ)$, is computable.
  \end{example}

  \begin{example}
      Take a finitely presented group $G= \langle S|R\rangle$ and consider its associated blueprint $\Gamma$ as in~\Cref{ex:group_blueprint}. Consider a $\Gamma$-SFT $X[\Gamma,\mathcal{F}]$ with alphabet $A$ and $(\varphi,x)\in X[\Gamma,\mathcal{F}]$. As there exists a unique $\Gamma$-model, we may identify the pairs $(\varphi,x)$ with their second coordinate $x$. Furthermore, if we take the canonical epimorphism $\psi \colon S^* \to G$ that identifies a word in $S$ with the corresponding element of $G$, then for every $u,v \in S^*$ with $\psi(u)=\psi(v)$, we have $x(u)=x(v)$. It follows that $x$ induces a map $\widehat{x}\colon G \to A$ given by $\widehat{x}(\psi(u))=x(u)$. Similarly, we can identify the forbidden patterns that give rise to non-trivial restrictions with patterns supported in $G$ and show that $\widehat{X} = \{ \widehat{x}\in A^G : (\varphi,x)\in X[\Gamma,\mathcal{F}]\}$ is a $G$-SFT. It follows that $\Gamma$-SFTs are in direct correspondence with $G$-SFTs and, furthermore, have the same Medvedev degrees.
  \end{example}

\subsection{Properties of the hyperbolic blueprints}\label{subsection:graphs-H_k}

Here we will delve a little bit further into properties of the blueprints $\mathcal{H}_k$ with the aim of proving~\Cref{mainthm:to_blueprint}. In what follows, we fix $k \geq 2$.

We say that a blueprint $\Gamma$ is \define{strongly connected} if every graph associated to a $\Gamma$-model is strongly connected, that is, if there is a directed path between every pair of vertices. It is clear from the definition that every generator in $\mathcal{H}_k$ admits both a left and a right inverse in the blueprint relations $R_k$ and thus every graph associated to an $\mathcal{H}_k$-model is strongly connected.

We say that $\Gamma$ has \define{decidable word problem} if there is an algorithm that decides whether two $\mathcal{H}_k$-consistent words on the generators are $\Gamma$-equivalent. In order to show that $\mathcal{H}_k$ has decidable word problem, we will introduce normal forms. For $i \in \ZZ/k\ZZ$ and an integer $t \geq 1$ we define the words \[ r(i,t) = a_ia_{i+1}\dots a_{i+t-1} \quad \mbox{and}\quad \ell(i,t) = \overline{a}_i\overline{a}_{i-1}\dots \overline{a}_{i-t+1}.  \]

We note that every $\mathcal{H}_k$-consistent word in $\{a_0,\dots, a_{k-1},\overline{a}_0,\dots, \overline{a}_{k-1}\}^*$ is $\mathcal{H}_k$-equivalent to either the empty word $\varepsilon$, $r(i,t)$ or $\ell(i,t)$ for some $i \in \ZZ/k\ZZ$ and $t \geq 1$.

\begin{definition}
    We say that an $\mathcal{H}_k$-consistent word $w \in S_k^*$ is in \define{normal form} if it is of one of the forms $\overline{b}^n$, $\ell(i,t)\overline{b}^n$, $r(i,t)\overline{b}^n$, $b_ib_0^n$, $b_ib_0^n\ell(0,t)$ or $b_ib_0^nr(0,t)$ for some $n \geq 0$, $i \in \ZZ/k\ZZ$, and $t \geq 1$.
\end{definition}

\begin{lemma}\label{lem:WP_decidable_Hk}
    There is an algorithm that takes an $\mathcal{H}_k$-consistent word $w\in S_k^*$ and produces an $\mathcal{H}_k$-equivalent normal form. Furthermore, two words are $\mathcal{H}_k$-equivalent if and only if their normal forms coincide.
\end{lemma}

\begin{proof}
    Set $S_k'\subset S_k$ as $S_k'= \{a_0,\dots, a_{k-1},\overline{a}_0,\dots, \overline{a}_{k-1}\}$. We present an algorithm which takes an $\mathcal{H}_k$-consistent word and puts it in normal form.

    \begin{itemize}
        \item First, check if the word has any occurrence of the form $\overline{b}ub_i$ or $b_i u \overline{b}$ with $u \in (S_k')^*$ and $i \in \ZZ/k\ZZ$. Using the cyclic relation we replace these words by an equivalent word in $S_k^*$ with one fewer occurrence of each of $\overline{b}$ and $b_i$. Repeat this process until either $\overline{b}$ or $b_i$ is no longer present in the word.
        \item If no occurrence of $b_i$ remains, we move every remaining occurrence of $\overline{b}$ to the rightmost side of the word iterating the following procedure: if a subword of the form $\overline{b}a_i$ occurs, we replace it by $a_0\dots a_{k-1}\overline{b}$. If $\overline{b}\overline{a}_{i}$ occurs, we replace it by $\overline{a}_0\overline{a}_{k-1}\dots \overline{a}_1\overline{b}$. 
        \item If any symbols $b_i$ remain, we move them to the leftmost side of the word through the following procedure: if a subword of the form $a_{i-1}b_i$ occurs, we replace it by $b_{i-1}a_0\dots a_{k-1}$. If $\overline{a}_{i+1}b_i$ occurs, we replace it by $b_{i+1}\overline{a}_0\overline{a}_{k-1}\dots \overline{a}_1$. 
        \item Finally, we cancel out all the occurrences of the trivial relations in the word until none remain (that is, any occurrence of $a_j\overline{a}_{j+1}$ or $\overline{a}_ja_{j-1}$ for some $j\in \ZZ/k\ZZ$). 
    \end{itemize}
    
    It is clear that this yields a normal form and that the procedure is computable.

    As in every step of the algorithm above we used relations in $\mathcal{H}_k$, it is clear that the original word is $\mathcal{H}_k$-equivalent to its normal form. Finally, it is easy to show that two distinct normal forms are not $\mathcal{H}_k$-equivalent (note that every relation preserves the number of occurrences of $\overline{b}$ minus the total number of occurrences of $b_i$, summed over all $i \in \ZZ/k\ZZ$). \end{proof}
    
    \begin{corollary}
        The blueprint $\mathcal{H}_k$ has decidable word problem.
    \end{corollary}

Next we will show that there is a computable map which takes an $\mathcal{H}_k$-model and produces a quasi-isometry to the complex upper half-plane model of the hyperbolic plane.

\begin{lemma}\label{lem:com_qi_explicit_horrendous}
    For $k \geq 2$ there is a computable map $f\colon S_k^* \to \HH$ and $D,M>0$ such that for every $\varphi\in \mathcal{M}(\mathcal{H}_k)$, the map $\xi_{\varphi }\colon \mathcal{G}(\mathcal{H}_k,\varphi)\to \HH$ given by $\xi_{\varphi}(\underline{v}) = f(v)$ for $v\in \supp(\varphi)$ is a well-defined quasi-isometry such that for every $z\in \HH$:
        \begin{itemize}
            \item There is $u \in \supp(\varphi)$ such that $d_{\HH}(z,\xi_\varphi(\underline{u}))\leq D$.
            \item We have $|\{ \underline{v} : d_{\HH}(z,\xi_{\varphi}(\underline{v})) \leq 1 \}|\leq M$.
        \end{itemize}
\end{lemma}

\begin{proof}
Consider first the height map $h\colon S_k^* \to \ZZ$ defined inductively on the length of the word by $h(\varepsilon)=0$ and for $w\in S_k^*$ and $s \in S_k$ by:

\[ h(ws) = \begin{cases}
        h(w)+1 & \mbox{ if } s = \overline{b}\\
         h(w)-1 & \mbox{ if } s \in \{b_0,\dots,b_{k-1}\}\\
         h(w) & \mbox{ if } s \in \{a_0,\dots,a_{k-1},\overline{a}_0,\dots,\overline{a}_{k-1}\}.
    \end{cases}  \]

    Note that $h$ just counts the number of symbols $\overline{b}$ that occur in the word and subtracts the number of $b_i$ that occur for all $i \in \ZZ/k\ZZ$. Note that every relation in $R_k$ preserves this count, and thus if two words $u,v$ are $\mathcal{H}_k$-equivalent, then $h(u)=h(v)$. 

    Next we define $f\colon S_k^* \to \HH$ inductively. Set $f(\varepsilon)=i$. For $w\in S_k^*$ and $s \in S_k$ we set, 
    \[ f(ws) = \begin{cases}
        f(w)+ik^{h(ws)}-ik^{h(w)} & \mbox{ if } s \in \{\overline{b},b_0,\dots,b_{k-1}\}\\
         f(w)+k^{h(w)} & \mbox{ if } s \in \{a_0,\dots,a_{k-1}\}\\
         f(w)-k^{h(w)} & \mbox{ if } s \in \{\overline{a}_0,\dots,\overline{a}_{k-1}\}.
    \end{cases}  \]

    It is evident that $f$ is computable. Also, a straightforward verification shows that $(u,v)\in R_{k}$ satisfies $f(wu)=f(wv)$ for all $w\in S_k^*$ and thus for every $\varphi \in \mathcal{M}(\mathcal{H}_k)$, the induced map $\xi_{\varphi}(\underline{v}) =f(v)$ is well-defined.

    To prove the rest, note that for $n \geq 0$, $j \in \ZZ/k\ZZ$, and $t \geq 1$, the values of $f$ on normal forms are given as follows:

    \begin{align*}
        f( \overline{b}^n) & = ik^n & f(b_j(b_0)^n) &= ik^{-(n+1)}\\
        f(\ell(j,t)\overline{b}^n) &= -t+ik^n &  f(r(j,t)\overline{b}^n) &= t+ik^n\\
        f(b_j(b_0)^n\ell(0,t)) &= -tk^{-(n+1)}+ik^{-(n+1)}& 
        f(b_j(b_0)^nr(0,t)) &= tk^{-(n+1)}+ik^{-(n+1)}
    \end{align*}
    
Let us estimate the least distance in $\HH$ that two distinct normal forms can have. Set $u,v \in \supp(\varphi)$. First suppose that the imaginary parts of $f(u)$ and $f(v)$ differ by some factor $k^n$ with $n \in \ZZ\setminus \{0\}$. As vertical lines are geodesics under $\HH$, we deduce that \[d_{\HH}(f(u),f(v)) =|\ln(\Im(f(u))/\Im(f(v)))| \geq |n|\ln(k) \geq \ln(k).\]

On the other hand, if the imaginary parts of $f(u)$ and $f(v)$ coincide, then their real parts differ by a non-zero multiple of $\Im(f(u))$. Thus we get \[d_{\HH}(f(u),f(v)) \geq 2 \operatorname{arcsinh}\left(\frac{|\Re(f(u))-\Re(f(v))|}{2\Im(f(u))}\right) = 2\operatorname{arcsinh}\left(\frac{1}{2}\right) = 2 \ln \left( \frac{1+\sqrt{5}}{2}\right). \]

We deduce that for any $u,v \in \supp(\varphi)$ such that $\underline{u} \neq \underline{v}$, we have that \[d_{\HH}(\xi_{\varphi}(\underline{u}),\xi_{\varphi}(\underline{v})) \geq C = \min\left(2 \ln \left( \frac{1+\sqrt{5}}{2}\right),\ln(k)\right).\]

Take $M$ as the maximal size of a $\frac{C}{2}$-separated set within balls of radius $2$ (this quantity exists because disks can be assigned an area in $\HH$ which depends solely on the radius), it is clear that with this choice we have that for any $z\in \HH$ then $|\{ \underline{v} : d_{\HH}(z,\xi_{\varphi}(\underline{v})) \leq 1 \}|\leq M$.

Using the previous computations, a routine argument shows that $\xi_{\varphi}$ is a quasi-isometric embedding. One can also check that $D =2\ln \left( \sqrt{k}\left(\frac{1+\sqrt{5}}{2}\right)\right)$ satisfies the density condition. \end{proof}


\begin{definition}\label{def:computable_embedding_QI}
    Let $\Gamma_1,\Gamma_2$ be two finitely presented blueprints with generating sets $S_1,S_2$ respectively. We shall say that $\Gamma_1$ has a \define{computable quasi-isometric image} in $\Gamma_2$ if there is a constant $N \geq 1$ and computable maps $\tau \colon \mathcal{M}(\Gamma_1)\to\mathcal{M}(\Gamma_2)$ and $h\colon \mathcal{M}(\Gamma_1)\times S_1^* \to S_2^*$ such that for every $\varphi \in \mathcal{M}(\Gamma_1)$ and $u,v \in \supp(\varphi)$ we have:
\begin{enumerate}
    \item if $u,v$ are $\Gamma_1$-equivalent, then $h(\varphi,u)$ and $h(\varphi,v)$ are $\Gamma_2$-equivalent words in $\supp(\tau(\varphi))$.
    \item If we let $h_{\varphi}\colon S_1^*\to S_2^*$ be given by $h_{\varphi}(s)=h(\varphi,s)$, then $h_{\varphi}$ induces a quasi-isometry from $\mathcal{G}(\Gamma_1,\varphi)$ to $\mathcal{G}(\Gamma_2,\tau(\varphi))$ in which every vertex has at most $N$ preimages.
\end{enumerate}

\end{definition}

In the definition above, the first condition is expressing that $h_{\varphi}$ induces a well-defined map from $\mathcal{G}(\Gamma_1,\varphi)$ to $\mathcal{G}(\Gamma_2,\tau(\varphi))$, while the second condition says that this map is a quasi-isometry with a universally bounded number of preimages. If each of the blueprints $\Gamma_1$ and $\Gamma_2$ has a computable quasi-isometric image in the other, we say that they are \define{computably quasi-isometric}.

The next technical lemma provides conditions under which a blueprint has a computably quasi-isometric image in another blueprint. 

\begin{lemma}\label{lem:computable_embedding_QI}
    Let $\Gamma_1,\Gamma_2$ be two finitely presented blueprints with decidable word problem and generating sets $S_1,S_2$ respectively. Suppose that we have the following two properties:
    \begin{enumerate}
        \item For $i \in \{1,2\}$ there exists a computable map $f_i\colon S_i^* \to \HH$ and constants $M,D >0$ such that for every $\varphi \in \mathcal{M}(\Gamma_i)$ the map $\xi_{\varphi}\colon \mathcal{G}(\Gamma_i,\varphi) \to \HH$ given by $\xi_{\varphi}(\underline{v})=f_i(v)$ for $v \in \supp(\varphi)$ is a well-defined quasi-isometry such that for every $z\in \HH$:
        \begin{itemize}
            \item There is $u \in \supp(\varphi)$ such that $d_{\HH}(z,\xi_\varphi(\underline{u}))\leq D$.
            \item We have $|\{ \underline{v} : d_{\HH}(z,\xi_{\varphi}(\underline{v})) \leq 1 \}|\leq M$.
        \end{itemize}
        \item There exists a computable map $\tau \colon \mathcal{M}(\Gamma_1)\to \mathcal{M}(\Gamma_2)$.
    \end{enumerate} Then $\Gamma_1$ has a computably quasi-isometric image in $\Gamma_2$.
\end{lemma}

\begin{proof}
    As $f_1$ is a computable map, for each integer $n \geq 1$ and $u \in S_1^*$ we can compute $q_n(u) \in \QQ(i)\cap \HH$ such that $d_{\HH}(f_1(u),q_n(u))\leq 2^{-n}$. Similarly, as $f_2$ is a computable map, for each integer $n \geq 1$ and $v \in S_2^*$ we can compute $r_n(v) \in \QQ(i)\cap \HH$ such that $d_{\HH}(f_2(v),r_n(v))\leq 2^{-n}$. Furthermore, as $\Gamma_1$ has decidable word problem, we may impose that the computation of $q_n(u)$ is uniform on $\Gamma_1$-equivalence classes.
    
    As the maps $\xi_{\varphi}$ are quasi-isometries, we may fix a positive integer $D >0$ such that for every $z\in \HH$ and $(\varphi_1,\varphi_2)\in \mathcal{M}(\Gamma_1)\times \mathcal{M}(\Gamma_2)$ there are $u\in \supp(\varphi_1)$ and $v \in  \supp(\varphi_2)$ with $\max(d_{\HH}(z,f_1(u)), d_{\HH}(z,f_2(v)))\leq D$. Fix also a positive integer $n$ such that $2^{-n}<D$.

    By definition of $D$, we know that for every $u \in S_1^*$ and $\varphi_1\in \mathcal{M}(\Gamma_1)$ there exists some $v\in \supp(\tau(\varphi_1))$ such that $d_{\HH}(f_1(u),f_2(v))\leq D$, and thus by our choice of $n$, we will also have that $d_{\HH}(q_n(u),r_n(v))<3D$.

    We construct a map $h\colon \mathcal{M}(\Gamma_1)\times S_1^* \to S_2^*$ as follows. For $u \in S_1^*$ and $\varphi_1\in \mathcal{M}(\Gamma_1)$ we set $h(\varphi_1,u)=v$ where $v\in S_2^*$ is the smallest word (in shortlex order) in $\supp(\tau(\varphi_1))$  which satisfies $d_{\HH}(q_n(u),r_n(v))\leq 3D$. As $\tau$ is computable and such a $v$ always exists, it follows that our map $h$ is also computable.

    Let us verify that $\tau,h$ satisfy the remaining requirements of~\Cref{def:computable_embedding_QI}. Fix $\varphi_1 \in \mathcal{M}(\Gamma_1)$, $u,u'\in\supp(\varphi_1)$, and set $\varphi_2 = \tau(\varphi_1)$. 
    
    If $u,u'$ are $\Gamma_1$-equivalent, then we must have $q_n(u)=q_n(u')$. In particular, it follows that $h(\varphi_1,u)=h(\varphi_1,u')$ and thus they must be $\Gamma_2$-equivalent words in $\supp(\varphi_2)$.

    To check that $h_{\varphi_1}$ induces a quasi-isometry from $\mathcal{G}(\Gamma_1,\varphi_1)$ to $\mathcal{G}(\Gamma_2,\varphi_2)$, it suffices to show that the maps $\xi_{\varphi_2}\circ h_{\varphi_1}$ and $\xi_{\varphi_1}$ are uniformly close. On the one hand we have $\xi_{\varphi_1}(\underline{u})=f_1(u)$, and on the other hand we have $\xi_{\varphi_2}(\underline{v}) = f_2(v)$ where $v = h(\varphi_1,u)$. By definition of $h$, we have that $d_{\HH}(q_n(u),r_n(v))\leq 3D$ and thus \[ d_{\HH}(f_1(u),f_2(v)) \leq d_{\HH}(f_1(u),q_n(u)) + d_{\HH}(q_n(u),r_n(v))+d_{\HH}(r_n(v),f_2(v))\leq 5D.  \]
    Hence $d_{\infty}(\xi_{\varphi_2}\circ h_{\varphi_1},\xi_{\varphi_1}) \leq 5D$. 

    Finally, we estimate the number of preimages of $\underline{v}$, where $v=h_{\varphi_1}(u)$. If $u'\in \supp(\varphi_1)$ is such that $h_{\varphi_1}(u')=v$, it follows that $d_{\HH}(f_1(u),f_1(u'))\leq 10D$. Hence $f_1(h_{\varphi_1}^{-1}(v))$ is contained in a disk of radius $10D$ centered in $f_1(u)$. Let $t\geq 0$ be such that for every disk of radius $10D$ there are $z_1,\dots, z_t\in \HH$ such that the collection of disks of radius $1$ centered at $z_i$ cover it. As we assumed that $|\{ \underline{v} : d_{\HH}(z_i,\xi_{\varphi}(\underline{v})) \leq 1 \}|\leq M$, we deduce that the number of preimages of $\underline{v}$ is bounded by $N=tM$.\end{proof}

\subsection{Transference of Medvedev degrees}

In this section we shall show that for any $k \geq 2$, the blueprint $\mathcal{H}_k$ admits the same Medvedev degrees of SFTs as any cocompact Fuchsian group (\Cref{mainthm:to_blueprint}). The main tool is the following result of~\cite{BarBit2026}.

\begin{theorem}[Corollary 5.12 of~\cite{BarBit2026}]\label{teo:BarBit2026}
    Let $\Gamma_1,\Gamma_2$ be two finitely presented, strongly connected blueprints with decidable word problem. If $\Gamma_1$ and $\Gamma_2$ are computably quasi-isometric, then $\mathfrak{M}_{\texttt{SFT}}(\Gamma_1)=\mathfrak{M}_{\texttt{SFT}}(\Gamma_2)$.
\end{theorem}

The proof of~\Cref{teo:BarBit2026} is very technical and thus here we will just provide a very broad intuition and refer the readers to~\cite{BarBit2026} for the details. Essentially, using the hypothesis that the blueprints are finitely presented, we can construct an SFT in $\Gamma_1$ where every configuration encodes a quasi-isometric image of a graph from $\Gamma_2$. On top of that, we may take any $\Gamma_2$-SFT and simulate it on top of its encoding in the $\Gamma_1$-SFT. The computability hypotheses in~\Cref{teo:BarBit2026} ensure that this construction does not accidentally increase the Medvedev degree through the inherent complexity of the encoding.

For the remainder of the section, fix $k \geq 2$, and let $\Gamma_1=\mathcal{H}_k$. By the results in the previous subsection, we have that $\Gamma_1$ is finitely presented and strongly connected, and has decidable word problem. 

Fix also a finite presentation $\langle S|R\rangle$ of a cocompact Fuchsian group $G$ and let $\Gamma_2$ be the blueprint induced by this presentation (as in~\Cref{ex:group_blueprint}). We have that $\Gamma_2$ is finitely presented, strongly connected and has decidable word problem. Note that it has a unique model given by the constant map $\varphi_c$, and that the associated graph $\mathcal{G}(\Gamma_2,\varphi_c)$ coincides with the Cayley graph of $G$ with respect to $S$. We also have that $\mathfrak{M}_{\texttt{SFT}}(\Gamma_2)=\mathfrak{M}_{\texttt{SFT}}(G)$.

Hence by~\Cref{teo:BarBit2026}, in order to prove~\Cref{mainthm:to_blueprint}, it suffices to show that $\Gamma_1$ and $\Gamma_2$ are computably quasi-isometric.

\begin{lemma}
    $\Gamma_1$ and $\Gamma_2$ are computably quasi-isometric.
\end{lemma}

\begin{proof}
    We verify the hypotheses of~\Cref{lem:computable_embedding_QI}. First we check condition (2) in both directions.

    Take $\tau_1 \colon \mathcal{M}(\Gamma_1)\to\mathcal{M}(\Gamma_2)$ as the constant map $\tau_1(\varphi)=\varphi_c$ for all $\varphi \in  \mathcal{M}(\Gamma_1)$. As $\varphi_c$ is a constant map, this is computable.

    Take $\tau_2 \colon \mathcal{M}(\Gamma_2)\to\mathcal{M}(\Gamma_1)$ as the constant map $\tau_2(\varphi_c) = \varphi_\circ$, where $\varphi_{\circ}$ is the model defined in~\Cref{ex:maingraph}. This is also clearly computable.

    Next we check condition (1). The existence of a map $f_1\colon S_k^*\to \HH$ that satisfies the requirements with constants $M_1,D_1$ is proven in~\Cref{lem:com_qi_explicit_horrendous}. For the second map, it follows from~\Cref{mainthm:computable_rep} (and more precisely, from~\Cref{cor:FuchsianQItoH}) that there exists a computable quasi-isometry $h\colon G \to \HH$. As $G$ has decidable word problem, by composing $h$ with the word evaluation map we obtain $f_2\colon S^*\to \HH$ such that equivalent words have the same image. Hence, the map $\xi_{\varphi_c}$ given by $\xi_{\varphi_c}(\underline{v})=f_2(v)$ is a well-defined quasi-isometry. In this case, we may take any constants $D_2,M_2$ that satisfy the requirements only for the quasi-isometry $\xi_{\varphi_c}$. Taking $D=\max(D_1,D_2)$ and $M = \max(M_1,M_2)$ we get that $f_1,f_2$ and satisfy (1).

    Hence from~\Cref{lem:computable_embedding_QI} we conclude that $\Gamma_1$ and $\Gamma_2$ are computably quasi-isometric.\end{proof}


\section{The sofic \texorpdfstring{$\mathcal{H}_3$}{H3}-subshift \texorpdfstring{$\mathfrak{X}_1$}{X1}}\label{sec:sofic}

In this section we describe $\mathfrak X_1$, an $\mathcal H_3$-subshift that has a hierarchical structure in the same spirit as the classical Robinson tiling. Furthermore, $\mathfrak X_1$ has the property of being sofic, that is, it can be obtained from an $\mathcal H_3$-SFT through a local map, and furthermore, it will have the property that such an $\mathcal H_3$-SFT can be chosen to have zero Medvedev degree. This is the basic framework for~\Cref{mainthm:grafo_medved}, which will be proven in the next section.


In what follows it will be convenient to ease the notation on the blueprint formalism a bit. In particular, for the downward generators $b_0,b_1$ and $b_2$ we will use generically the symbol $b$ and the horizontal generators $a_0,a_1$ and $a_2$ will be represented by the symbol $a$. Also, we will implicitly identify elements in the support of $\varphi$ with their vertices in the graph $\mathcal{G}(\mathcal{H}_3,\varphi)$.

Most of the objects described in this section are generalized versions of those in the classical Robinson tiling~\cite{robinson_undecidability_1971}. While familiarity with the properties of that tiling is not needed, it is encouraged, as it helps to understand why we are defining these structures. 

\subsection{Supports and cells}\label{subsec:supports}
We will first describe a hierarchy of finite subgraphs which occur in any graph associated to $\mathcal{H}_3$. These play the same role as $n\times n$ squares in $\ZZ^2$ for the Robinson tiling.

\begin{definition}\label{def:n-cell}
Let $n\geq 1$. An $n$-\define{square} in a model $\varphi \in \mathcal{M}(\mathcal{H}_3)$ is a subgraph of $\mathcal{G}(\mathcal{H}_3,\varphi)$ induced by $S_{n}\subset \supp(\varphi)$ for which there is $w\in \supp(\varphi)$ such that if we let $t =\varphi(w)$ then \[ S_{n} = \{wa_ta_{t+1}\dots a_{t+j-1} : 0 \leq j \leq n\} \cup \{wb_tb_0^{i-1}a_0a_1\dots a_{j-1} : 1 \leq i \leq n, 0 \leq j \leq 3^i n\}. \]

More intuitively, using the identification described above, this set can be described by \[ \{wb^ia^j : 0\leq i\leq n, 0\leq j\leq n 3^i\}.  \]

For an $n$-square, its \define{boundary} is the set of edges in its border, namely the edges of $\mathcal{G}(\mathcal{H}_3,\varphi)$ that link vertices in
\[ \partial S_n = \{wb^ia^j : i \in \{0,n\}, 0\leq j\leq n 3^i\} \cup \{wb^ia^j : 1 \leq i \leq n-1, j\in\{0, n 3^i\}\}.    \]

Finally, an \define{$n$-cell} is the subgraph of an $n$-square obtained by removing the edges in its boundary and the four vertices at the corners (namely $w$, $wb^n, wa^n$ and $wb^na^{n3^n}$). The boundary of an $n$-cell is defined as the boundary of its $n$-square containing it. These concepts are illustrated in~\Cref{fig:n-square}.
\end{definition}

\begin{figure}
    \centering
    \begin{tikzpicture}[scale=3]
            \hypersquare{3}{3}{3}
  \draw[decorate,decoration={brace,mirror,amplitude=8pt}] 
    (3.15,0) -- (3.15,1) node[midway,right=6pt] {$n$};
  \draw[decorate,decoration={brace,amplitude=8pt}] 
    (0,1.15) -- (3,1.15) node[midway,above=6pt] {$n$};
    \node at (1.5,1.5) {\textrm{Upper boundary}};
  \draw[decorate,decoration={brace,mirror,amplitude=8pt}]
    (0,-0.15) -- (3,-0.15) node[midway,below=6pt] {$n\cdot 3^n$};
    \node at (1.5,-0.5) {\textrm{Lower boundary}};
    \draw[very thick] (0,0) rectangle (3,1);
    \end{tikzpicture}
    \caption{The dimensions of an $n$-square represented in the case $n=3$. The boundary is represented with a thick line.}
    \label{fig:n-square}
\end{figure}
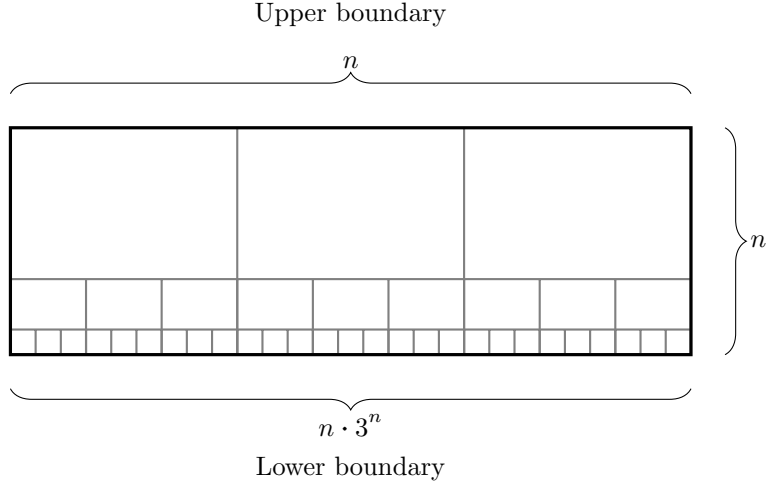
\begin{remark}\label{rem:n-squares-sizes}
An $n$-square sits above $3^n$ other $n$-squares in any model of $\mathcal{H}_3$. This follows by comparing the dimensions of the upper and lower horizontal sides. 
\end{remark}

In what follows we will be mostly interested in $2^n$-squares for $n\geq 1$. A $2^{n+1}$-square can be decomposed into two rows of $2^{n}$-squares: the upper row has two of them, and the lower row has $2\cdot 3^{2^n}$  of them. We call this the \define{natural decomposition} of a $2^{n+1}$-square into $2^n$-squares. This decomposition is not formally a partition because the boundaries of adjacent $2^n$-squares overlap; see~\Cref{fig:overlaps}. For $n\geq 1$, the \define{center} of a $2^n$-square is the vertex given by the coordinate $wb^{2^{n-1}}a^{2^{n-1}3^{2^{n-1}}}$. It corresponds to the vertex which is halfway vertically and then halfway horizontally; see~\Cref{fig:overlaps} where it is marked by a star.

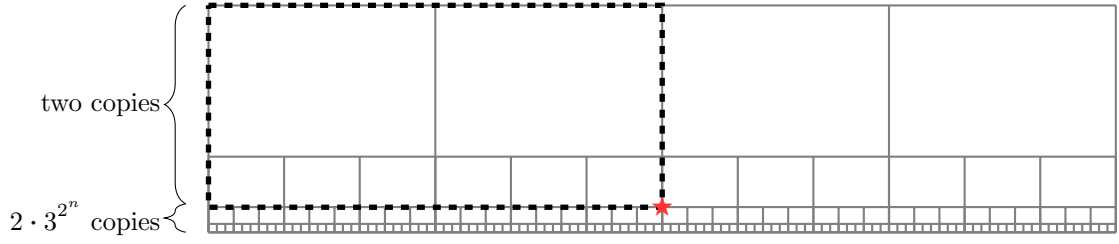
\begin{figure}[ht!]
    \begin{tikzpicture}[scale=3]
\hypersquare{3}{4}{4}    
\draw[black, line width=2pt, dashed] (0,1) -- (2,1) -- (2,1/8-0.015) -- (0,1/8-0.015) -- (0,1);
\node at (2,1/8-0.015) {{\color{red!80} \scalebox{1.5}{$\star$}}};
 \draw[decorate,decoration={brace,amplitude=8pt}] 
    (-0.1,1/8) -- (-0.1,1) node[midway,left=6pt] {two copies};
 \draw[decorate,decoration={brace,amplitude=8pt}] 
    (-0.1,0) -- (-0.1,1/8) node[midway,left=6pt] {$2\cdot 3^{2^n}$ copies};
\end{tikzpicture}\captionof{figure}{A dashed $2^n$-square inside a  $2^{n+1}$-square for $n=1$.}\label{fig:overlaps}
\end{figure}

\subsection{The basic hierarchical subshift and its blocks}\label{subsection:definition-of-X_1}

In what follows, it will be convenient to describe the alphabet $A_1$ of $\mathfrak X_1$ and its configurations pictorially. Recall that for any model $\varphi \in \mathcal{M}(\mathcal{H}_3)$, vertices in the associated graph $\mathcal{G}(\mathcal{H}_3,\varphi)$ that carry state $0$ have four outgoing edges, and vertices that carry states $1$ and $2$ have three outgoing edges. Each symbol in the alphabet can be represented as one of the diagrams shown in~\Cref{fig:alfabetito_1}. The thirteen symbols on the left can go on states marked by $0$, and the three symbols on the right on states marked by $1$ and $2$. 

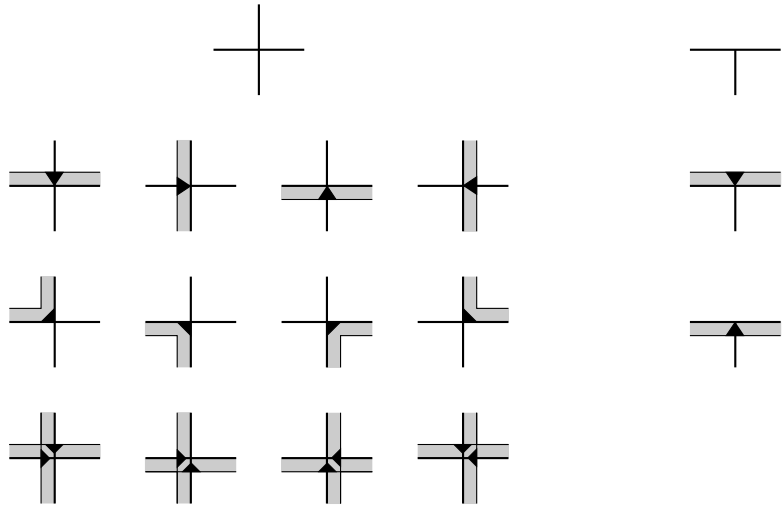
\begin{figure}[ht!]
    \centering
    \begin{tikzpicture}[scale= 0.6]

    \begin{scope}[ shift = {(15,0)}]

    \begin{scope}[ shift = {(0,3)}]
            \draw[thick] (-1,0) to (1,0);
            \draw[thick] (0,-1) to (0,0);
        \end{scope}
    
        \begin{scope}[ shift = {(0,0)}]
            \filldraw[black!20] (-1,0) -- (1,0) -- (1,0.3) -- (-1,0.3);
        \filldraw[black] (0.2,0.3) -- (-0.2,0.3) -- (0,0);
        \draw[] (-1,0.3) to (1,0.3);
            \draw[thick] (-1,0) to (1,0);
            \draw[thick] (0,-1) to (0,0);
        \end{scope}

        \begin{scope}[ shift = {(0,-3)}, rotate=0]
    \filldraw[black!20] (-1,0) -- (1,0) -- (1,-0.3) -- (-1,-0.3);
        \filldraw[black] (0.2,-0.3) -- (-0.2,-0.3) -- (0,0);
        \draw[] (-1,-0.3) to (1,-0.3);
            \draw[thick] (-1,0) to (1,0);
            \draw[thick] (0,-1) to (0,0);
    \end{scope}
    
    \end{scope}

    \begin{scope}[ shift = {(4.5,3)}]
        \draw[thick] (-1,0) to (1,0);
        \draw[thick] (0,-1) to (0,1);
    \end{scope}

    \begin{scope}[ shift = {(0,0)}]
        \filldraw[black!20] (-1,0) -- (1,0) -- (1,0.3) -- (-1,0.3);
    \filldraw[black] (0.2,0.3) -- (-0.2,0.3) -- (0,0);
    \draw[] (-1,0.3) to (1,0.3);
        \draw[thick] (-1,0) to (1,0);
        \draw[thick] (0,-1) to (0,1);
    \end{scope}

    \begin{scope}[ shift = {(3,0)}, rotate=90]
    \filldraw[black!20] (-1,0) -- (1,0) -- (1,0.3) -- (-1,0.3);
    \filldraw[black] (0.2,0.3) -- (-0.2,0.3) -- (0,0);
        \draw[thick] (-1,0) to (1,0);
        \draw[thick] (0,-1) to (0,1);
        \draw[] (-1,0.3) to (1,0.3);
    \end{scope}

    \begin{scope}[ shift = {(6,0)}, rotate=180]
    \filldraw[black!20] (-1,0) -- (1,0) -- (1,0.3) -- (-1,0.3);
    \filldraw[black] (0.2,0.3) -- (-0.2,0.3) -- (0,0);
        \draw[thick] (-1,0) to (1,0);
        \draw[thick] (0,-1) to (0,1);
        \draw[] (-1,0.3) to (1,0.3);
    \end{scope}

    \begin{scope}[ shift = {(9,0)}, rotate=270]
    \filldraw[black!20] (-1,0) -- (1,0) -- (1,0.3) -- (-1,0.3);
    \filldraw[black] (0.2,0.3) -- (-0.2,0.3) -- (0,0);
        \draw[thick] (-1,0) to (1,0);
        \draw[thick] (0,-1) to (0,1);
        \draw[] (-1,0.3) to (1,0.3);
    \end{scope}

    \begin{scope}[ shift = {(0,-3)}]
    \filldraw[black!20] (-1,0) -- (0,0) -- (0,1) -- (-0.3,1) -- (-0.3,0.3) -- (-1,0.3);
    \filldraw[black] (0,0) -- (-0.3,0) -- (0,0.3);
        \draw[thick] (-1,0) to (1,0);
        \draw[thick] (0,-1) to (0,1);
        \draw[] (-1,0.3) -- (-0.3,0.3) -- (-0.3,1);
    \end{scope}

    \begin{scope}[ shift = {(3,-3)}, rotate=90]
    \filldraw[black!20] (-1,0) -- (0,0) -- (0,1) -- (-0.3,1) -- (-0.3,0.3) -- (-1,0.3);
    \filldraw[black] (0,0) -- (-0.3,0) -- (0,0.3);
        \draw[thick] (-1,0) to (1,0);
        \draw[thick] (0,-1) to (0,1);
        \draw[] (-1,0.3) -- (-0.3,0.3) -- (-0.3,1);
    \end{scope}

    \begin{scope}[ shift = {(6,-3)}, rotate=180]
    \filldraw[black!20] (-1,0) -- (0,0) -- (0,1) -- (-0.3,1) -- (-0.3,0.3) -- (-1,0.3);
    \filldraw[black] (0,0) -- (-0.3,0) -- (0,0.3);
        \draw[thick] (-1,0) to (1,0);
        \draw[thick] (0,-1) to (0,1);
        \draw[] (-1,0.3) -- (-0.3,0.3) -- (-0.3,1);
    \end{scope}

    \begin{scope}[ shift = {(9,-3)}, rotate=270]
    \filldraw[black!20] (-1,0) -- (0,0) -- (0,1) -- (-0.3,1) -- (-0.3,0.3) -- (-1,0.3);
    \filldraw[black] (0,0) -- (-0.3,0) -- (0,0.3);
        \draw[thick] (-1,0) to (1,0);
        \draw[thick] (0,-1) to (0,1);
        \draw[] (-1,0.3) -- (-0.3,0.3) -- (-0.3,1);
    \end{scope}

    \begin{scope}[ shift = {(0,-6)}]
        \filldraw[black!20] (-1,0) -- (1,0) -- (1,0.3) -- (-1,0.3);
    
     \begin{scope}[rotate=90]
        \filldraw[black!20] (-1,0) -- (1,0) -- (1,0.3) -- (-1,0.3);
    \filldraw[black] (0.2,0.3) -- (-0.2,0.3) -- (0,0.1);
    \end{scope}
    \filldraw[black] (0.2,0.3) -- (-0.2,0.3) -- (0,0.1);
    \draw[] (-1,0.3) to (-0.3,0.3) -- (-0.3,1);
    \draw[] (-0.3,-1) -- (-0.3,1);
    \draw[] (-1,0.3) to (1,0.3);
    \draw[thick] (-1,0) to (1,0);
    \draw[thick] (0,-1) to (0,1);
        
    \end{scope}

     \begin{scope}[ shift = {(3,-6)}, rotate = 90]
        \filldraw[black!20] (-1,0) -- (1,0) -- (1,0.3) -- (-1,0.3);
    
     \begin{scope}[rotate=90]
        \filldraw[black!20] (-1,0) -- (1,0) -- (1,0.3) -- (-1,0.3);
    \filldraw[black] (0.2,0.3) -- (-0.2,0.3) -- (0,0.1);
    \end{scope}
    \filldraw[black] (0.2,0.3) -- (-0.2,0.3) -- (0,0.1);
    \draw[] (-1,0.3) to (-0.3,0.3) -- (-0.3,1);
    \draw[] (-0.3,-1) -- (-0.3,1);
    \draw[] (-1,0.3) to (1,0.3);
    \draw[thick] (-1,0) to (1,0);
    \draw[thick] (0,-1) to (0,1);
        
    \end{scope}

    \begin{scope}[ shift = {(6,-6)}, rotate = 180]
        \filldraw[black!20] (-1,0) -- (1,0) -- (1,0.3) -- (-1,0.3);
    
     \begin{scope}[rotate=90]
        \filldraw[black!20] (-1,0) -- (1,0) -- (1,0.3) -- (-1,0.3);
    \filldraw[black] (0.2,0.3) -- (-0.2,0.3) -- (0,0.1);
    \end{scope}
    \filldraw[black] (0.2,0.3) -- (-0.2,0.3) -- (0,0.1);
    \draw[] (-1,0.3) to (-0.3,0.3) -- (-0.3,1);
    \draw[] (-0.3,-1) -- (-0.3,1);
    \draw[] (-1,0.3) to (1,0.3);
    \draw[thick] (-1,0) to (1,0);
    \draw[thick] (0,-1) to (0,1);
        
    \end{scope}

    \begin{scope}[ shift = {(9,-6)}, rotate = -90]
        \filldraw[black!20] (-1,0) -- (1,0) -- (1,0.3) -- (-1,0.3);
    
     \begin{scope}[rotate=90]
        \filldraw[black!20] (-1,0) -- (1,0) -- (1,0.3) -- (-1,0.3);
    \filldraw[black] (0.2,0.3) -- (-0.2,0.3) -- (0,0.1);
    \end{scope}
    \filldraw[black] (0.2,0.3) -- (-0.2,0.3) -- (0,0.1);
    \draw[] (-1,0.3) to (-0.3,0.3) -- (-0.3,1);
    \draw[] (-0.3,-1) -- (-0.3,1);
    \draw[] (-1,0.3) to (1,0.3);
    \draw[thick] (-1,0) to (1,0);
    \draw[thick] (0,-1) to (0,1);
        
    \end{scope}
\end{tikzpicture} 
    \caption{The alphabet $A_1$ of $\mathfrak{X}_1$.}
    \label{fig:alfabetito_1}
\end{figure}

As is usual in these constructions, we will impose the local rule that the decorations of symbols on adjacent vertices must match along their common edge. Hence, we may as well conceive that the alphabet lies on the edges (although formally it is on the vertices). Of special importance will be the 6 symbols in~\Cref{fig:crosses} which are called \define{crosses} and to which we give a special name. These play a role analogous to that of the crosses in the Robinson tiling.

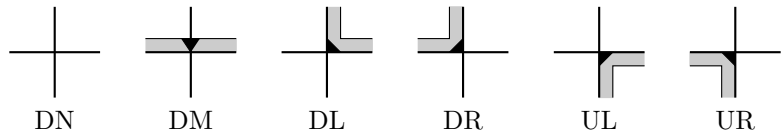
\begin{figure}
    \centering
    
\begin{tikzpicture}[scale = 0.6]

    \begin{scope}[ shift = {(0,0)}]
            \draw[thick] (-1,0) to (1,0);
            \draw[thick] (0,-1) to (0,1);
            \node at (0,-1.5) {$\mathrm{DN}$};
        \end{scope}

        \begin{scope}[ shift = {(3,0)}]
            \filldraw[black!20] (-1,0) -- (1,0) -- (1,0.3) -- (-1,0.3);
        \filldraw[black] (0.2,0.3) -- (-0.2,0.3) -- (0,0);
        \draw[] (-1,0.3) to (1,0.3);
            \draw[thick] (-1,0) to (1,0);
            \draw[thick] (0,-1) to (0,1);
            \node at (0,-1.5) {$\mathrm{DM}$};
        \end{scope}

        \begin{scope}[ shift = {(6,0)}]
        \begin{scope}[ rotate=270]
            \filldraw[black!20] (-1,0) -- (0,0) -- (0,1) -- (-0.3,1) -- (-0.3,0.3) -- (-1,0.3);
    \filldraw[black] (0,0) -- (-0.3,0) -- (0,0.3);
        \draw[thick] (-1,0) to (1,0);
        \draw[thick] (0,-1) to (0,1);
        \draw[] (-1,0.3) -- (-0.3,0.3) -- (-0.3,1);
        
        \end{scope}
        \node at (0,-1.5) {$\mathrm{DL}$};
    \end{scope}

        \begin{scope}[ shift = {(9,0)}]
    \filldraw[black!20] (-1,0) -- (0,0) -- (0,1) -- (-0.3,1) -- (-0.3,0.3) -- (-1,0.3);
    \filldraw[black] (0,0) -- (-0.3,0) -- (0,0.3);
        \draw[thick] (-1,0) to (1,0);
        \draw[thick] (0,-1) to (0,1);
        \draw[] (-1,0.3) -- (-0.3,0.3) -- (-0.3,1);
        \node at (0,-1.5) {$\mathrm{DR}$};
    \end{scope}

    \begin{scope}[ shift = {(15,0)}]
        \begin{scope}[ rotate=90]
            \filldraw[black!20] (-1,0) -- (0,0) -- (0,1) -- (-0.3,1) -- (-0.3,0.3) -- (-1,0.3);
    \filldraw[black] (0,0) -- (-0.3,0) -- (0,0.3);
        \draw[thick] (-1,0) to (1,0);
        \draw[thick] (0,-1) to (0,1);
        \draw[] (-1,0.3) -- (-0.3,0.3) -- (-0.3,1);
        
        \end{scope}
        \node at (0,-1.5) {$\mathrm{UR}$};
    \end{scope}

    \begin{scope}[ shift = {(12,0)}]
        \begin{scope}[ rotate=180]
            \filldraw[black!20] (-1,0) -- (0,0) -- (0,1) -- (-0.3,1) -- (-0.3,0.3) -- (-1,0.3);
    \filldraw[black] (0,0) -- (-0.3,0) -- (0,0.3);
        \draw[thick] (-1,0) to (1,0);
        \draw[thick] (0,-1) to (0,1);
        \draw[] (-1,0.3) -- (-0.3,0.3) -- (-0.3,1);
        
        \end{scope}
        \node at (0,-1.5) {$\mathrm{UL}$};
    \end{scope}

\end{tikzpicture}
    \caption{The $6$ types of crosses. The letters U and D stand for up and down respectively, while L, R, N and M stand for left, right, null and middle respectively.}
    \label{fig:crosses}
\end{figure}

The decorations in symbols of $A_1$ are meant to encode the borders of rectangles, and the black triangles are meant to encode orientation; they will always point towards the outside of a rectangle.

Next we shall describe a collection of patterns on this alphabet whose supports are $2^n$-cells ($n\geq 1$). We shall use those to define $\mathfrak{X}_1$.

\begin{definition}\label{def:2-blocks}
    A $2$-block of type $\tau$ is a pattern whose support is a $2$-cell and which in the center carries a cross of type $\tau$. Furthermore, the decorations emanating from this cross must propagate in a straight line until they reach the boundary. See~\Cref{fig:2-block-example}. 
\end{definition}

\newcommand{\hypersquarebis}[3]{
\pgfmathsetmacro{\k}{1/pow(#1,#2-1)}
\pgfmathsetmacro{\sy}{0}
\foreach \i in {1,...,#2}{
\pgfmathsetmacro{\j}{pow(#1,\i-1)}
        \begin{scope}[shift={(0,(\sy*\k)}]
        \draw[step={\j*\k},black, thick] (0,0) grid ({pow(#1,#2-1)*#3*\k},\j*\k);
        \end{scope}
        \pgfmathsetmacro{\sy}{\sy + \j}
        } 
}
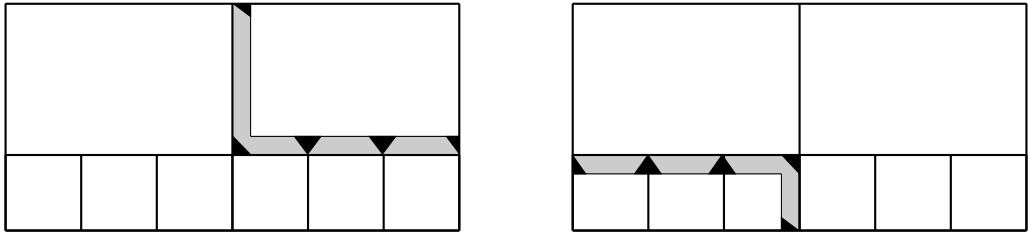
\begin{figure}[ht!]
    \centering
    \begin{tikzpicture}[scale = 3]
    \filldraw[black!20] (1.08,0.335+0.08) -- (2,0.335+0.08) -- (2,0.335) -- (1,0.335) -- (1,1) -- (1.08,1);
    \filldraw[black] (1,0.335) -- (1.08,0.335) -- (1,0.335+0.082);
    \filldraw[black] (1.33,0.335) -- (1.39,0.335+0.08) -- (1.27,0.335+0.08);
    \filldraw[black] (1.66,0.335) -- (1.72,0.335+0.08) -- (1.60,0.335+0.08);
    \filldraw[black] (2,0.335) -- (2,0.335+0.08) -- (1.94,0.335+0.08);
    \filldraw[black] (1,1) -- (1.08,1) -- (1.08,0.94);
    \hypersquarebis{3}{2}{2}
    \draw (1.08,0.335+0.08) -- (2,0.335+0.08);
    \draw (1.08,0.335+0.08) -- (1.08,1);

    \begin{scope}[shift = {(2.5,0)}]
    \filldraw[black!20] (0.92,0.25) -- (0,0.25) -- (0,0.335) -- (1,0.335) -- (1,0) -- (0.92,0);
    \filldraw[black] (1,0.335) -- (0.92,0.335) -- (1,0.25);
    \filldraw[black] (0.66,0.335) -- (0.6,0.25) -- (0.72,0.25);
    \filldraw[black] (0.33,0.335) -- (0.39,0.25) -- (0.27,0.25);
    \filldraw[black] (0,0.335) -- (0,0.25) -- (0.06,0.25);
    \filldraw[black] (1,0) -- (0.92,0) -- (0.92,0.06);
    \hypersquarebis{3}{2}{2}
    \draw (0.92,0.25) -- (0,0.25);
    \draw (0.92,0.25) -- (0.92,0);
\end{scope}
\end{tikzpicture}
    \caption{On the left a $2$-block of type DL and on the right a $2$-block of type UR. }
    \label{fig:2-block-example}
\end{figure}

We shall define $2^n$-blocks for $n \geq 1$ inductively. For illustration purposes we will begin with the case of $4$-blocks. A $4$-block of type $\tau$ is a pattern whose support is a $4$-cell and is obtained through the following four steps.

\newcommand{\DL}{
\filldraw[black!20] (1.08,0.335+0.08) -- (2,0.335+0.08) -- (2,0.335) -- (1,0.335) -- (1,1) -- (1.08,1);
    \filldraw[black] (1,0.335) -- (1.08,0.335) -- (1,0.335+0.082);
    \filldraw[black] (1.33,0.335) -- (1.39,0.335+0.08) -- (1.27,0.335+0.08);
    \filldraw[black] (1.66,0.335) -- (1.72,0.335+0.08) -- (1.60,0.335+0.08);
    \filldraw[black] (2,0.335) -- (2,0.335+0.08) -- (1.94,0.335+0.08);
    \filldraw[black] (1,1) -- (1.08,1) -- (1.08,0.94);
    \hypersquarebis{3}{2}{2}
    \draw (1.08,0.335+0.08) -- (2,0.335+0.08);
    \draw (1.08,0.335+0.08) -- (1.08,1);
}

\newcommand{\DLbis}{
\filldraw[black!20] (1.2,0.335+0.20) -- (2,0.335+0.20) -- (2,0.335) -- (1,0.335) -- (1,1) -- (1.2,1);
    \filldraw[black] (1,0.335) -- (1.2,0.335) -- (1,0.335+0.202);
    \filldraw[black] (1.33,0.335) -- (1.39,0.335+0.20) -- (1.27,0.335+0.20);
    \filldraw[black] (1.66,0.335) -- (1.72,0.335+0.20) -- (1.60,0.335+0.20);
    \filldraw[black] (2,0.335) -- (2,0.335+0.20) -- (1.94,0.335+0.20);
    \filldraw[black] (1,1) -- (1.2,1) -- (1.2,1-0.1);
    \hypersquarebis{3}{2}{2}
    \draw (1.2,0.335+0.20) -- (2,0.335+0.20);
    \draw (1.2,0.335+0.20) -- (1.2,1);
}

\newcommand{\DRbis}{
    \begin{scope}[rotate=90,yscale=-1,rotate=-90]
    \DLbis
    \end{scope}
}

\newcommand{\UR}{
    \filldraw[black!20] (1-0.03,0.335-0.04) -- (0,0.335-0.04) -- (0,0.335) -- (1,0.335) -- (1,0) -- (1-0.03,0);
    \filldraw[black] (1,0.335) -- (1-0.03,0.335) -- (1,0.335-0.04);
    \filldraw[black] (0.66,0.335) -- (0.66-0.03,0.335-0.04) -- (0.66+0.03,0.335-0.04);
    \filldraw[black] (0.33,0.335) -- (0.33-0.03,0.335-0.04) -- (0.33+0.03,0.335-0.04);
    \filldraw[black] (0,0.335) -- (0,0.335-0.04) -- (0+0.03,0.335-0.04);
    \filldraw[black] (1,0) -- (1-0.03,0) -- (1-0.03,0.03);
    \hypersquarebis{3}{2}{2}
    \draw (1-0.03,0.335-0.04) -- (0,0.335-0.04);
    \draw (1-0.03,0.335-0.04) -- (1-0.03,0);
}

\newcommand{\UL}{
    \begin{scope}[rotate=90,yscale=-1,rotate=-90]
    \UR
    \end{scope}
}

\newcommand{\DN}{
    \hypersquarebis{3}{2}{2}
}

\newcommand{\DM}{
    \filldraw[black!20] (0,0.335) rectangle (2,0.335+0.20);
    \filldraw[black] (1.33,0.335) -- (1.39,0.335+0.20) -- (1.27,0.335+0.20);
    \filldraw[black] (1.66,0.335) -- (1.72,0.335+0.20) -- (1.60,0.335+0.20);
    \filldraw[black] (2,0.335) -- (2,0.335+0.20) -- (1.94,0.335+0.20);
    \filldraw[black] (1,0.335) -- (0.94,0.335+0.20) -- (1.06,0.335+0.20);
    \filldraw[black] (0.66,0.335) -- (0.6,0.335+0.20) -- (0.72,0.335+0.20);
    \filldraw[black] (0.33,0.335) -- (0.39,0.335+0.20) -- (0.27,0.335+0.20);
    \filldraw[black] (0,0.335) -- (0,0.335+0.20) -- (0.06,0.335+0.20);
    \hypersquarebis{3}{2}{2}
    \draw (0,0.335+0.20) -- (2,0.335+0.20);
}

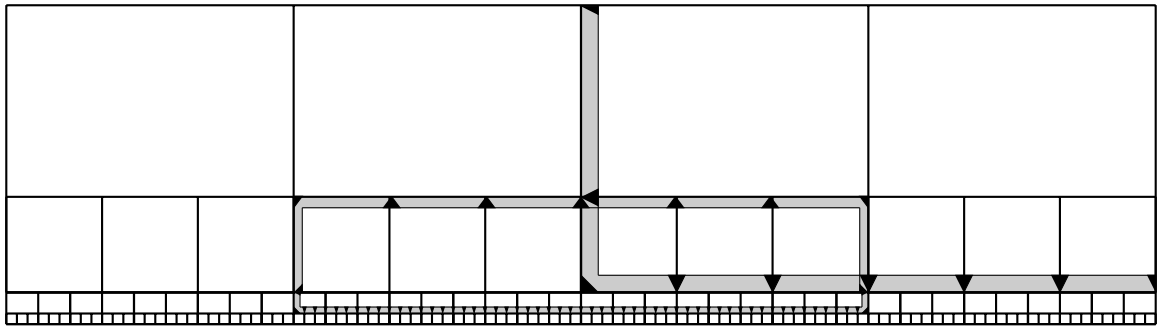
\begin{figure}
    \centering
    
\begin{tikzpicture}[scale = 3.8]

\begin{scope}[shift = {(-2,-0.67)}, scale = 2]
    \filldraw[black!20] (1+0.03,0.335+0.03) -- (2,0.335+0.03) -- (2,0.335) -- (1,0.335) -- (1,0.833) -- (1+0.03,0.833);
\end{scope}

    \begin{scope}[shift = {(0,0)}]
        \UL
    \end{scope}
    \begin{scope}[shift = {(0,0)}]
        \UR
    \end{scope}
    \foreach \i in {0,...,3}{
        \begin{scope}[shift = {(2*\i/9-2,-0.11)}, scale = 1/9]
            \DN
        \end{scope}
     }
     \foreach \i in {4}{
        \begin{scope}[shift = {(2*\i/9-2,-0.11)}, scale = 1/9]
            \DLbis
        \end{scope}
     }
     \foreach \i in {5,...,12}{
        \begin{scope}[shift = {(2*\i/9-2,-0.11)}, scale = 1/9]
            \DM
        \end{scope}
     }
     \foreach \i in {14}{
        \begin{scope}[shift = {(2*\i/9-2,-0.11)}, scale = 1/9]
            \DRbis
        \end{scope}
     }
     \foreach \i in {14,...,17}{
        \begin{scope}[shift = {(2*\i/9-2,-0.11)}, scale = 1/9]
            \DN
        \end{scope}
     }
     \begin{scope}[shift = {(-2,-0.67)}, scale = 2]

    \filldraw[black] (1,0.335) -- (1+0.03,0.335) -- (1,0.335+0.03);
    \foreach \i in {1.166, 1.333, 1.5, 1.666, 1.833}{
        \filldraw[black] (\i,0.335) -- (\i+0.015,0.335+0.03) -- (\i-0.015,0.335+0.03);
    }
    \filldraw[black] (2,0.335) -- (2,0.335+0.03) -- (2-0.015,0.335+0.03);
    \filldraw[black] (1,0.5) -- (1+0.03,0.5+0.015) -- (1+0.03,0.5-0.015);
    \filldraw[black] (1,0.833) -- (1+0.03,0.833) -- (1+0.03,0.833-0.015);
    \draw (1+0.03,0.335+0.03) -- (2,0.335+0.03);
    \draw (1+0.03,0.335+0.03) -- (1+0.03,0.833);
\end{scope}
\end{tikzpicture}
    
    \caption{A $4$-block of type DL. The decorations have been rescaled for the sake of visibility.}
    \label{fig:4-block}
\end{figure}

\textbf{Step 1.} We decompose a $4$-cell into $2$-cells according to the natural decomposition. Hence the upper row consists of two $2$-cells, and the bottom row consists of eighteen $2$-cells.

\textbf{Step 2.} In the upper row, we place from left to right a $2$-block of type UL and a $2$-block of type UR.

\textbf{Step 3.} In the bottom row, we place from left to right the following $2$-blocks:
\begin{itemize}
    \item Four $2$-blocks of type DN.
    \item A $2$-block of type DL.
    \item Eight $2$-blocks of type DM.
    \item A $2$-block of type DR.
    \item Four $2$-blocks of type DN.
\end{itemize}

\textbf{Step 4.} In the center of the $4$-cell we place a cross of type $\tau$ and extend the decorations in a straight line until they hit the boundary.

An illustration can be seen in~\Cref{fig:4-block}. As a $4$-block is entirely determined by the cross at its center, there are $6$ types of $4$-blocks, and thus we may think of $4$-blocks as larger copies of $2$-blocks. 

Next we define $2^n$-blocks inductively.
\begin{definition}
    Let $n\geq 1$. A \define{$2^{n+1}$-block of type} $\tau$ is a pattern whose support is a $2^{n+1}$-cell and which is constructed through the following four steps.

\textbf{Step 1.} We decompose the $2^{n+1}$-cell into $2^{n}$-cells according to the natural decomposition. Hence the upper row consists of two $2^{n}$-cells, and the bottom row consists of $2\cdot 3^{2^{n}}$ $2^{n}$-cells.

\textbf{Step 2.} In the upper row, we place from left to right a $2^{n}$-block of type UL and a $2^{n}$-block of type UR.

\textbf{Step 3.} In the bottom row, we place from left to right the following $2^n$-blocks:
\begin{itemize}
    \item $\frac{3^{2^{n}}-1}{2}$ $2^{n}$-blocks of type DN.
    \item A $2^{n}$-block of type DL.
    \item $3^{2^{n}}-1$ $2^{n}$-blocks of type DM.
    \item A $2^{n}$-block of type DR.
    \item $\frac{3^{2^{n}}-1}{2}$ $2^{n}$-blocks of type DN.
\end{itemize}

\textbf{Step 4.} In the center of the $2^{n+1}$-cell we place a cross of type $\tau$ and extend the decorations by straight lines until they hit the boundary.
\end{definition}

The fundamental idea is that within a $2^{n+1}$-block, the $2^n$-blocks in its natural decomposition must arrange themselves forming a ``rectangle'' with their central crosses. Furthermore, the new decoration on the cross at the center must extend towards the boundary as in the Robinson tiling. Notice that with this description the triangle markings in any rectangle always point towards the exterior.

\begin{definition}
    The $\mathcal{H}_3$-subshift $\mathfrak{X}_1$ is the set of all pairs $(\varphi,x)$ where $\varphi \in \mathcal{M}(\mathcal{H}_3)$ and $x$ is a configuration on alphabet $A_1$ such that for every subpattern of $x$, there is $n \geq 1$ for which it occurs as a subpattern of a $2^n$-block.
\end{definition}

The next result states that it is possible to turn $\mathfrak{X}_1$ into an SFT by adding extra decorations. This SFT extension can be guaranteed to simultaneously have zero Medvedev degree and to satisfy a technical condition that we will need in the next section. 
\begin{theorem}\label{thm:X_1-sofic}
The subshift $\mathfrak{X}_1$ is sofic. That is, there exists a $\mathcal{H}_3$-SFT $\mathfrak{X}_0$ with alphabet $A_0$ and a map $\pi\colon A_0 \to A_1$ such that \[ \mathfrak{X}_1 = \{\Phi(\varphi,x) : (\varphi,x)\in \mathfrak{X}_0\},  \]
where $\Phi(\varphi,x)=(\varphi,\pi(x(w))_{w \in \supp(\varphi)})$. 

Furthermore, the SFT $\mathfrak X_0$ can be chosen to have zero Medvedev degree, with this being witnessed by a computable configuration $(\varphi_0,x_0)\in\mathfrak X_0$ whose image $(\varphi_0,x_1)=\Phi(\varphi_0,x_0)$ satisfies the following technical condition: for every $w\in\supp(\varphi_0)$ there is $n \geq 1$ such that $w$ is inside the $4^n$-border determined by a $2^{2n+1}$-block in $x_1$.
\end{theorem}

As the proof of this result is very technical, we postpone it to~\Cref{sec:appendix}. Intuitively, we just have to add a lot of extra decorations and local rules that enforce the global structure of $\mathfrak{X}_1$, and then define $\pi$ as the map that erases these extra decorations. However, in order to classify the Medvedev degrees of SFTs on $\mathcal{H}_3$, it will be sufficient to use~\Cref{thm:X_1-sofic}.

\subsection{The borders inside \texorpdfstring{$\mathfrak{X}_1$}{X1} }
Here we review in detail the hierarchical structure of $\mathfrak{X}_1$ induced by borders inside blocks. These results will be of importance in the next section.

\begin{definition}
Fix $n\geq 2$ and consider a $2^n$-block. The rectangle formed by the central decorations of the $2^{n-1}$-blocks delimits a region, which we refer to as a $2^{n-1}$-\define{border} (see \Cref{fig:4-block}). 
\end{definition}
 This terminology is motivated by the fact that the height of a $2^{n-1}$-border  is $2^{n-1}$ edges. However, we remark that the upper and lower horizontal sides have different dimensions. More specifically, if we decompose the $2^n$-block into $2^{n-1}$-blocks and then $2^{n-2}$-blocks, one sees that the dimensions of the $2^{n-1}$-border inside the $2^n$-block are as in \Cref{fig:dimensions-of-border}.

\begin{figure}[]
\centering
\begin{tikzpicture}[scale=3.8]
\hypersquare{3}{4}{4}    
\draw[black, line width=1pt] (1,1/3) -- (3,1/3) -- (3,1/27) -- (1,1/27) -- (1,1/3);
 \draw[decorate,decoration={brace,amplitude=8pt}] 
    (1,1/27) -- (1,1/3) node[midway,left=4pt] {$2^{n-1}$};
\draw[decorate,decoration={brace,amplitude=8pt,aspect=0.3}] 
    (1,1/3) -- (3,1/3) node[pos=0.3,above=6pt] {$2^{n-1}3^{2^{n-2}}$};
\draw[decorate,decoration={brace,mirror,amplitude=14pt,aspect=0.3}] 
    (1,1/27) -- (3,1/27) node[pos=0.3,below=14pt] {$2^{n-1}3^{3\cdot 2^{n-2}}$};
\end{tikzpicture}
\caption{A $2^n$-block and the $2^{n-1}$-border inside it (represented with $n=2$).}
\label{fig:dimensions-of-border}
\end{figure}
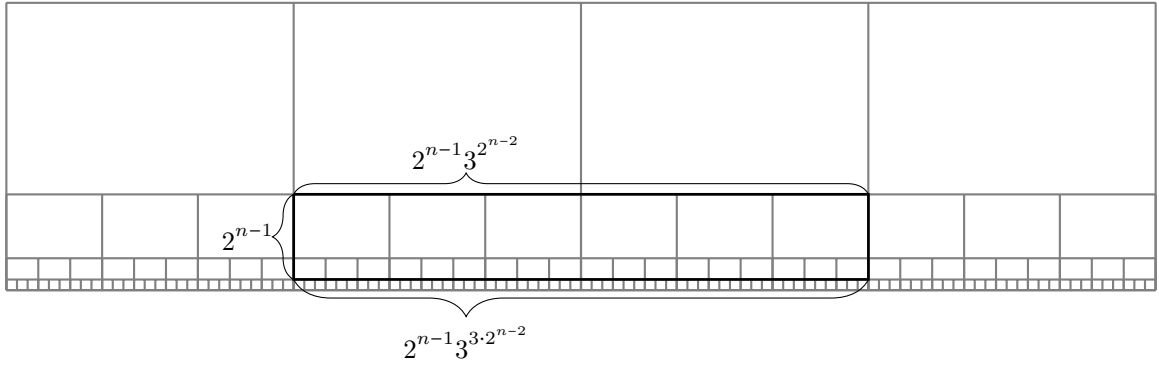

It is convenient to understand a $2^{n-1}$-border (for $n \geq 3)$ as a pair of rows of $2^{n-2}$-blocks, in the sense of the following proposition.
\begin{proposition}\label{prop:border-decomposition}
    Let $n\geq 3$. A $2^{n-1}$-border is formed by two horizontal rows of $2^{n-2}$-blocks. The number of $2^{n-2}$-blocks in the upper row is $2\cdot 3^{2^{n-2}}$, and the number of $2^{n-2}$-blocks in the lower row is $2\cdot 3^{2^{n-1}}$. 
\end{proposition}
\begin{proof}
The first claim follows by decomposing a $2^n$-block into $2^{n-1}$-blocks, and then decomposing those $2^{n-1}$-blocks into $2^{n-2}$-blocks. For the dimensions, we already noted that the upper horizontal side of a $2^{n-1}$-border has $2^{n-1}\cdot 3^{2^{n-2}}$ edges. Furthermore, recall that the upper horizontal side of a $2^{n-2}$-block has $2^{n-2}$ edges. Since $2^{n-2}\cdot 2\cdot 3^{2^{n-2}}=2^{n-1}\cdot 3^{2^{n-2}}$, the claim regarding the upper row of $2^{n-2}$-blocks follows. Next, recall that a $2^{n-2}$-block sits above $3^{2^{n-2}}$ $2^{n-2}$-blocks. Thus the number of $2^{n-2}$-blocks in the lower half of our $2^{n-1}$-border equals $2\cdot 3^{2^{n-2}}\cdot 3^{2^{n-2}}=2\cdot 3^{2^{n-2}+2^{n-2}}=2\cdot 3^{2^{n-1}}$. 
\end{proof}

\begin{proposition}\label{borders-are-nicely-aligned}
Let $n\geq 2$ and consider a $2^{n-1}$-border inside a $2^n$-block. Let $v$ be the vertex in the lower left or lower right corner of the $2^{n-1}$-border. If we follow the vertical edges from $v$ downwards, then we eventually reach an upper corner of another $2^{n-1}$-border. 
\end{proposition}
\begin{proof}
Observe that $v$ is the central vertex of one of the $2^{n-1}$-blocks which form our $2^n$-block. Denote by $B$ this $2^{n-1}$-block. If we walk downwards from $v$, then we will reach the central cross of a $2^{n-1}$-block which is directly below $B$, because below $B$ there is an odd number of $2^{n-1}$-blocks (exactly $3^{2^{n-1}}$). The central vertex of this $2^{n-1}$-block necessarily has a cross of type UL or UR, so it is the upper corner of a $2^{n-1}$-border. 
\end{proof}

Let $k \geq 1$ and consider a $4^k$-border $B$. An \define{interior row} is a row of horizontal edges within the rectangle delimited by $B$ that connects its left side with the right side and does not intersect the upper or lower sides. Similarly, an \define{interior column} is a column of vertical edges within the rectangle delimited by $B$ that connects the upper side with the bottom side and does not intersect the left or right sides.

We shall be interested in the interior rows and columns of a $4^k$-border that do not intersect smaller $4^{i}$-borders inside it for $i < k$. We introduce the following terminology. 

\begin{definition}\label{def:free}
Let $k\geq 1$ and consider a $4^{k}$-border. An interior row or column is called \define{free} if it does not intersect a $4^{i}$-border for $1\leq i< k$. Otherwise, it is called \define{obstructed}. 
\end{definition}

In order to count the number of free rows inside a border, it is convenient to first count the analogous number of free rows inside a block. 
\begin{lemma}\label{lem:free-rows-inside-a-block}
    Let $n\geq 1$ and let $B$ be a $2^{2n+1}$-block. There are precisely $2^n$ free rows in $B$ (excluding those on the upper and lower boundaries) that do not intersect any $2^{2k}$-border inside $B$, for $1\leq k\leq n$. The row directly below the upper boundary (resp. above the lower boundary) is among them. Furthermore, no pair of these rows can be consecutive.   
\end{lemma}
\begin{proof}
The proof is by induction. Let $E_n$ be the number of rows in the statement.  For $n=1$, a direct verification shows that $E_1=2$. That is, one builds a $2^3$-block, which contains a $4$-border inside it. Excluding the upper and lower exterior rows of the $2^3$-block, we have 7 rows. From these 7, 5 intersect the $4$-border inside it, and the remaining 2 do not. Also observe that the two remaining rows are nonconsecutive. More precisely, they are the rows directly below the upper boundary and directly above the lower boundary respectively.

Now assume the claim holds for $n-1$. Given a $2^{2n+1}$-block, we decompose it into $2^{2n}$-blocks and then into $2^{2n-1}$-blocks. The $2^{2n}$-border inside the given $2^{2n+1}$-block creates a vertical strip of height $2^{2n}$ where every row is obstructed. The regions above and below this $2^{2n}$-border are formed by rows of $2^{2n-1}$-blocks. In these two rows of  $2^{2n-1}$-blocks, the only obstructions occur inside the individual $2^{2n-1}$-blocks, as there are no $2^{2k}$-borders intersecting their boundaries ($k\leq n$). Thus $E_{n}=2E_{n-1}$ and our claim that $E_n=2^n$ follows by the inductive hypothesis. Since the $2^{2n-1}$-blocks are between obstructed rows (obstructed by the $2^{2n}$-border inside our $2^{2n+1}$-block), the claim about consecutive rows also follows for $n$ by the inductive hypothesis. 
\end{proof}
\begin{proposition}\label{lem:free-rows-inside-a-border}
    Let $n\geq 1$. The number of free interior rows in a $2^{2n}$-border is  $2^n+1$. Furthermore, the number of free interior columns is at least $2\cdot 3^{2^{2n-1}}-1$. 
\end{proposition}
\begin{proof}
Case $n=1$ is direct. For $n \geq 2$ we know from \Cref{prop:border-decomposition} that our $2^{2n}$-border is formed by two rows of $2^{2n-1}$-blocks. It follows from \Cref{lem:free-rows-inside-a-block} that each row of $2^{2n-1}$-blocks contributes exactly $2^{n-1}$ free rows to our $2^{2n}$-border. There is one more free row, which separates the rows of $2^{2n-1}$-blocks. Thus $2^{n-1}+2^{n-1}+1=2^n+1$ as claimed. 

For the lower bound for the number of free vertical columns, we decompose our $2^{2n}$-border into two rows of $2^{2n-1}$-blocks. In the upper half we have exactly $2\cdot 3^{2^{2n-1}}$ of them. It is clear that every pair of consecutive ones generates a free vertical column, so we obtain at least $2\cdot 3^{2^{2n-1}}-1$ free vertical columns. 
\end{proof}
\begin{lemma}\label{lem:free-obstructed-rows-disposition}
Let $n\geq 1$. In a $2^{2n}$-border, each free row above (resp. below) the center row sits directly below (resp. above) an obstructed row, or directly below (resp. above) a row in the boundary of the $2^{2n}$-border. Furthermore, the row at the center is between two free rows.
\end{lemma}
\begin{proof}
This follows from \Cref{prop:border-decomposition} and \Cref{lem:free-rows-inside-a-block}.
\end{proof}
\begin{remark}\label{rem:leftmost-red-border}
    Let $n\geq 1$. In a $4^n$-border, the leftmost free interior column is adjacent to the left side of the border. This follows from the same argument used in \Cref{lem:free-obstructed-rows-disposition}.
\end{remark}

\section{The Medvedev degrees of SFTs on \texorpdfstring{$\mathcal{H}_3$}{H3}}\label{sec:medvedev}
The goal of this section is to prove~\Cref{mainthm:grafo_medved}, which states that every $\Pi_1^0$ Medvedev degree is realized by an SFT on $\mathcal H_3$. 


Following Robinson's proof  \cite{robinson_undecidability_1971}, it is straightforward to embed computation diagrams of Turing machines into the ``free space'' inside borders. However, in order to realize a prescribed Medvedev degree, we must ensure that the Turing machines throughout the configuration have the same input. This is the main difficulty of the proof. In the Euclidean setting, the same problem can be solved using ``diagonal signals'' \cite{myers_nonrecursive_1974,Simpson_2014_Medvedev}, but the argument cannot be replicated easily in the hyperbolic setting.

The sketch of the proof in this section is as follows. We shall construct an $\mathcal H_3$-SFT $\mathfrak{X}$ with the property that for every $n\in\NN$ there is a single bit from $\{0,1\}$ associated to all $4^n$-borders. In this manner we encode in every configuration a sequence in $\{0,1\}^{\NN}$. Furthermore, we shall impose local rules that encode ``free'' rows and columns which induce a grid within every $4^n$-border. Finally, we will take an arbitrary $\Pi_1^0$ set $P\subset \{0,1\}^\NN$ whose Medvedev degree we want to realize and choose a Turing machine which halts precisely on those words whose cylinder sets lie in the complement of $P$. We shall encode this Turing machine within every grid induced by $4^n$-borders. This machine will have access to the first $n$ elements of the sequence, and it will be able to halt for some sufficiently large $n$ if the encoded sequence is not in $P$.

We start by fixing a $\Pi_1^0$ set $P\subset \{0,1\}^\NN$. To describe our $\mathcal H_3$-SFT $\mathfrak{X}$ in a more understandable manner, we will use the notion of ``layers''. Formally, at every step we begin with an $\mathcal H_3$-SFT, take the product of the current alphabet with a new one, and impose new local rules to produce a new $\mathcal H_3$-SFT. Each layer $i=0,1,\dots,5$ has an alphabet $A_i$, and the final SFT will have alphabet $A=A_0\times\dots\times A_5$.

\subsection*{Layer 0 (The SFT extension \texorpdfstring{$\mathfrak{X}_0$}{X0})} This layer is the SFT $\mathfrak{X}_0$ from~\Cref{thm:X_1-sofic} which extends $\mathfrak{X}_1$. Adding this layer ensures that the final result of our construction is an SFT. 



\subsection*{Layer 1 (The hierarchical shift \texorpdfstring{$\mathfrak{X}_1$}{X1})} In this layer we have the image of Layer $0$ by the map $\pi$ from $\mathfrak{X}_0$ to $\mathfrak{X}_1$ in~\Cref{thm:X_1-sofic}. Thus in Layer 1 we have configurations from $\mathfrak{X}_1$. 

\subsection*{Layer 2 (Alternating layer \texorpdfstring{$\mathfrak{X}_2$}{X2})} In this layer, we take $A_2$ as the alphabet where the boundary decorations in~\Cref{fig:alfabetito_1} are assigned either the color green $\begin{tikzpicture}[scale = 0.2]
    \draw[thick] (0,0) rectangle (1,1);
    \filldraw[green!20] (0,0) rectangle (1,1);
\end{tikzpicture}$ or red $\begin{tikzpicture}[scale = 0.2]
    \draw[thick] (0,0) rectangle (1,1);
    \filldraw[pattern=crosshatch] (0,0) rectangle (1,1);
    \filldraw[color=red, opacity=0.5] (0,0) rectangle (1,1);
\end{tikzpicture}$ with the only restriction that two decorations that cross must have distinct colors. More precisely, $A_2$ is the alphabet shown in~\Cref{fig:alfabetito_2}. Note that we have removed the black triangle markings on green borders as they will no longer be needed.

\begin{figure}[ht!]
    \centering
     \begin{tikzpicture}[scale= 0.6]

    \begin{scope}[ shift = {(15,0)}]

    \begin{scope}[ shift = {(0,3)}]
            \draw[thick] (-1,0) to (1,0);
            \draw[thick] (0,-1) to (0,0);
        \end{scope}
    
        \begin{scope}[ shift = {(0,0)}]
            \filldraw[green!20] (-1,0) -- (1,0) -- (1,0.3) -- (-1,0.3);
        \draw[] (-1,0.3) to (1,0.3);
            \draw[thick] (-1,0) to (1,0);
            \draw[thick] (0,-1) to (0,0);
        \end{scope}

        \begin{scope}[ shift = {(0,-2.5)}, rotate=0]
        \draw[thick] (0,-1) to (0,0);
    \filldraw[green!20] (-1,0) -- (1,0) -- (1,-0.3) -- (-1,-0.3);
        \draw[thick] (-1,-0.3) to (1,-0.3);
            \draw[thick] (-1,0) to (1,0);
    \end{scope}

    \begin{scope}[ shift = {(0,-5)}]
    \filldraw[pattern=crosshatch] (-1,0) -- (1,0) -- (1,0.3) -- (-1,0.3);
            \filldraw[color=red, opacity=0.5] (-1,0) -- (1,0) -- (1,0.3) -- (-1,0.3);
        \filldraw[black] (0.2,0.3) -- (-0.2,0.3) -- (0,0);
        \draw[] (-1,0.3) to (1,0.3);
            \draw[thick] (-1,0) to (1,0);
            \draw[thick] (0,-1) to (0,0);
        \end{scope}

        \begin{scope}[ shift = {(0,-7.5)}, rotate=0]
    \filldraw[pattern=crosshatch] (-1,0) -- (1,0) -- (1,-0.3) -- (-1,-0.3);
    \filldraw[color=red, opacity=0.5] (-1,0) -- (1,0) -- (1,-0.3) -- (-1,-0.3);
        \filldraw[black] (0.2,-0.3) -- (-0.2,-0.3) -- (0,0);
        \draw[] (-1,-0.3) to (1,-0.3);
            \draw[thick] (-1,0) to (1,0);
            \draw[thick] (0,-1) to (0,0);
    \end{scope}
    
    \end{scope}

    \begin{scope}[ shift = {(4.5,3)}]
        \draw[thick] (-1,0) to (1,0);
        \draw[thick] (0,-1) to (0,1);
    \end{scope}

    \begin{scope}[ shift = {(0,0)}]
        \draw[thick] (0,-1) to (0,1);
        \filldraw[green!20] (-1,0) -- (1,0) -- (1,0.3) -- (-1,0.3);
        \draw[thick] (-1,0) to (1,0);
        \draw[] (-1,0.3) to (1,0.3);
    \end{scope}

    \begin{scope}[ shift = {(3,0)}, rotate=90]
    \draw[thick] (0,-1) to (0,1);
    \filldraw[green!20] (-1,0) -- (1,0) -- (1,0.3) -- (-1,0.3);
        \draw[thick] (-1,0) to (1,0);
        \draw[] (-1,0.3) to (1,0.3);
    \end{scope}

    \begin{scope}[ shift = {(6,0)}, rotate=180]
    \draw[thick] (0,-1) to (0,1);
    \filldraw[green!20] (-1,0) -- (1,0) -- (1,0.3) -- (-1,0.3);
        \draw[thick] (-1,0) to (1,0);
        
        \draw[] (-1,0.3) to (1,0.3);
    \end{scope}

    \begin{scope}[ shift = {(9,0)}, rotate=270]
    \draw[thick] (0,-1) to (0,1);
    \filldraw[green!20] (-1,0) -- (1,0) -- (1,0.3) -- (-1,0.3);
    \draw[thick] (-1,0) to (1,0);
        \draw[thick] (-1,0.3) to (1,0.3);
    \end{scope}

    \begin{scope}[ shift = {(0,-2.5)}]
    \filldraw[green!20] (-1,0) -- (0,0) -- (0,1) -- (-0.3,1) -- (-0.3,0.3) -- (-1,0.3);
        \draw[thick] (-1,0) to (1,0);
        \draw[thick] (0,-1) to (0,1);
        \draw[] (-1,0.3) -- (-0.3,0.3) -- (-0.3,1);
    \end{scope}

    \begin{scope}[ shift = {(3,-2.5)}, rotate=90]
    \filldraw[green!20] (-1,0) -- (0,0) -- (0,1) -- (-0.3,1) -- (-0.3,0.3) -- (-1,0.3);
        \draw[thick] (-1,0) to (1,0);
        \draw[thick] (0,-1) to (0,1);
        \draw[] (-1,0.3) -- (-0.3,0.3) -- (-0.3,1);
    \end{scope}

    \begin{scope}[ shift = {(6,-2.5)}, rotate=180]
    \filldraw[green!20] (-1,0) -- (0,0) -- (0,1) -- (-0.3,1) -- (-0.3,0.3) -- (-1,0.3);
        \draw[thick] (-1,0) to (1,0);
        \draw[thick] (0,-1) to (0,1);
        \draw[] (-1,0.3) -- (-0.3,0.3) -- (-0.3,1);
    \end{scope}

    \begin{scope}[ shift = {(9,-2.5)}, rotate=270]
    \filldraw[green!20] (-1,0) -- (0,0) -- (0,1) -- (-0.3,1) -- (-0.3,0.3) -- (-1,0.3);
        \draw[thick] (-1,0) to (1,0);
        \draw[thick] (0,-1) to (0,1);
        \draw[] (-1,0.3) -- (-0.3,0.3) -- (-0.3,1);
    \end{scope}

\begin{scope}[ shift = {(0,-5)}]
       \begin{scope}[ shift = {(0,0)}]
        \filldraw[pattern=crosshatch](-1,0) -- (1,0) -- (1,0.3) -- (-1,0.3);
        \filldraw[color=red, opacity=0.5](-1,0) -- (1,0) -- (1,0.3) -- (-1,0.3);
    \filldraw[black] (0.2,0.3) -- (-0.2,0.3) -- (0,0);
    \draw[] (-1,0.3) to (1,0.3);
        \draw[thick] (-1,0) to (1,0);
        \draw[thick] (0,-1) to (0,1);
    \end{scope}

    \begin{scope}[ shift = {(3,0)}, rotate=90]
    \filldraw[pattern=crosshatch](-1,0) -- (1,0) -- (1,0.3) -- (-1,0.3);
    \filldraw[color=red, opacity=0.5](-1,0) -- (1,0) -- (1,0.3) -- (-1,0.3);
    \filldraw[black] (0.2,0.3) -- (-0.2,0.3) -- (0,0);
        \draw[thick] (-1,0) to (1,0);
        \draw[thick] (0,-1) to (0,1);
        \draw[] (-1,0.3) to (1,0.3);
    \end{scope}

    \begin{scope}[ shift = {(6,0)}, rotate=180]
    \filldraw[pattern=crosshatch] (-1,0) -- (1,0) -- (1,0.3) -- (-1,0.3);
    \filldraw[color=red, opacity=0.5] (-1,0) -- (1,0) -- (1,0.3) -- (-1,0.3);
    \filldraw[black] (0.2,0.3) -- (-0.2,0.3) -- (0,0);
        \draw[thick] (-1,0) to (1,0);
        \draw[thick] (0,-1) to (0,1);
        \draw[] (-1,0.3) to (1,0.3);
    \end{scope}

    \begin{scope}[ shift = {(9,0)}, rotate=270]
    \filldraw[pattern=crosshatch] (-1,0) -- (1,0) -- (1,0.3) -- (-1,0.3);
    \filldraw[color=red, opacity=0.5] (-1,0) -- (1,0) -- (1,0.3) -- (-1,0.3);
    \filldraw[black] (0.2,0.3) -- (-0.2,0.3) -- (0,0);
        \draw[thick] (-1,0) to (1,0);
        \draw[thick] (0,-1) to (0,1);
        \draw[] (-1,0.3) to (1,0.3);
    \end{scope}

    \begin{scope}[ shift = {(0,-2.5)}]
    \filldraw[pattern=crosshatch] (-1,0) -- (0,0) -- (0,1) -- (-0.3,1) -- (-0.3,0.3) -- (-1,0.3);
    \filldraw[color=red, opacity=0.5] (-1,0) -- (0,0) -- (0,1) -- (-0.3,1) -- (-0.3,0.3) -- (-1,0.3);
    \filldraw[black] (0,0) -- (-0.3,0) -- (0,0.3);
        \draw[thick] (-1,0) to (1,0);
        \draw[thick] (0,-1) to (0,1);
        \draw[] (-1,0.3) -- (-0.3,0.3) -- (-0.3,1);
    \end{scope}

    \begin{scope}[ shift = {(3,-2.5)}, rotate=90]
    \filldraw[pattern=crosshatch] (-1,0) -- (0,0) -- (0,1) -- (-0.3,1) -- (-0.3,0.3) -- (-1,0.3);
    \filldraw[color=red, opacity=0.5] (-1,0) -- (0,0) -- (0,1) -- (-0.3,1) -- (-0.3,0.3) -- (-1,0.3);
    \filldraw[black] (0,0) -- (-0.3,0) -- (0,0.3);
        \draw[thick] (-1,0) to (1,0);
        \draw[thick] (0,-1) to (0,1);
        \draw[] (-1,0.3) -- (-0.3,0.3) -- (-0.3,1);
    \end{scope}

    \begin{scope}[ shift = {(6,-2.5)}, rotate=180]
    \filldraw[pattern=crosshatch] (-1,0) -- (0,0) -- (0,1) -- (-0.3,1) -- (-0.3,0.3) -- (-1,0.3);
    \filldraw[color=red, opacity=0.5] (-1,0) -- (0,0) -- (0,1) -- (-0.3,1) -- (-0.3,0.3) -- (-1,0.3);
    \filldraw[black] (0,0) -- (-0.3,0) -- (0,0.3);
        \draw[thick] (-1,0) to (1,0);
        \draw[thick] (0,-1) to (0,1);
        \draw[] (-1,0.3) -- (-0.3,0.3) -- (-0.3,1);
    \end{scope}

    \begin{scope}[ shift = {(9,-2.5)}, rotate=270]
    \filldraw[pattern=crosshatch] (-1,0) -- (0,0) -- (0,1) -- (-0.3,1) -- (-0.3,0.3) -- (-1,0.3);
    \filldraw[color=red, opacity=0.5] (-1,0) -- (0,0) -- (0,1) -- (-0.3,1) -- (-0.3,0.3) -- (-1,0.3);
    \filldraw[black] (0,0) -- (-0.3,0) -- (0,0.3);
        \draw[thick] (-1,0) to (1,0);
        \draw[thick] (0,-1) to (0,1);
        \draw[] (-1,0.3) -- (-0.3,0.3) -- (-0.3,1);
    \end{scope}

\end{scope}

    \begin{scope}[ shift = {(0,-12.5)}]
        
     \begin{scope}[rotate=90]
        \filldraw[green!20] (-1,0) -- (1,0) -- (1,0.3) -- (-1,0.3);
    \end{scope}
    \filldraw[pattern=crosshatch] (-1,0) -- (1,0) -- (1,0.3) -- (-1,0.3);
        \filldraw[color=red, opacity=0.5] (-1,0) -- (1,0) -- (1,0.3) -- (-1,0.3);
        \begin{scope}[rotate=90]
    \end{scope}
    \filldraw[black] (0.2,0.3) -- (-0.2,0.3) -- (0,0.1);
    \draw[] (-1,0.3) to (-0.3,0.3) -- (-0.3,1);
    \draw[] (-0.3,-1) -- (-0.3,1);
    \draw[] (-1,0.3) to (1,0.3);
    \draw[thick] (-1,0) to (1,0);
    \draw[thick] (0,-1) to (0,1);
        
    \end{scope}

     \begin{scope}[ shift = {(3,-12.5)}, rotate = 90]
        \filldraw[green!20] (-1,0) -- (1,0) -- (1,0.3) -- (-1,0.3);
    
     \begin{scope}[rotate=90]
        \filldraw[pattern=crosshatch] (-1,0) -- (1,0) -- (1,0.3) -- (-1,0.3);
        \filldraw[color=red, opacity=0.5] (-1,0) -- (1,0) -- (1,0.3) -- (-1,0.3);
    \filldraw[black] (0.2,0.3) -- (-0.2,0.3) -- (0,0.1);
    \end{scope}
    \draw[] (-1,0.3) to (-0.3,0.3) -- (-0.3,1);
    \draw[] (-0.3,-1) -- (-0.3,1);
    \draw[] (-1,0.3) to (1,0.3);
    \draw[thick] (-1,0) to (1,0);
    \draw[thick] (0,-1) to (0,1);
        
    \end{scope}

    \begin{scope}[ shift = {(6,-12.5)}, rotate = 180]
   \begin{scope}[rotate=90]
        \filldraw[green!20] (-1,0) -- (1,0) -- (1,0.3) -- (-1,0.3);
    \end{scope}
        \filldraw[pattern=crosshatch] (-1,0) -- (1,0) -- (1,0.3) -- (-1,0.3);
        \filldraw[color=red, opacity=0.5] (-1,0) -- (1,0) -- (1,0.3) -- (-1,0.3);
     \begin{scope}[rotate=90]
    \end{scope}
    \filldraw[black] (0.2,0.3) -- (-0.2,0.3) -- (0,0.1);
    \draw[] (-1,0.3) to (-0.3,0.3) -- (-0.3,1);
    \draw[] (-0.3,-1) -- (-0.3,1);
    \draw[] (-1,0.3) to (1,0.3);
    \draw[thick] (-1,0) to (1,0);
    \draw[thick] (0,-1) to (0,1);
        
    \end{scope}

    \begin{scope}[ shift = {(9,-12.5)}, rotate = -90]
        \filldraw[green!20] (-1,0) -- (1,0) -- (1,0.3) -- (-1,0.3);
    
     \begin{scope}[rotate=90]
        \filldraw[pattern=crosshatch] (-1,0) -- (1,0) -- (1,0.3) -- (-1,0.3);
        \filldraw[color=red, opacity=0.5] (-1,0) -- (1,0) -- (1,0.3) -- (-1,0.3);
    \filldraw[black] (0.2,0.3) -- (-0.2,0.3) -- (0,0.1);
    \end{scope}
    \draw[] (-1,0.3) to (-0.3,0.3) -- (-0.3,1);
    \draw[] (-0.3,-1) -- (-0.3,1);
    \draw[] (-1,0.3) to (1,0.3);
    \draw[thick] (-1,0) to (1,0);
    \draw[thick] (0,-1) to (0,1);
        
    \end{scope}

    \begin{scope}[shift = {(0,2.5)}]
        \begin{scope}[ shift = {(0,-12.5)}]
        \filldraw[green!20] (-1,0) -- (1,0) -- (1,0.3) -- (-1,0.3);
         \begin{scope}[rotate=90]
            \filldraw[pattern=crosshatch] (-1,0) -- (1,0) -- (1,0.3) -- (-1,0.3);
            \filldraw[color=red, opacity=0.5] (-1,0) -- (1,0) -- (1,0.3) -- (-1,0.3);
        \end{scope}
        \begin{scope}[rotate=90]
        \filldraw[black] (0.2,0.3) -- (-0.2,0.3) -- (0,0.1);
        \end{scope}
        \draw[] (-1,0.3) to (-0.3,0.3) -- (-0.3,1);
        \draw[] (-0.3,-1) -- (-0.3,1);
        \draw[] (-1,0.3) to (1,0.3);
        \draw[thick] (-1,0) to (1,0);
        \draw[thick] (0,-1) to (0,1);   
        \end{scope}

     \begin{scope}[ shift = {(3,-12.5)}, rotate = 90]
     \begin{scope}[rotate=90]
        \filldraw[green!20] (-1,0) -- (1,0) -- (1,0.3) -- (-1,0.3);
    \end{scope}
        \filldraw[pattern=crosshatch] (-1,0) -- (1,0) -- (1,0.3) -- (-1,0.3);
        \filldraw[color=red, opacity=0.5] (-1,0) -- (1,0) -- (1,0.3) -- (-1,0.3);
     \begin{scope}[rotate=90]
    \end{scope}
    \filldraw[black] (0.2,0.3) -- (-0.2,0.3) -- (0,0.1);
    \draw[] (-1,0.3) to (-0.3,0.3) -- (-0.3,1);
    \draw[] (-0.3,-1) -- (-0.3,1);
    \draw[] (-1,0.3) to (1,0.3);
    \draw[thick] (-1,0) to (1,0);
    \draw[thick] (0,-1) to (0,1);
        
    \end{scope}

    \begin{scope}[ shift = {(6,-12.5)}, rotate = 180]
    \filldraw[green!20] (-1,0) -- (1,0) -- (1,0.3) -- (-1,0.3);
   \begin{scope}[rotate=90]
        \filldraw[pattern=crosshatch] (-1,0) -- (1,0) -- (1,0.3) -- (-1,0.3);
        \filldraw[color=red, opacity=0.5] (-1,0) -- (1,0) -- (1,0.3) -- (-1,0.3);
    \end{scope}
        
     \begin{scope}[rotate=90]
    \filldraw[black] (0.2,0.3) -- (-0.2,0.3) -- (0,0.1);
    \end{scope}
    \draw[] (-1,0.3) to (-0.3,0.3) -- (-0.3,1);
    \draw[] (-0.3,-1) -- (-0.3,1);
    \draw[] (-1,0.3) to (1,0.3);
    \draw[thick] (-1,0) to (1,0);
    \draw[thick] (0,-1) to (0,1);
        
    \end{scope}

    \begin{scope}[ shift = {(9,-12.5)}, rotate = -90]
    \begin{scope}[rotate=90]
        \filldraw[green!20] (-1,0) -- (1,0) -- (1,0.3) -- (-1,0.3);
    \end{scope}
        \filldraw[pattern=crosshatch] (-1,0) -- (1,0) -- (1,0.3) -- (-1,0.3);
        \filldraw[color=red, opacity=0.5] (-1,0) -- (1,0) -- (1,0.3) -- (-1,0.3);
    \begin{scope}[rotate=90]
    \end{scope}
    \filldraw[black] (0.2,0.3) -- (-0.2,0.3) -- (0,0.1);
    \draw[] (-1,0.3) to (-0.3,0.3) -- (-0.3,1);
    \draw[] (-0.3,-1) -- (-0.3,1);
    \draw[] (-1,0.3) to (1,0.3);
    \draw[thick] (-1,0) to (1,0);
    \draw[thick] (0,-1) to (0,1);
        
    \end{scope}
    \end{scope}
\end{tikzpicture} 
    \caption{The alphabet $A_2$.}
    \label{fig:alfabetito_2}
\end{figure}
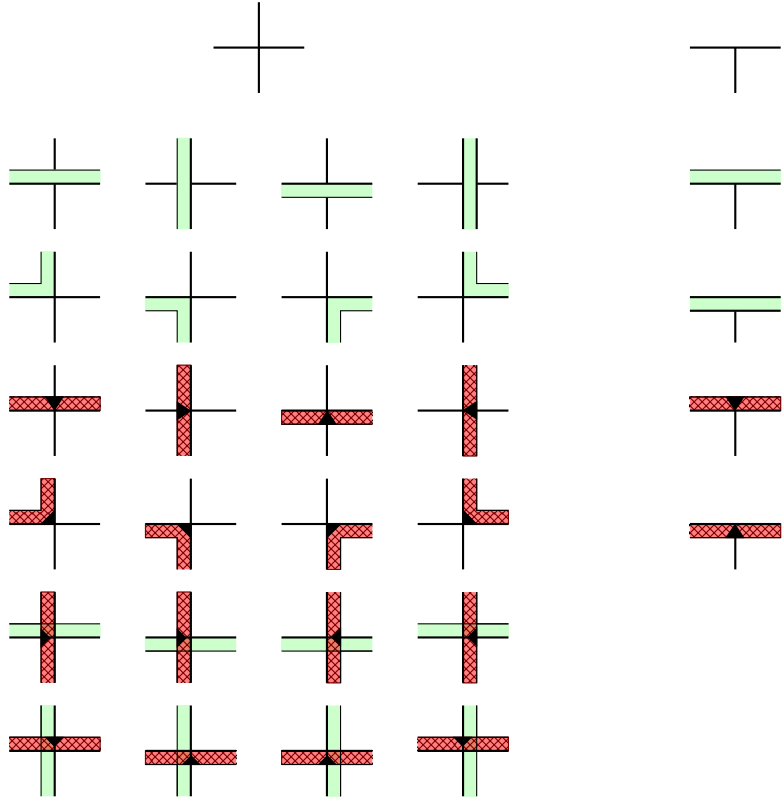

The local rules we impose in this layer are the following: first, our decorated colored symbols must match with the corresponding colorless symbol from layer 1 (thus we may think that in this layer we just added the colors). Second, we ask that 
adjacent decorations carry the same color (thus each rectangle has a unique color), and finally, that $2$-blocks must all carry the color green.

As we have forced that every intersection between borders creates a color change, and since 2-borders are required to be green, it follows that the color of a $2^n$-border is red whenever $n$ is even, and green whenever $n$ is odd.

From this point forwards we will only consider the $4^n$-borders, which are thus all colored red. We remark that in this manner red borders of different size do not intersect.

\subsection*{Layer 3 (Traveling bits layer \texorpdfstring{$\mathfrak{X}_3$}{X3})} In this layer, we shall assign to every horizontal and vertical edge a bit in $\{0,1\}$ in such a way that the bit gets transmitted horizontally and vertically. Formally, the alphabet $A_3$ is given by the decorations shown in~\Cref{fig:alfabetito_3}.

\begin{figure}[ht!]
    \centering
    \begin{tikzpicture}
    \begin{scope}[shift = {(0,0)}]
        \draw [thick] (0,0) -- (-1,0) node[fill=white, midway] {$\mathtt{1}$};
        \draw [thick] (0,0) -- (1,0) node[fill=white,midway] {$\mathtt{1}$};
        \draw [thick] (0,0) -- (0,-1) node[fill=white,midway] {$\mathtt{1}$};
        \draw [thick] (0,0) -- (0,1) node[fill=white,midway] {$\mathtt{1}$};
    \end{scope}
    \begin{scope}[shift = {(0,3)}]
        \draw [thick] (0,0) -- (-1,0) node[fill=white, midway] {$\mathtt{1}$};
        \draw [thick] (0,0) -- (1,0) node[fill=white,midway] {$\mathtt{1}$};
        \draw [thick] (0,0) -- (0,-1) node[fill=white,midway] {$\mathtt{0}$};
        \draw [thick] (0,0) -- (0,1) node[fill=white,midway] {$\mathtt{0}$};
    \end{scope}
    \begin{scope}[shift = {(-3,0)}]
        \draw [thick] (0,0) -- (-1,0) node[fill=white, midway] {$\mathtt{0}$};
        \draw [thick] (0,0) -- (1,0) node[fill=white,midway] {$\mathtt{0}$};
        \draw [thick] (0,0) -- (0,-1) node[fill=white,midway] {$\mathtt{1}$};
        \draw [thick] (0,0) -- (0,1) node[fill=white,midway] {$\mathtt{1}$};
    \end{scope}
    \begin{scope}[shift = {(-3,3)}]
        \draw [thick] (0,0) -- (-1,0) node[fill=white, midway] {$\mathtt{0}$};
        \draw [thick] (0,0) -- (1,0) node[fill=white,midway] {$\mathtt{0}$};
        \draw [thick] (0,0) -- (0,-1) node[fill=white,midway] {$\mathtt{0}$};
        \draw [thick] (0,0) -- (0,1) node[fill=white,midway] {$\mathtt{0}$};
    \end{scope}

    \begin{scope}[shift = {(6,0)}]
        \begin{scope}[shift = {(0,0)}]
        \draw [thick] (0,0) -- (-1,0) node[fill=white, midway] {$\mathtt{1}$};
        \draw [thick] (0,0) -- (1,0) node[fill=white,midway] {$\mathtt{1}$};
        \draw [thick] (0,0) -- (0,-1) node[fill=white,midway] {$\mathtt{1}$};
    \end{scope}
    \begin{scope}[shift = {(0,3)}]
        \draw [thick] (0,0) -- (-1,0) node[fill=white, midway] {$\mathtt{1}$};
        \draw [thick] (0,0) -- (1,0) node[fill=white,midway] {$\mathtt{1}$};
        \draw [thick] (0,0) -- (0,-1) node[fill=white,midway] {$\mathtt{0}$};
    \end{scope}
    \begin{scope}[shift = {(-3,0)}]
        \draw [thick] (0,0) -- (-1,0) node[fill=white, midway] {$\mathtt{0}$};
        \draw [thick] (0,0) -- (1,0) node[fill=white,midway] {$\mathtt{0}$};
        \draw [thick] (0,0) -- (0,-1) node[fill=white,midway] {$\mathtt{1}$};
    \end{scope}
    \begin{scope}[shift = {(-3,3)}]
        \draw [thick] (0,0) -- (-1,0) node[fill=white, midway] {$\mathtt{0}$};
        \draw [thick] (0,0) -- (1,0) node[fill=white,midway] {$\mathtt{0}$};
        \draw [thick] (0,0) -- (0,-1) node[fill=white,midway] {$\mathtt{0}$};
    \end{scope}
    \end{scope}
    
\end{tikzpicture}
    \caption{The alphabet $A_3$.}
    \label{fig:alfabetito_3}
\end{figure}
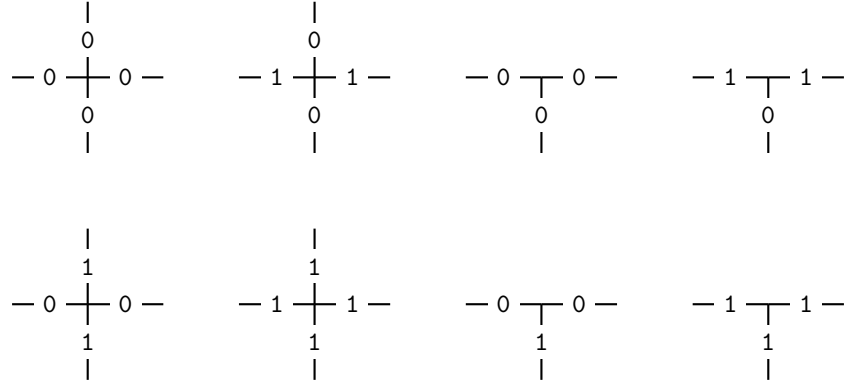

We impose the local rule that two adjacent vertices must have the same bit along their common edge. This ensures that every row and column has a single associated bit. Finally, we impose the local rule that if in layer 1 we are on top of a cross of type DL, DR, UL or UR (a corner) associated to a red $\begin{tikzpicture}[scale = 0.2]
    \draw[thick] (0,0) rectangle (1,1);
    \filldraw[pattern=crosshatch] (0,0) rectangle (1,1);
    \filldraw[color=red, opacity=0.5] (0,0) rectangle (1,1);
\end{tikzpicture}$ border in layer 2, then the vertical and horizontal bits must coincide. This ensures that every $4^n$-border has a unique bit associated to it. 

\begin{proposition}\label{prop:border-bit-constancy}
All $4^n$-borders in a configuration carry the same bit ($n\geq 1$).
\end{proposition}
\begin{proof}
    We start with the observation that the bit carried by a $4^n$-border is the same as the bit carried by the red $4^n$-border to its right. Indeed, this is because the upper sides of both borders belong to the same infinite row. Hence in a given infinite row of red $4^n$-borders, all of them carry the same bit. 
    
    Let us now observe that this bit is also transmitted to the infinite row of red $4^n$-borders immediately below. Indeed, by~\Cref{borders-are-nicely-aligned}, if we follow the vertical line of edges that starts at a lower corner of a red $4^n$-border, then we eventually reach the upper corner of a red $4^n$-border, which belongs to the row of $4^n$-borders directly below.
\end{proof}
An immediate consequence of~\Cref{prop:border-bit-constancy} is that we can define $\phi\colon \mathfrak{X}_3 \to\{0,1\}^\NN$ where $\phi(\varphi,x)=y$ is given by

\begin{equation}\label{eq:definition-y}
y(n) =  \text{bit over red $4^n$-borders in $x$.}
\end{equation}
One easily verifies that with the rules imposed so far, $\phi$ is a computable map. Furthermore, every $y\in\{0,1\}^\NN$ can be obtained in this manner. 
\begin{proposition}\label{prop:phi-surjective}
    The map $\phi$ is surjective. 
\end{proposition}
\begin{proof}

For fixed $n\geq 2$, pick an infinite row that intersects the corners of red $4^n$-borders. We claim that for $k < n$ it cannot intersect the corners of red $4^k$-borders. To see this, first recall that a red $4^n$-border lies within a $2^{2n+1}$-block. Decomposing our $2^{2n+1}$-blocks into $2^{2n}$-blocks and then into $2^{2n-1}$-blocks, we see that our infinite row travels precisely along the boundaries of the $2^{2n-1}$-blocks. Clearly it cannot intersect the corners of the red $2^{2n-2}$-borders. Thus our claim holds for $k=n-1$. For $k<n-1$, it suffices to observe that every $4^k$-border is contained in a $2^{2n-1}$-block. 

The same claim in the previous paragraph can be applied to infinite columns that intersect the corners of red $4^n$-borders.

The last two claims show the following: given an element in $\mathfrak X_3$ and a fixed $n$, it is possible to change the bit in the red $4^n$-borders, while preserving the bit in the red $4^k$-borders for $k<n$, and obtain another element in $\mathfrak X_3$. By compactness, this implies that $\phi$ is surjective. 
\end{proof}
Over the next two layers we will construct the necessary structure to be able to discard those $y \in \{0,1\}^\NN$ that do not belong to $P$ by simulating a Turing machine in the free space of the red borders.

\subsection*{Layer 4 (Obstruction signal layer \texorpdfstring{$\mathfrak{X}_4$}{X4})} In this layer we identify through local rules the free rows and columns inside borders (see~\Cref{def:free}). For this we shall assign to every edge one of various obstruction signals ($\times$, $\blacktriangleright$, $\blacktriangleleft$, $\blacktriangledown$, $\blacktriangle$) or a free symbol $\scalebox{0.8}{\begin{tikzpicture}
    \draw [] (0,0) circle (0.2);
    \node at (0,0) {$\checkmark$};
\end{tikzpicture}}$ which indicates a free row or column. The intuitive idea is that every vertex lying on a $4^n$-border will emit an obstruction signal to the outside that travels until it hits another red border. The remaining edges will be precisely the free rows and columns.

We remark that this is very similar to the way that grids are constructed on top of the Robinson tiling to prove the undecidability of the domino problem~\cite{robinson_undecidability_1971}; however, in our case we also need special rules for the bottom edge at vertices of degree 3.

Let us now proceed formally. The alphabet $A_4$ is formed by all possible combinations where vertical edges and horizontal edges are decorated by \define{signals} as in~\Cref{fig:alfabetito_4}. A few examples are shown in~\Cref{fig:alfabetito_4ex}.

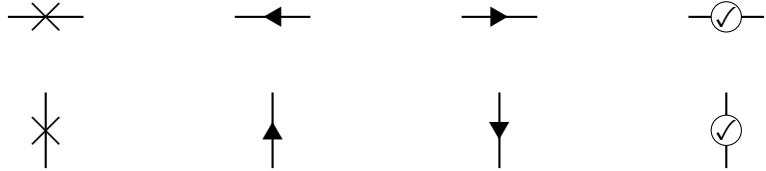
\begin{figure}[ht!]
    \centering
    \begin{tikzpicture}
    \begin{scope}[shift = {(0,0)}]
        \draw [thick] (0,0) -- (1,0);
        \node at  (0.5,0) {$\bigtimes$};
    \end{scope}
    \begin{scope}[shift = {(3,0)}]
        \draw [thick] (0,0) -- (1,0) node[midway] {$\blacktriangleleft$};
    \end{scope}
    \begin{scope}[shift = {(6,0)}]
        \draw [thick] (0,0) -- (1,0) node[midway] {$\blacktriangleright$};
    \end{scope}
    \begin{scope}[shift = {(9,0)}]
        \draw [thick] (0,0) -- (1,0);
        \draw[fill=white]  (0.5,0) circle (0.2);
        \node at (0.5,0) {$\mathtt{\checkmark}$};
    \end{scope}

    \begin{scope}[shift = {(0.5,-2)}]
        \draw [thick] (0,0) -- (0,1);
        \node at  (0,0.5) {$\bigtimes$};
    \end{scope}
    \begin{scope}[shift = {(3.5,-2)}]
        \draw [thick] (0,0) -- (0,1) node[midway] {$\blacktriangle$};
    \end{scope}
    \begin{scope}[shift = {(6.5,-2)}]
        \draw [thick] (0,0) -- (0,1) node[midway] {$\blacktriangledown$};
    \end{scope}
    \begin{scope}[shift = {(9.5,-2)}]
        \draw [thick] (0,0) -- (0,1);
        \draw[fill=white]  (0,0.5) circle (0.2);
        \node at (0,0.5) {$\mathtt{\checkmark}$};
    \end{scope}

\end{tikzpicture}
    \caption{The alphabet $A_4$ is given by all possible combinations of the decorations above. }
    \label{fig:alfabetito_4}
\end{figure}

\begin{figure}[ht!]
    \centering
    \begin{tikzpicture}
    \begin{scope}[shift = {(0,0)}]
        \draw [thick] (0,0) -- (0,1);
        \draw[fill=white]  (0,0.5) circle (0.2);
        \node at (0,0.5) {$\mathtt{\checkmark}$};
        \draw [thick] (0,0) -- (0,-1);
        \draw[fill=white]  (0,-0.5) circle (0.2);
        \node at (0,-0.5) {$\mathtt{\checkmark}$};
        \draw [thick] (0,0) -- (1,0) node[midway] {$\blacktriangleright$};
        \draw [thick] (0,0) -- (-1,0) node[midway] {$\blacktriangleright$};
    \end{scope}
    \begin{scope}[shift = {(3,0)}]
       \draw [thick] (0,0) -- (1,0);
        \node at  (0.5,0) {$\bigtimes$};
        \draw [thick] (0,0) -- (-1,0);
        \node at  (-0.5,0) {$\bigtimes$};
        \draw [thick] (0,0) -- (0,-1);
        \draw[fill=white]  (0,-0.5) circle (0.2);
        \node at (0,-0.5) {$\mathtt{\checkmark}$};
        
    \end{scope}
    \begin{scope}[shift = {(6,0)}]
        \draw [thick] (0,0) -- (-1,0) node[midway] {$\blacktriangleleft$};
        \draw [thick] (0,0) -- (1,0) node[midway] {$\blacktriangleleft$};
        \draw [thick] (0,0) -- (0,-1) node[midway] {$\blacktriangle$};
        \draw [thick] (0,0) -- (0,1) node[midway] {$\blacktriangle$};
        
    \end{scope}
    \begin{scope}[shift = {(9,0)}]
        \draw [thick] (0,0) -- (-1,0) node[midway] {$\blacktriangleright$};
        \draw [thick] (0,0) -- (1,0) node[midway] {$\blacktriangleright$};
        \draw [thick] (0,0) -- (0,-1) node[midway] {$\blacktriangle$};
        
    \end{scope}
    
\end{tikzpicture}
    \caption{A few examples of symbols in $A_4$.}
    \label{fig:alfabetito_4ex}
\end{figure}
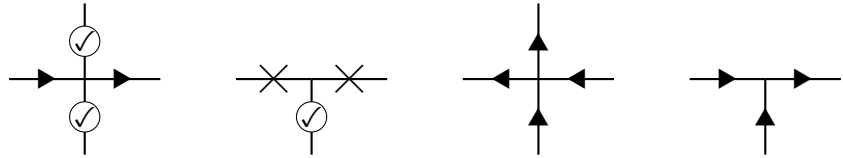

For a vertex which carries a red border in Layer 2, we say an edge is \define{outside} if it is being pointed at by the black triangle marking (see~\Cref{fig:alfabetito_2}) on the red border. In the case of corners, both edges which do not form part of the border are considered outside. Conversely, an edge which is not part of a red border and it not outside is said to be \define{inside}. 

As is usual, we ask that the common edge between two adjacent vertices has the same decoration. Furthermore, we shall impose the following set of local rules that define $\mathfrak{X}_4$.

\begin{itemize}
    \item An edge is marked by $\times$ if and only if it is associated to a red border in Layer 2. 
    \item The only admissible signals on an edge which is outside of a red border are arrows.
    \item The only admissible signal on an edge which is inside of a border is an incoming arrow (pointing towards the border) or the free symbol $\scalebox{0.8}{\begin{tikzpicture}
    \draw [] (0,0) circle (0.2);
    \node at (0,0) {$\checkmark$};
\end{tikzpicture}}$.
\item If there are no red border markings, the signals on the left and right edges must coincide. Similarly, if the vertex has degree 4, then the signals on the bottom and top edges must coincide.
\item If there are no red border markings and the vertex has degree 3, then the signal on the bottom edge cannot be the free symbol $\scalebox{0.8}{\begin{tikzpicture}
    \draw [] (0,0) circle (0.2);
    \node at (0,0) {$\checkmark$};
\end{tikzpicture}}$.
\end{itemize}

A few valid placements of signals are shown in~\Cref{fig:valid-signals}. Note that edges on red borders are always marked by crosses. The intuition is that outside edges can only emit or absorb arrows (thus the corresponding row/column cannot be free), while the inside edges can only absorb signals or be free. In this setting, the bottom edge of a vertex of degree 3 which is not incident to a border is thought of as an outside edge (and thus cannot be free). We denote the subshift obtained by adding this layer by $\mathfrak{X}_4$.

\begin{figure}[ht!]
    \centering
    \begin{tikzpicture}

\begin{scope}[ shift = {(0,0)}, rotate=270, opacity = 0.5]
    \filldraw[pattern=crosshatch] (-1,0) -- (0,0) -- (0,1) -- (-0.3,1) -- (-0.3,0.3) -- (-1,0.3);
    \filldraw[color=red, opacity=0.5] (-1,0) -- (0,0) -- (0,1) -- (-0.3,1) -- (-0.3,0.3) -- (-1,0.3);
    \filldraw[black] (0,0) -- (-0.3,0) -- (0,0.3);
        \draw[thick] (-1,0) to (1,0);
        \draw[thick] (0,-1) to (0,1);
        \draw[] (-1,0.3) -- (-0.3,0.3) -- (-0.3,1);
    \end{scope}
    \begin{scope}[shift = {(0,0)}]
    \draw[thick] (0,-1) to (0,1);
        \node at (0,0.5) {$\bigtimes$};
        \node at (0,-0.5) {$\blacktriangledown$};
        \draw [thick] (0,0) -- (1,0) node[midway] {$\bigtimes$};
        \draw [thick] (0,0) -- (-1,0) node[midway] {$\blacktriangleleft$};
    \end{scope}

    \begin{scope}[ shift = {(3,0)}, opacity = 0.5]
    \filldraw[pattern=crosshatch] (-1,0) -- (1,0) -- (1,0.3) -- (-1,0.3);
            \filldraw[color=red, opacity=0.5] (-1,0) -- (1,0) -- (1,0.3) -- (-1,0.3);
        \filldraw[black] (0.2,0.3) -- (-0.2,0.3) -- (0,0);
        \draw[] (-1,0.3) to (1,0.3);
            \draw[thick] (-1,0) to (1,0);
            \draw[thick] (0,-1) to (0,0);
        \end{scope}
    
    \begin{scope}[shift = {(3,0)}]
        \draw[thick] (-1,0) to (1,0);
        \draw[thick] (0,0) to (0,-1);
        \node at  (0.5,0) {$\bigtimes$};
        \node at  (-0.5,0) {$\bigtimes$};
        \node at (0,-0.5) {$\blacktriangledown$};
        
    \end{scope}

\begin{scope}[ shift = {(6,0)}, opacity = 0.5]
    \filldraw[pattern=crosshatch] (-1,0) -- (1,0) -- (1,-0.3) -- (-1,-0.3);
    \filldraw[color=red, opacity=0.5] (-1,0) -- (1,0) -- (1,-0.3) -- (-1,-0.3);
        \filldraw[black] (0.2,-0.3) -- (-0.2,-0.3) -- (0,0);
        \draw[] (-1,-0.3) to (1,-0.3);
            \draw[thick] (-1,0) to (1,0);
            \draw[thick] (0,-1) to (0,0);
    \end{scope}
    
    \begin{scope}[shift = {(6,0)}]
        \draw [thick] (0,0) -- (-1,0) node[midway] {$\bigtimes$};
        \draw [thick] (0,0) -- (1,0) node[midway] {$\bigtimes$};
        \draw[thick] (0,-1) to (0,0);
        \draw[fill=white]  (0,-0.5) circle (0.2);
        \node at (0,-0.5) {$\mathtt{\checkmark}$};
        
    \end{scope}

    \begin{scope}[ shift = {(9,0)}, rotate=90, opacity=0.5]
    \filldraw[pattern=crosshatch](-1,0) -- (1,0) -- (1,0.3) -- (-1,0.3);
    \filldraw[color=red, opacity=0.5](-1,0) -- (1,0) -- (1,0.3) -- (-1,0.3);
    \filldraw[black] (0.2,0.3) -- (-0.2,0.3) -- (0,0);
        \draw[thick] (-1,0) to (1,0);
        \draw[thick] (0,-1) to (0,1);
        \draw[] (-1,0.3) to (1,0.3);
    \end{scope}
    
    \begin{scope}[shift = {(9,0)}]
        \draw [thick] (0,0) -- (-1,0) node[midway] {$\blacktriangleright$};
        \draw [thick] (0,0) -- (1,0) node[midway] {$\blacktriangleleft$};
        \draw [thick] (0,0) -- (0,-1) node[midway] {$\bigtimes$};
        \draw [thick] (0,0) -- (0,1) node[midway] {$\bigtimes$};
        
    \end{scope}

    \begin{scope}[shift = {(0,-3)}]

    \begin{scope}[shift = {(0,0)}]
    \draw[thick] (0,-1) to (0,1);
        \node at (0,0.5) {$\blacktriangledown$};
        \node at (0,-0.5) {$\blacktriangledown$};
        \draw [thick] (0,0) -- (1,0) node[midway] {$\blacktriangleleft$};
        \draw [thick] (0,0) -- (-1,0) node[midway] {$\blacktriangleleft$};
    \end{scope}

    \begin{scope}[shift = {(3,0)}]
        \draw[thick] (-1,0) to (1,0);
        \draw[thick] (0,0) to (0,-1);
        \draw[fill=white]  (-0.5,0) circle (0.2);
        \draw[fill=white]  (0.5,0) circle (0.2);
        \node at  (0.5,0) {$\mathtt{\checkmark}$};
        \node at  (-0.5,0) {$\mathtt{\checkmark}$};
        \node at (0,-0.5) {$\blacktriangledown$};
        
    \end{scope}

    \begin{scope}[shift = {(6,0)}]
        \draw[thick] (-1,0) to (1,0);
        \draw[thick] (0,0) to (0,-1);
        \node at  (0.5,0) {$\blacktriangleright$};
        \node at  (-0.5,0) {$\blacktriangleright$};
        \node at (0,-0.5) {$\blacktriangle$};
        
    \end{scope}

    \begin{scope}[shift = {(9,0)}]
        \draw [thick] (0,0) -- (-1,0) node[midway] {$\blacktriangleright$};
        \draw [thick] (0,0) -- (1,0) node[midway] {$\blacktriangleright$};
        \draw [thick] (0,0) -- (0,-1);
        \draw [thick] (0,0) -- (0,1);
        \draw[fill=white]  (0,-0.5) circle (0.2);
        \draw[fill=white]  (0,0.5) circle (0.2);
        \node at (0,-0.5) {$\mathtt{\checkmark}$};
        \node at (0,0.5) {$\mathtt{\checkmark}$};

    \end{scope}
        
    \end{scope}
    
\end{tikzpicture}
    \caption{Examples of valid signal placements in $A_4$.}
    \label{fig:valid-signals}
\end{figure}
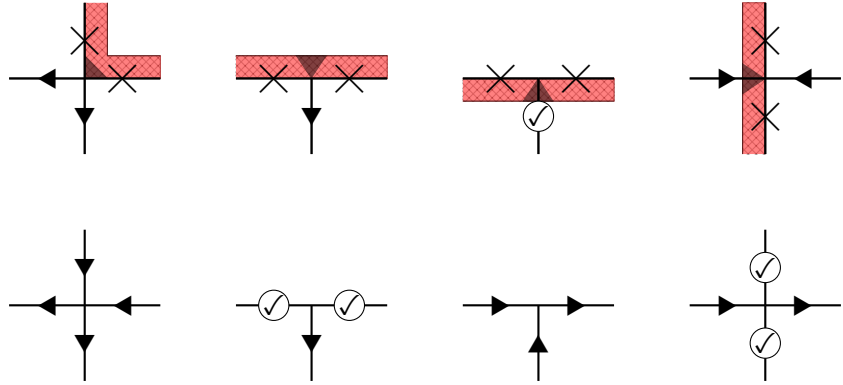

\begin{proposition}
    For every $4^n$-border in $\mathfrak{X}_4$ and every row (resp. column) within, the following are equivalent:
    \begin{enumerate}
        \item The row (resp. column) is free.
        \item The entire row (resp. column) is marked by the signal $\scalebox{0.8}{\begin{tikzpicture}
    \draw [] (0,0) circle (0.2);
    \node at (0,0) {$\checkmark$};
\end{tikzpicture}}$.
        \item The leftmost (resp. top) signal is $\scalebox{0.8}{\begin{tikzpicture}
    \draw [] (0,0) circle (0.2);
    \node at (0,0) {$\checkmark$};
\end{tikzpicture}}$.
        \item The rightmost (resp. bottom) signal is $\scalebox{0.8}{\begin{tikzpicture}
    \draw [] (0,0) circle (0.2);
    \node at (0,0) {$\checkmark$};
\end{tikzpicture}}$.  
    \end{enumerate}

\end{proposition}

\begin{proof}
    Consider a row of edges within a $4^n$-border. If this row is free, then it has a single signal associated to it. Note that on the leftmost side the only possible signals are $\blacktriangleleft$ and $\scalebox{0.8}{\begin{tikzpicture}
    \draw [] (0,0) circle (0.2);
    \node at (0,0) {$\checkmark$};
\end{tikzpicture}}$, while on the rightmost side the only possibilities are $\blacktriangleright$ and $\scalebox{0.8}{\begin{tikzpicture}
    \draw [] (0,0) circle (0.2);
    \node at (0,0) {$\checkmark$};
\end{tikzpicture}}$. Hence we conclude that the signal is necessarily $\scalebox{0.8}{\begin{tikzpicture}
    \draw [] (0,0) circle (0.2);
    \node at (0,0) {$\checkmark$};
\end{tikzpicture}}$. 

Conversely, suppose the leftmost signal is $\scalebox{0.8}{\begin{tikzpicture}
    \draw [] (0,0) circle (0.2);
    \node at (0,0) {$\checkmark$};
\end{tikzpicture}}$. If the row were not free, then the row must intersect some $4^k$-border for $k <n$. As red borders do not intersect, it follows that from left to right, the first intersection occurs on the outside of one of these borders, hence the signal must be $\blacktriangleleft$ and extend all the way to the left boundary, contradicting the assumption. The argument when the rightmost signal is $\scalebox{0.8}{\begin{tikzpicture}
    \draw [] (0,0) circle (0.2);
    \node at (0,0) {$\checkmark$};
\end{tikzpicture}}$ is analogous.

The argument for columns is analogous, with the additional case that starting from a free symbol at the bottom and going up, it might eventually reach a vertex of degree 3, but as we have defined the local rules so that the vertical edge is not allowed to be free, this case cannot occur.\end{proof}

\begin{remark}
    We note that alternatively, we can get rid of the $\times$ signal and instead consider the border rows and columns ``free'' as well. The reason why we don't do it is because the sequence of words that encodes the bits that occur above the free rows has a more convenient structure with our convention. This will become evident in what follows.
\end{remark}

Next we exhibit how the sequence encoded by layer 3 is encoded by the free rows within the red borders. Let $n \geq 1$ and take a $4^n$-border. Recall that there are $2^n+1$ free rows (\Cref{lem:free-rows-inside-a-border}), and thus $2^{n-1}$ free rows above the center row. Enumerate each one of these from the center upwards, i.e., the first row is the free row directly above the center row, and the $2^{n-1}$-th row is the top free row. Each of these rows sit directly below an obstructed row (which carries a bit from a $4^k$ border, $k<n$) or the $4^n$-border.  
\begin{definition}
    For every $n \geq 1$ and $(\varphi,x)\in \mathfrak{X}_4$, define  $w_n \in \{0,1\}^*$ by
\[
w_n(i) = \text{the horizontal bit carried by the row directly above the $i$-th free row}, \ 1\leq i \leq 2^{n-1}.
\]
\end{definition}
Thus $w_n$ has length $2^{n-1}$. The fact that $w_n$ is independent of the chosen $4^n$-border in $x$ follows from \Cref{prop:border-bit-constancy}.
\begin{lemma}\label{lem:input-recursion}
    Let $y \in \{0,1\}^\NN$ be the sequence encoded by the third layer in some configuration of $\mathfrak{X}_4$. The words $w_n$ satisfy the following recursion
    \begin{align*}
    w_1 &= y(1) \\
    w_{n+1} &= w_nw_n',
\end{align*}
where $w_n'$ is $w_n$ with its last symbol replaced by $y(n+1)$.
\end{lemma}
\begin{proof}
It is convenient to define an auxiliary sequence of words $(u_n)_{n\geq 1}$, which is similar to $(w_n)_{n\geq 1}$ but where we consider blocks instead of borders. Let $n\geq 1$, and consider a $2^{2n-1}$-block. We enumerate \textit{all} horizontal rows inside our block that do not intersect a $2^{2i}$-border, $1\leq i\leq n-1$, from the bottom upwards. There are $2^{n-1}$ such rows (\Cref{lem:free-rows-inside-a-block}). Furthermore, each one of them, except the uppermost one, is directly below an obstructed row which carries a bit coming from a $2^{2i}$-border inside our initial block. We define a word $u_n$ of length $2^{n-1}-1$ by  
\[u_n(i)=\text{ the bit carried by the row directly above the $i$-th free row}\]
for all $1\leq i\leq 2^{n-1}-1$. 

We claim that $(u_n)_{n\geq 1}$ and $(w_n)_{n\geq 1}$ are related by
\[
w_n=u_n y(n).
\]
Indeed, take a $4^n=2^{2n}$-border. After we exclude the horizontal row at the middle, the upper half of our border is a horizontal row of $2^{2n-1}$-blocks (\Cref{prop:border-decomposition}). It follows from the definitions that $w_n$ and $u_n$ coincide in their first $2^{n-1}-1$ digits, while it is clear that the last digit of $w_n$ is $y(n)$, as the uppermost free row is directly below the $2^{2n}$-border. 

Now consider a $4^{n+1}=2^{2(n+1)}$-border. Decomposing the $2^{2n+1}$-blocks in its upper half into $2^{2n-1}$-blocks, one can see that 
\[
w_{n+1}=u_{n}y(n)u_{n}y(n+1).
\]
We remark that the bit $y(n)$ comes from the $4^n$-borders inside our initial $4^{n+1}$-border, while the bit $y(n+1)$ is transmitted from the $4^{n+1}$-border itself to the free row below it. Since $w_n=u_ny(n)$ and $w_{n}'=u_ny(n+1)$, we have $w_{n+1}=w_nw_n'$ as claimed. 
\end{proof}
We define a function $f\colon\{0,1\}^{\NN}\to\{0,1\}^{\NN}$ as follows. Given $y\in\{0,1\}^{\NN}$, we define $f(y)\in\{0,1\}^{\NN}$ by
\[f(y)(k)=\text{the $k$-th digit in $w_n$, for $n$ sufficiently large}\]
for all $k\in\NN$. Thanks to \Cref{lem:input-recursion}, the $k$-th digit in $w_n$ remains constant as $n$ increases, so this is well-defined. 

Recall that we already fixed a $\Pi_1^0$ set $P$ whose Medvedev degree we want to realize. The following result will be important for our proof. 
\begin{proposition}\label{prop:f-P-and-P}
$f(P)$ is a $\Pi_1^0$ set, and it has the same Medvedev degree as $P$.
\end{proposition}
\begin{proof}
It is sufficient to observe that $f$ is computable, injective, and it admits a computable left inverse $g$ from $f(\{0,1\}^{\NN})$ to $\{0,1\}^{\NN}$. It is clear from the definition that $f$ is computable and injective. Furthermore, note that one can computably recover the first $n$ digits of $y$ from $w_n$ (equivalently, the first $2^{n-1}$ digits in $f(y)$). This procedure provides an algorithm for $g$. 
\end{proof}

\subsection*{Layer 5 (Computation layer \texorpdfstring{$\mathfrak{X}_5$}{X5})} This technique of simulating Turing machines on grids is well-known; we describe the details for completeness. Importantly, the passage of time is from left to right, each free row in a border carries one step of computation, and the starting tape is coded vertically in the leftmost free column (which is the one right after the left border).

We shall consider Turing machines that take as input an infinite sequence and return no output; they will either halt or not. More precisely, a Turing machine $M$ will consist of a bi-infinite vertical tape with the alphabet $S = \{0,1,b\}$, a finite set of states $Q = \{q_1,q_2,\dots,q_n \}$, a reading head which may go up $(u)$ or down $(d)$, and a transition function
\[
\delta \colon D \to Q \times S \times \{u,d\},
\]
where $D$ is a subset of $Q \times S$. For instance, $\delta(q_i,s) = (q_j,s',u)$ means that if the head is in state $q_i$ reading the symbol $s$, then it will transition to the state $q_j$, replace $s$ by $s'$ and move up.

Given $x \in \{0,1\}^\NN$, the starting position of $M$ on input $x$ is defined as follows. The head starts in state $q_1$ reading the blank symbol $b$ at position $0$ of the tape. The infinite sequence $x$ is inscribed at positions $1,2,3,\dots$ going upwards. The rest of the cells are left with the blank symbol $b$.

We say that the machine $M$ halts on input $x$ if, from the starting position, the head eventually reaches a state $(q_i,s)$ which is not in the domain $D$ of the transition function $\delta$. Otherwise, we say that the machine does not halt.

With this description, a set $C\subset\{0,1\}^{\NN}$ is $\Pi_1^0$ if and only if there exists a Turing machine $M$ such that $x \in C$ if and only if $M$ does not halt on input $x$. Since $f(P)$ is a $\Pi_1^0$ set (\Cref{prop:f-P-and-P}), we can fix a Turing machine $M$ which does not halt if and only if its input belongs to $f(P)$. 

We are ready to describe the alphabet $A_5$ of this layer. Let $\delta \colon D\to Q \times S \times \{u,d\}$ be the transition function of the Turing machine $M$. The horizontal edges will be either blank ($\sqcup$) or labeled by elements from
\[
S \cup \{ q_is : q_i \in Q, s \in S \},
\]
and the vertical edges may be blank ($\sqcup$) or labeled by
\[
\{\uparrow, \downarrow \} \cup Q.
\]

Besides the usual rule that the labels on edges shared by adjacent vertices must coincide, we impose the following local rules:

\begin{enumerate}
    \item A vertex which has no incident edges marked by the free symbol $\scalebox{0.8}{\begin{tikzpicture}
    \draw [] (0,0) circle (0.2);
    \node at (0,0) {$\checkmark$};
\end{tikzpicture}}$ will have all its edges labeled blank.
\item For a vertex which has two edges marked with the free symbol and an obstruction signal on the other two, we make it transmit the information by allowing only the following diagram for all (possibly blank) symbols $a$ and $b$.

\begin{center}
\begin{tikzpicture}[scale=1.2]
    
    \begin{scope}[shift={(0,0)}]
        \vertexlabels{a}{b}{a}{b}
    \end{scope}

\end{tikzpicture}
\end{center}
\item For a vertex of degree 3 which has two horizontal edges marked with the free symbol, we make it transmit the information horizontally by allowing only the following diagram for all (possibly blank) symbols $a$ and $b$.

\begin{center}
\begin{tikzpicture}[scale=1.2]
    
    \begin{scope}[shift={(0,0)}]
        \vertexlabeldegreethree{a}{b}{a}
    \end{scope}

\end{tikzpicture}
\end{center}
\item We set the starting position of the machine and the input tape as follows. On a vertex which has four incident edges marked by the free symbol $\scalebox{0.8}{\begin{tikzpicture}
    \draw [] (0,0) circle (0.2);
    \node at (0,0) {$\checkmark$};
\end{tikzpicture}}$, and which lies on the leftmost free vertical column (this is a local condition thanks to \Cref{rem:leftmost-red-border}), we allow the following diagrams for each $s \in \{0,1\}$. 
\begin{figure}[ht!]
\begin{tikzpicture}[scale=1.2]
    \begin{scope}[shift={(0,0)}]
        \upborder{s}
    \end{scope}

    \begin{scope}[shift={(2.5,0)}]
        \centerborder{q_1\sqcup}
    \end{scope}

    \begin{scope}[shift={(5,0)}]
        \downborder{\sqcup}
    \end{scope}
\end{tikzpicture}
\end{figure}
We impose that when the first diagram occurs, $s$ must be equal to the bit from layer 3 carried by the row immediately above. Finally, we also force the second diagram to occur on the center row by imposing that this diagram must be immediately between two free rows. Note that this can be determined locally as it is the only free row having free rows above and below (\Cref{lem:free-obstructed-rows-disposition}). The arrows in the image represent the symbols $\{\uparrow,\downarrow\}$ and thus are transmitted vertically.

    \item On a vertex which has four incident edges marked by the free symbol $\scalebox{0.8}{\begin{tikzpicture}
    \draw [] (0,0) circle (0.2);
    \node at (0,0) {$\checkmark$};
\end{tikzpicture}}$, and which is not on the leftmost free column, the following diagrams are allowed to occur for all $s \in S$ and $q_i \in Q$.

\begin{center}
\begin{tikzpicture}[scale = 1.2]
    \begin{scope}[shift={(0,0)}]
        \vertexlabels{s}{}{s}{}
    \end{scope}

    \begin{scope}[shift={(2.5,0)}]
        \vertexlabels{q_is}{}{s}{q_i}
    \end{scope}

    \begin{scope}[shift={(5,0)}]
        \vertexlabels{q_is}{q_i}{s}{}
    \end{scope}
\end{tikzpicture}
\end{center}

Furthermore, if $\delta(q_i,s) = (q_j,s',u)$ or $\delta(q_i,s) = (q_j,s',d)$, then the following diagrams are allowed respectively.

\begin{center}
\begin{tikzpicture}[scale=1.2]
    
    \begin{scope}[shift={(2.5,0)}]
        \vertexlabels{s'}{q_j}{q_is}{}
    \end{scope}

    \begin{scope}[shift={(5,0)}]
        \vertexlabels{s'}{}{q_is}{q_j}
    \end{scope}

\end{tikzpicture}
\end{center}
\end{enumerate}

    




Thus layer $5$ is the $\mathcal{H}_3$-SFT $\mathfrak{X}_5$ defined by those local rules above. At this point we set $\mathfrak{X}=\mathfrak{X}_5$ and thus we have finished the definition of our $\mathcal{H}_3$-SFT. It only remains to prove that $\mathfrak{X}$ has the desired properties. Let $\pi_3 \colon \mathfrak{X}\to \mathfrak{X}_3$ be the projection to layer 3 and let $\widehat{\phi}\colon \mathfrak{X}\to \{0,1\}^{\NN}$ be given by $\widehat{\phi}(\varphi,x)= \phi(\pi_3(\varphi,x))$. Note that this is still a computable map, but we may no longer have surjectivity. In what follows we will show that in fact $\widehat{\phi}(\mathfrak{X})=P$.

\begin{lemma}\label{lem:X-and-P}
    $\widehat\phi(\mathfrak{X})\subset P$.
\end{lemma}
\begin{proof}

Let $(\varphi,x)\in\mathfrak{X}$, and let $\widehat{\phi}(\varphi,x)=y\in\{0,1\}^{\NN}$. We argue that $y\in P$. Suppose for a contradiction that this is not the case. Since $y\in P$ if and only if $f(y)\in f(P)$, it follows that $f(y)\not\in f(P)$. By our choice of $M$, it must halt on input $f(y)$. Let $t\in\NN$ be large enough so that $M$ halts on input $f(y)$ on less than $t$ steps. In order to do this, $M$ only needs the first $t$ symbols of its input tape. Pick $n\in\NN$ large enough so that the grid of free columns and rows in the upper half of a red $4^n$-border has size at least $t\times t$. In the definition of Layer 5, the only allowed transitions are those that do not reach a halting state (see the condition on the domain of $\delta$). As a consequence, there is no valid pattern with support a red $4^n$-border whose initial bits are a prefix of $f(y)$, hence $(\varphi,x)\notin \mathfrak{X}$, yielding a contradiction. \end{proof}

\begin{lemma}\label{lem:P-and-X}
    There is a computable map that, on input $y \in P$, computes $(\varphi,x)\in \mathfrak{X}$ such that $\widehat{\phi}(\varphi,x)=y$.
\end{lemma}
\begin{proof}
Let $y \in P$. Let $(\varphi_0,x_0)\in \mathfrak X_0$ and $(\varphi_0,x_1)\in \mathfrak X_1$ be computable configurations as in \Cref{thm:X_1-sofic}. Thus every vertex eventually lies within a $4^n$-border in $x_1$. 

From $(\varphi_0,x_1)$ we compute $(\varphi_0,x_2)\in \mathfrak{X}_2$ by assigning color red to the $4^n$-borders and green to the rest ($2^k$-borders with $k\geq 1$ odd). Using as input $y \in P$ we construct $(\varphi_0,x_3)\in \mathfrak{X}_3$ by assigning symbol $y(n)$ to the $4^n$-borders iteratively on $n$, and extending the symbols horizontally and vertically according to the local rules of layer 3. If within some $4^n$-border some row or column remains undefined at some step of this procedure, we assign it the symbol $0$. This procedure is computable from $y$ and thus it provides a computable configuration $(\varphi_0,x_3)\in \mathfrak{X}_3$. 

Next, we add the obstruction signals corresponding to layer 4. By the assumption that every vertex is eventually within some $4^n$-border, we just need to describe how to algorithmically assign signals to $4^n$-borders for fixed $n\geq 1$. The first step is to assign symbols to any $4^k$-border for $k <n$ that lies within the $4^n$-border and has still not been filled. Next we assign all symbols that have only one possibility ($\times$, the free symbol and arrows that connect the border to some smaller border or a vertex of degree 3). Finally, the remaining arrows are assigned by giving priority to left arrows over right arrows and down arrows over upward arrows. Through this procedure we have computed from $y \in P$ a configuration $(\varphi_0,x_4)\in \mathfrak{X}_4$. 

Finally, we compute $(\varphi_0,x_5)\in \mathfrak{X}_5$ as follows. Note that as $y \in P$, we have $f(y)\in f(P)$, and thus the machine $M$ does not halt on input any prefix of $f(y)$ (meaning, the states stay in $D$). Hence for any such prefix the machine $M$ produces arbitrarily long space-time diagrams, which we use to fill in (deterministically) the symbols in layer 5 within every $4^n$-border that contains the origin. This procedure is clearly computable.

Setting $(\varphi,x)=(\varphi_0,x_5)\in \mathfrak{X}$ and noting that by construction $\widehat{\phi}(\varphi,x)=y$ finishes the proof.
\end{proof}

\begin{corollary}
    The $\mathcal{H}_3$-SFT $\mathfrak{X}$ has the same Medvedev degree as $P$. 
\end{corollary}

\begin{proof}
    By~\Cref{lem:X-and-P} we have $\widehat{\phi}(\mathfrak{X})\subset P$. As $\widehat{\phi}$ is a computable map, it follows that $P\preceq_{\mathfrak M} \mathfrak{X}$. Conversely, by~\Cref{lem:P-and-X} there is a computable map which takes $y\in P$ and computes $(\varphi,x)\in \mathfrak{X}$. In particular $\mathfrak{X} \preceq_{\mathfrak M} P$.
\end{proof}

\section{A hyperbolic analogue of Robinson's tiling, and proof of \texorpdfstring{\Cref{thm:X_1-sofic}}{Theorem 4.6}}\label{sec:appendix}%
In this section we define and study the SFT $\mathfrak X_0$, and prove \Cref{thm:X_1-sofic} involving $\mathfrak X_0$ and $\mathfrak X_1$. The SFT $\mathfrak X_0$ can be seen as a hyperbolic version of Robinson's tiling of the plane \cite{robinson_undecidability_1971} for $\mathcal H_3$.  

As mentioned in the introduction, finding such an adaptation was an open problem for some time (see \cite{Robinson_undecidable_hyperbolic}, \cite{Margenstern_2008_domino_hyp}, and \cite{goodman-strauss_hierarchical_2010}). The SFT we present is based on Goodman-Strauss's construction \cite{goodman-strauss_hierarchical_2010}. Besides the blueprint formalism, a difference in our approach is that we place the symbols of the alphabet on vertices of the graph instead of its faces. These two approaches are in a sense equivalent, but working with vertices creates somewhat simpler hierarchical decompositions\footnote{In contrast to \cite[Figure 6]{goodman-strauss_hierarchical_2010}, there are no Ziggurats or stair-like borders.}. This allows us to drop the requirement $k\equiv 1\mod 4$ from \cite{goodman-strauss_hierarchical_2010}, and the construction we present can be easily adapted to $\mathcal H_k$ for all odd $k\geq 3$. 

We will define $\mathfrak X_0$ in \Cref{subsec:definition-X-0}. We remark that it can be seen as a decorated version of the simpler object $\mathfrak X_1$, defined in \Cref{sec:sofic}. We will continue to use some terminology introduced in \Cref{sec:sofic}. In order to facilitate the understanding of the decorations we use, we will first explain them in the Euclidean case in \Cref{subsec:explanation-robinson}, using Robinson's SFT as a toy case. In fact, let us stress that the modifications we need to do are rather elementary. 

\begin{remark}\label{rem:standing-convention}
We will continue to use the following convention from \Cref{sec:sofic}. When we place symbols on vertices of a model graph, these symbols are ordered tuples of decorations for the edges incident to a vertex. The following local rule is always included: it is forbidden that two neighboring vertices assign different decorations to the same edge. This convention allows us to speak of decorations for edges, if convenient. We illustrate this in \Cref{fig:matching-rules}.
\end{remark}

\begin{figure}[h]
    \centering
    \includegraphics[page=1,width=0.3\textwidth]{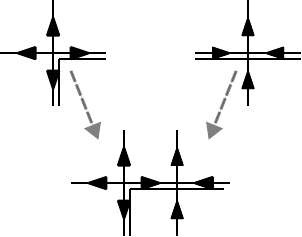}
    \caption{In the first row we have two alphabet symbols (coming from \Cref{fig:robinson-presentation}). In the second row, we have two neighboring vertices where we place these symbols, obeying our standing local rule that symbols on neighboring vertices must assign the same decoration to a common edge. }
    \label{fig:matching-rules}
\end{figure}

\subsection{The subshift $\mathcal R$ and some decorated versions}\label{subsec:explanation-robinson} 
We now do a short overview of Robinson's tiling of the plane, some relevant properties, and some elementary variations of it. Our goal is to highlight what properties we want to translate from the Euclidean case to $\mathcal H_3$, and what kind of extra decorations we use for this.

The objects in this subsection can be interpreted with the usual meaning of an SFT on the group $\ZZ^2$, or on the blueprint $\mathcal{H}_1$, whose only model gives rise to the Cayley graph of $\ZZ^2$ with the generating set $\{(0,1),(1,0)\}$. See \Cref{rem:H-1}. 
\subsubsection{Definition of $\mathcal R$} 
The alphabet of $\mathcal R$, consisting of elements called \define{crosses} and \define{arms}, is defined in \Cref{fig:robinson-presentation}. Then $\mathcal R$ is the set of all configurations satisfying the following constraints: 
\begin{enumerate}
    \item Neighboring vertices assign the same decoration to a common edge (\Cref{fig:matching-rules}).
    \item There exists an infinite $2\times 2$ grid such that every vertex in this set carries a cross. By an infinite $2\times 2$ grid we mean a translate of $(2\ZZ)^2$. 
\end{enumerate}
\begin{remark}
The second condition can be enforced with forbidden patterns of size $8\times 8$. Indeed, there are exactly four ways to place a $2\times 2$ grid in a support of size $8\times 8$. The key observation is that it is not possible to decorate this $8\times 8$ support with the symbols from \Cref{fig:robinson-presentation}, respecting the constraint that neighboring vertices assign the same decoration to a common edge, and with at least two of the $2\times 2$ sub-grids of the support having crosses. 

Let us recall that in Robinson's original construction \cite{robinson_undecidability_1971}, the property that crosses appear in an infinite $2\times 2$ grid is obtained by adding parity markings or bumpy corners in a different step, but this is done to guarantee the nearest neighbor property.
\end{remark}

\begin{figure}[h]
    \centering
    \includegraphics[page=1,width=0.6\textwidth]{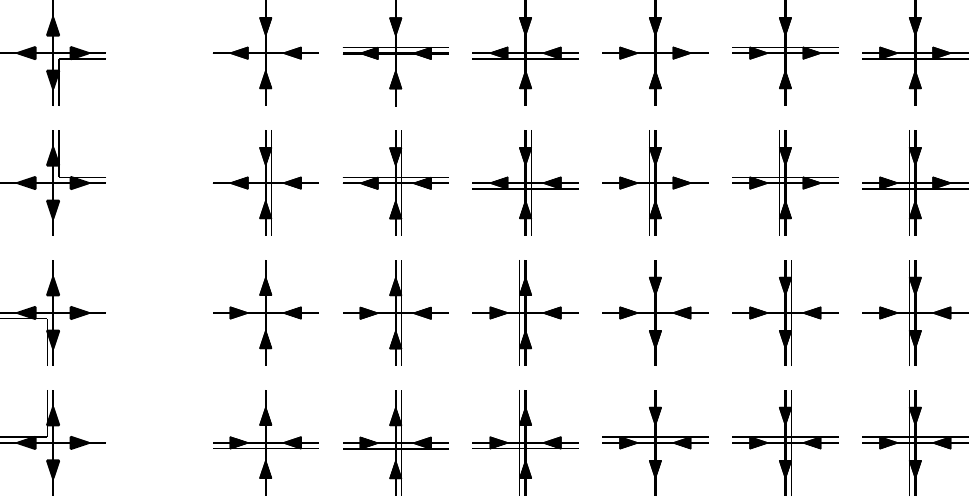}
    \caption{Alphabet of $\mathcal R$. Those on the left are called \define{crosses}, and those on the right are called \define{arms}. The first two rows of arms are \define{horizontal arms}, and the next two rows are \define{vertical arms}. It is convenient to think that a horizontal (resp. vertical) arm transmits arrows in the horizontal (resp. vertical) direction. We say that a cross \textbf{faces} in the directions indicated by its double arrows. For instance, the first cross faces right and down.}
    \label{fig:robinson-presentation}
\end{figure}
We remark that our pictures for double-arrows are slightly different from those in \cite[Figure 2]{robinson_undecidability_1971}, but this difference is just cosmetic (after identifying the vertices and faces of the Cayley graph of $\ZZ^2$ with respect to $\{(0,1),(1,0)\}$, which is possible because it is self-dual as a planar graph). 
\begin{figure}[h]
    \centering
\includegraphics[page=1,width=1\textwidth]{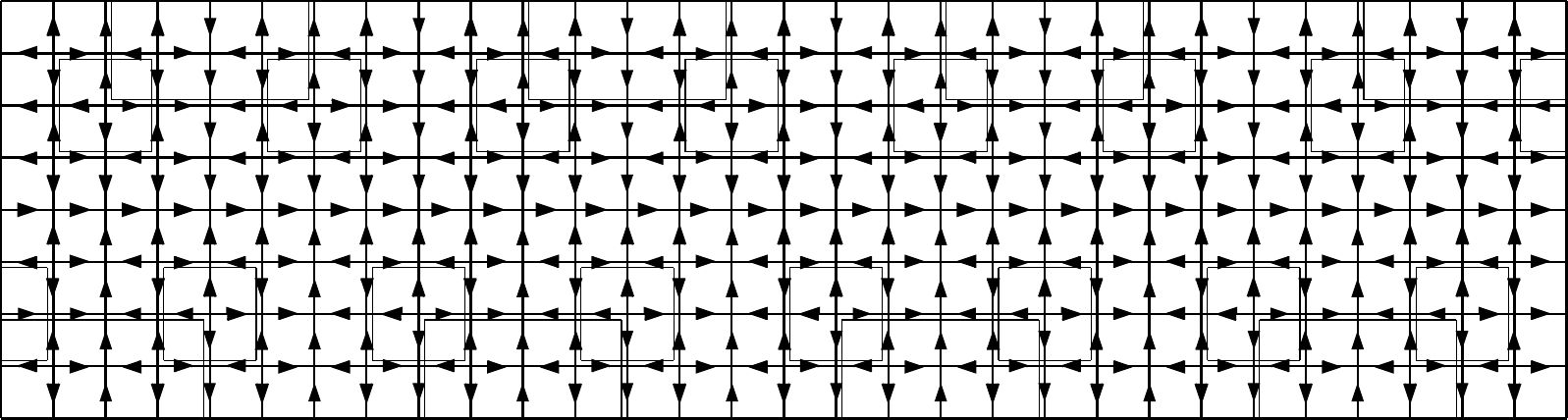}
    \caption{Finite portion of a configuration in $\mathcal R$.}\label{fig:appendix-robinson-0-example}
\end{figure}
\subsubsection{The hierarchical structure in $\mathcal R$} 
The hierarchical structure of $\mathcal R$ is given by patterns which behave like larger and larger copies of crosses (called $(2^n-1)$-squares in \cite{robinson_undecidability_1971}), which we call \define{blocks}. A $2$-block is just a cross. Inductively, a $2^{n+1}$-block is obtained by putting together four $2^n$-blocks so that their central crosses face each other and form the corners of a square. In the center of the $2^{n+1}$-block we place a cross, and arrows born in this cross propagate through the boundaries of the four $2^n$-blocks, via arms. Thus the support of a $2^n$-block is a square grid of vertices with side $2^n-1$. We say that a block \define{faces} in the direction indicated by its central cross (see \Cref{fig:robinson-presentation}).
\begin{figure}[h]
    \centering
\includegraphics[page=1,width=0.5\textwidth]{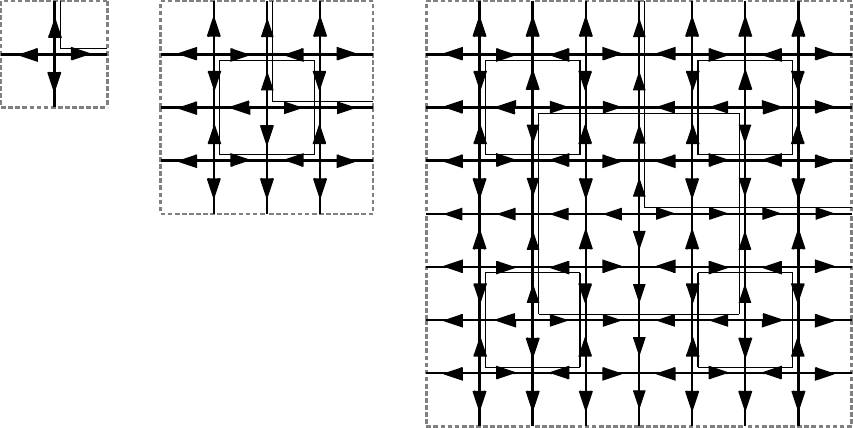}
    \caption{A $2$-block, a $4$-block (formed by four $2$-blocks), and an $8$-block (formed by four $4$-blocks) in $\mathcal R$. They face up and right, and edges in their boundaries are drawn with gray dashed lines. }
    \label{fig:robinson-0-blocks}
\end{figure}
It can be proved by induction on $n$ that every configuration in $\mathcal R$ can be partitioned into an infinite arrangement of $2^n$-blocks (up to the boundaries of these $2^n$-blocks), for all $n\geq 1$. See \cite[\S 3]{robinson_undecidability_1971}. 
\subsubsection{The alignment property in $\mathcal R$}\label{subsec:euclidean-alignment} There is a crucial asymmetry between $2$-blocks and larger blocks in $\mathcal R$. That is, $2$-blocks are always well-aligned, while this is not necessarily true for $2^n$-blocks, $n>1$. We will now explain what we mean by alignment (see also \cite[p. 189]{robinson_undecidability_1971} and \cite[p. 31]{gayral_complexite_2023}). 
\begin{figure}[h]
    \centering
    \includegraphics[page=1,width=1\textwidth]{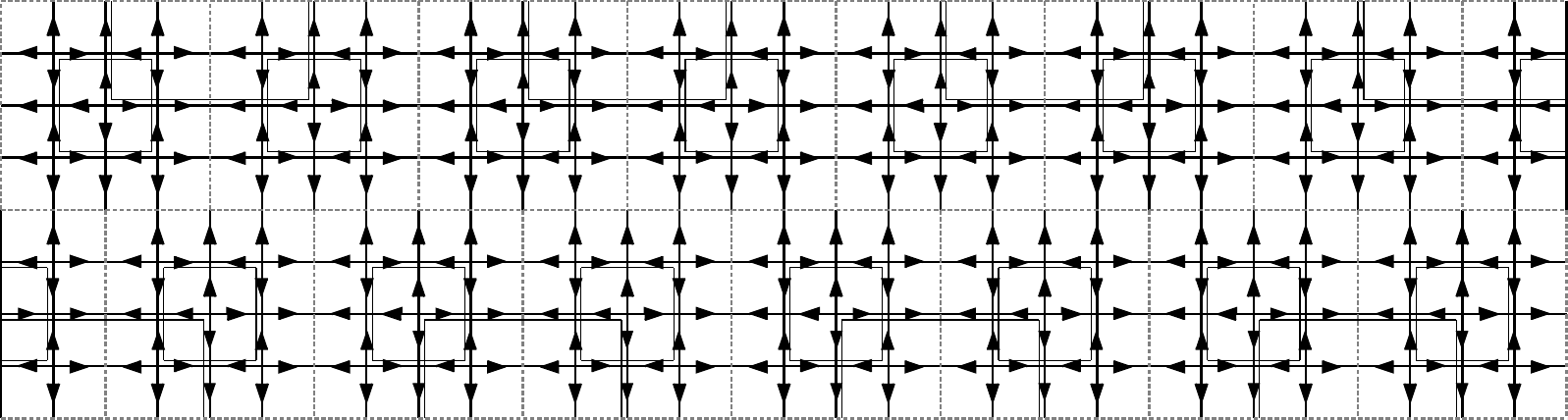}
    \caption{This is the pattern from \Cref{fig:appendix-robinson-0-example}, where we have removed the decorations on boundaries of $4$-blocks, and replaced them by gray dashed lines. We can see that the $4$-blocks are not vertically aligned. Decomposing them into $2$-blocks, we see that this is because $2$-blocks facing right (resp. left) sit above $2$-blocks facing left (resp. right). The infinite horizontal line separating these $4$-blocks is called a ``fault line'' in \cite{robinson_undecidability_1971} (the horizontal gray line in the middle).}
\label{fig:robinson-0-alignment}
\end{figure}

We say that $2^n$-blocks in a configuration of $\mathcal R$ are \define{vertically well-aligned} when every $2^n$-block has exactly one lower vertical neighbor. We illustrate the failure of vertical alignment for $4$-blocks in \Cref{fig:robinson-0-alignment}. This illustration for $4$-blocks is in fact representative of what happens for $2^{n+1}$-blocks for all $n$, and we have the following characterization.
\begin{proposition}\label{prop:euclidean-alignment-characterization}
For a configuration in $\mathcal R$ and $n\geq 1$, the following two properties are equivalent.
\begin{enumerate}
\item $2^n$-blocks are vertically aligned, but $2^{n+1}$-blocks are not.
\item There exists a $2^n$-block facing right (resp. left), sitting above a $2^n$-block facing left (resp. right). 
\end{enumerate}
\end{proposition}
\begin{proof}
    This can be proved by induction, and the argument is left to the interested reader.
\end{proof}
\begin{remark}\label{rem:alignment}
It follows from \Cref{prop:euclidean-alignment-characterization}  that one can enforce vertical alignment of $2^n$-blocks for all $n$ by adding further decorations to arrows in $\mathcal R$, so that blocks with crosses facing left (resp. right) can only sit above blocks with crosses facing left (resp. right). The same applies to horizontal alignment. This method is well-known; see for instance \cite[\S 2.7.2]{gayral_complexite_2023} for a detailed exposition.
\end{remark}
\subsubsection{A first variation of $\mathcal R$, which ensures vertical alignment} 
We now define $\mathcal R_1$, a variation of $\mathcal R$ where we add decorations to the arrows to ensure vertical alignment for blocks, as explained in \Cref{rem:alignment} (see \Cref{prop:euclidean-alignment-characterization}). The alphabet for $\mathcal R_1$ is defined in \Cref{fig:robinson-1-presentation}. Then $\mathcal R_1$ and its blocks are defined as we did for $\mathcal R$. Some blocks for $\mathcal R_1$ are illustrated in \Cref{fig:robinson-1-blocks}.

In $\mathcal R_1$ we decorate the central crosses according to whether they face left or right. These decorations also propagate vertically, ensuring that a $2^n$-block facing left (resp. right) can only be a vertical neighbor of a $2^n$-block facing left (resp. right). This prevents the situation in \Cref{fig:robinson-0-alignment}. It follows from \Cref{prop:euclidean-alignment-characterization}  that in $\mathcal R_1$, $2^n$-blocks are vertically aligned for all $n\geq 1$.  
\begin{figure}[h]
    \centering
    \includegraphics[page=1,width=0.6\textwidth]{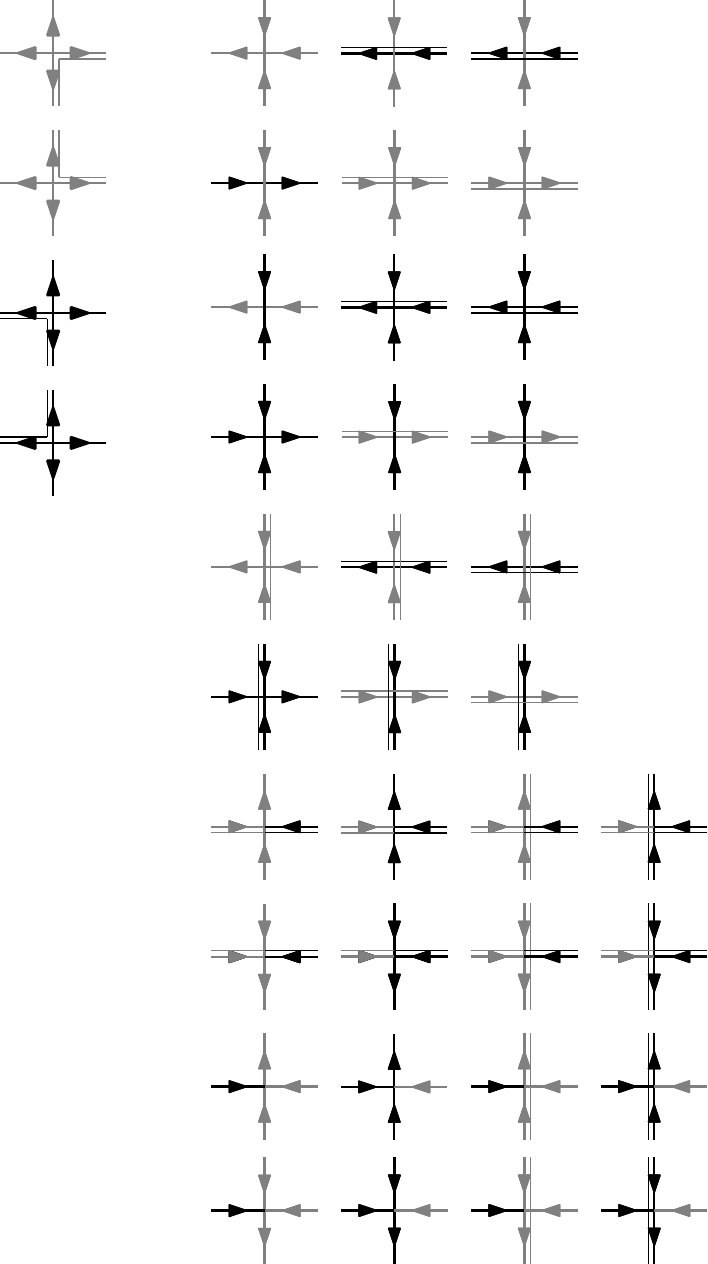}
    \caption{Alphabet of $\mathcal R_1$, obtained by adding extra marks to the alphabet of $\mathcal R$.} 
    \label{fig:robinson-1-presentation}
\end{figure}
\begin{figure}[h]
    \centering
    \includegraphics[page=1,width=0.6\textwidth]{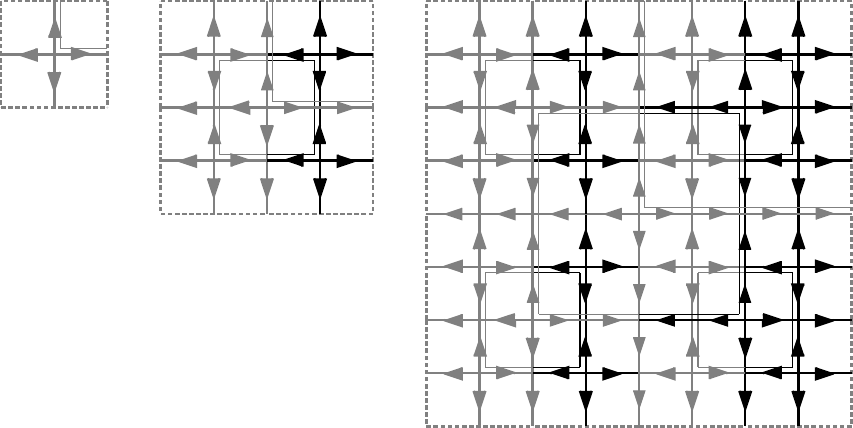}
    \caption{A $2$-block, a $4$-block, and an $8$-block in $\mathcal R_1$. Their boundaries are drawn with gray dashed lines.} 
    \label{fig:robinson-1-blocks}
\end{figure}

\subsubsection{A second variation of $\mathcal R$, which ensures horizontal alignment}
We now define $\mathcal R_2$, a variation of $\mathcal R$ where we decorate the arrows to ensure horizontal alignment. The alphabet for $\mathcal R_2$ is defined in \Cref{fig:robinson-2-presentation}. Then $\mathcal R_2$ and its blocks are defined as we did for $\mathcal R$. Some blocks for $\mathcal R_2$ are illustrated in \Cref{fig:robinson-2-blocks}. 

The extra decorations in $\mathcal R_2$ ensure horizontal alignment for the same reason as in $\mathcal R_1$, replacing ``vertical'' by ``horizontal''. 
\begin{remark}
It might seem arbitrary that we use a different visual mark for $\mathcal R_1$ and $\mathcal R_2$. In fact, we could have also used either decoration for both vertical and horizontal arrows in a single step. However, the symmetry between ``horizontal'' and ``vertical'' that we have in $\ZZ^2$ will be lost in $\mathcal H_3$, and for this reason we prefer different visual marks. 
\end{remark}
\begin{figure}[h]
    \centering
    \includegraphics[page=1,width=0.6\textwidth]{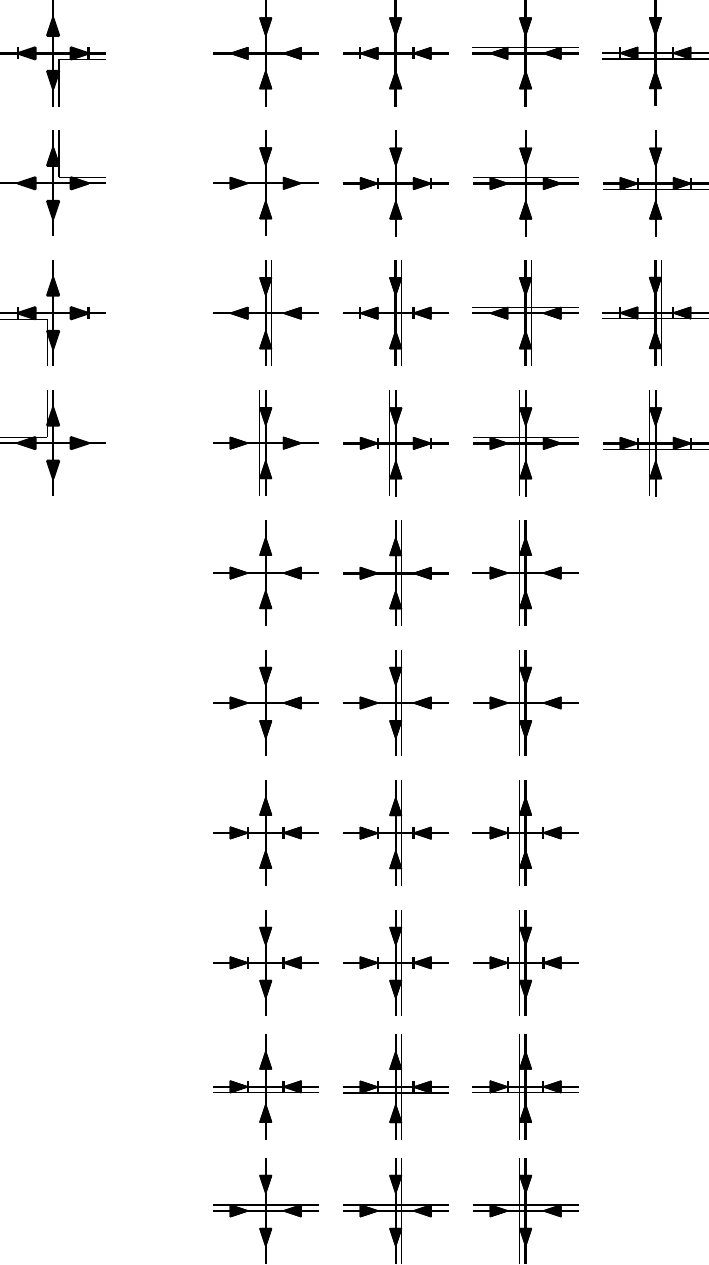}
    \caption{Alphabet of $\mathcal R_2$, obtained by adding extra marks to the alphabet of $\mathcal R$.} 
    \label{fig:robinson-2-presentation}
\end{figure}
\begin{figure}[h]
    \centering
    \includegraphics[page=1,width=0.6\textwidth]{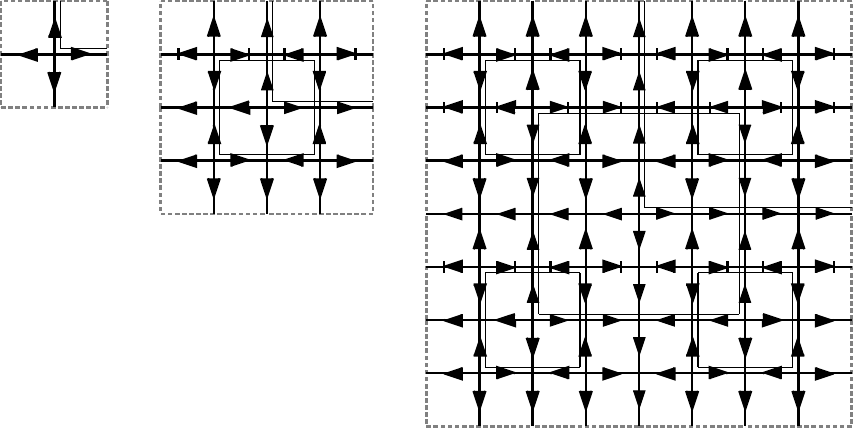}
    \caption{A $2$-block, a $4$-block, and an $8$-block in $\mathcal R_2$. Their boundaries are drawn with gray dashed lines.} 
\label{fig:robinson-2-blocks}
\end{figure}

\subsubsection{Decorating boundaries of blocks in $\mathcal R$}\label{subsubsection:boundaries-property}
Consider a $2^n$-block in $\mathcal R$ which lies above a well-aligned vertical neighbor. Their shared boundary is a horizontal strip of $2^n$ edges and $2^n-1$ vertices. We claim that this strip is decorated in a uniform manner, meaning that all these horizontal edges carry the same kind of arrow. Indeed, first observe that each vertex in this strip receives vertical arrows from above and below, coming from the $2^n$-blocks. This implies that all vertices in this strip carry arms (\Cref{fig:robinson-presentation}). 
Because neighboring vertices must put the same decoration on a shared edge,  these arms will ``transmit'' any possible decoration from any edge to the whole strip of edges. This is illustrated in \Cref{fig:boundaries}.
\begin{remark}
The previous argument does not work naturally in $\mathcal H_3$, as in this case, the bottom side is much longer than the upper side (see \Cref{rem:n-squares-sizes}). 
\end{remark}
\begin{figure}[h]
    \centering
    \includegraphics[page=1,width=0.6\textwidth]{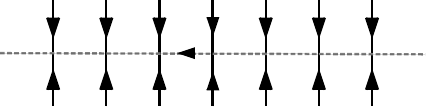}
    \caption{Illustration of a horizontal strip of edges in-between two 8-blocks (see \Cref{fig:robinson-0-blocks}). This horizontal strip of edges is decorated uniformly. To see this, it suffices to note that each vertex receives vertical arrows from above and below, coming from the $8$-blocks, and then it necessarily carries a horizontal arm. Such arms ``transmit'' arrows horizontally.}
    \label{fig:boundaries}
\end{figure}
\subsubsection{A third variation of $\mathcal R$, which ensures that bottom boundaries are decorated uniformly}
We now define $\mathcal R_3$, a further variation of $\mathcal R$. The alphabet for $\mathcal R_3$ is defined in \Cref{fig:robinson-3-presentation}. Then $\mathcal R_3$ and its blocks are defined as we did for $\mathcal R$. Some blocks for $\mathcal R_3$ are illustrated in \Cref{fig:robinson-3-blocks}.

The extra decorations in $\mathcal R_3$ allow us to prove the following, without looking at the lower vertical neighbor of a block (as opposed to the argument in \Cref{subsubsection:boundaries-property}).
\begin{proposition}
Let $n\geq 1$. The bottom horizontal side of the boundary of a $2^n$-block in $\mathcal R_3$ is decorated with arrows in a uniform direction, and  this direction is determined by the central cross of the $2^n$-block. 
\end{proposition}
\begin{proof}
The proof is by induction on $n$, and we only sketch the induction step. We decompose a $2^{n+1}$-block into $2^n$-blocks, and focus on the $2^n$-blocks in the lower half. Local rules ensure that these $2^n$-blocks have the same ``dot'' decoration, which in turn determines the direction of the arrows in their bottom boundaries. Finally, observe that this dot decoration is determined by the kind of central cross in the $2^{n+1}$-block (the arrow it emits downwards eventually goes through an arm as in the seventh row in \Cref{fig:robinson-3-presentation}).
\end{proof}
This argument will continue to work in $\mathcal H_3$.
\begin{figure}[h]
    \centering
    \includegraphics[page=1,width=0.6\textwidth]{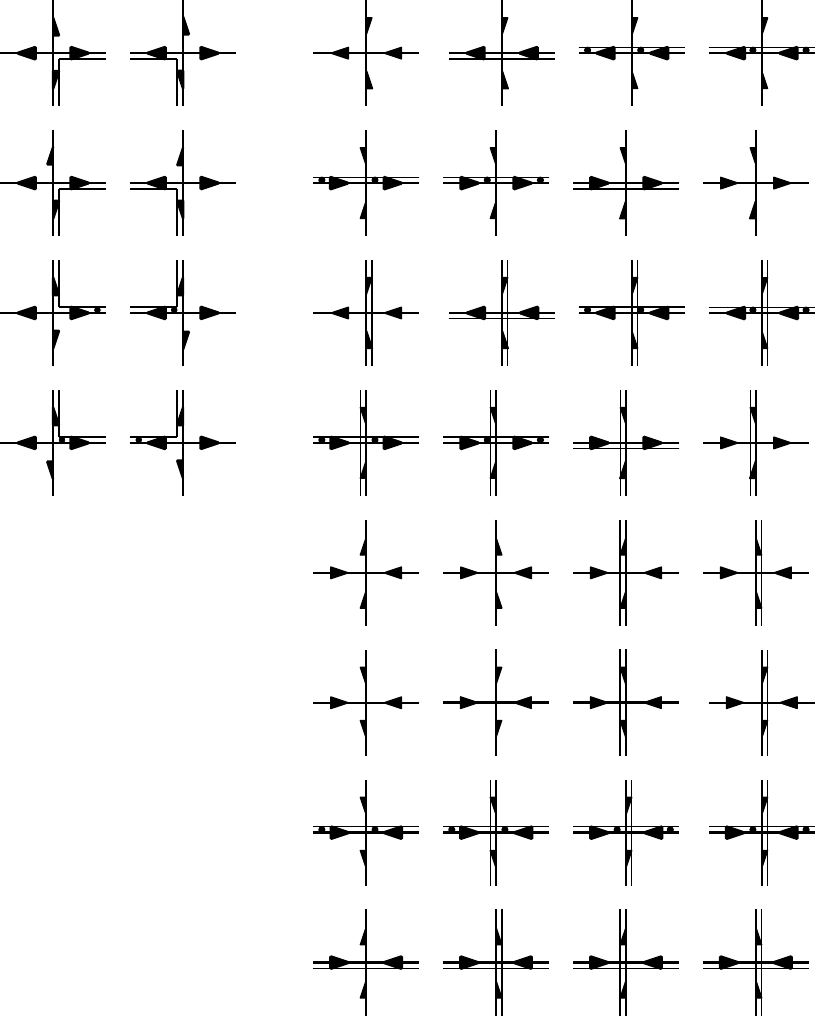}
    \caption{Alphabet of $\mathcal R_3$, obtained by adding extra marks to the alphabet of $\mathcal R$.} 
    \label{fig:robinson-3-presentation}
\end{figure}

\begin{figure}[h]
    \centering
    \includegraphics[page=1,width=0.6\textwidth]{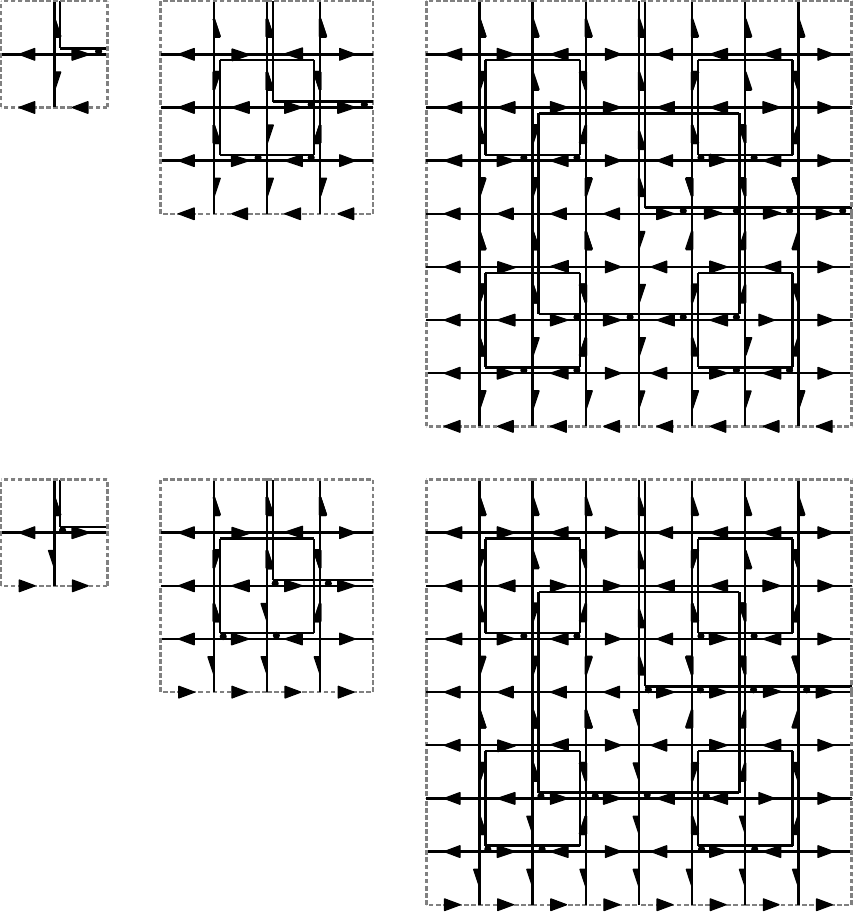}
    \caption{Two $2$-blocks, two $4$-blocks, and two $8$-blocks in $\mathcal R_3$. Their boundaries are highlighted with dashed gray lines, and arrows have been included on their bottom boundaries. The direction of these horizontal arrows is tied to the dot decoration in their central crosses. This can be seen by a direct inspection of the list of arms.} 
    \label{fig:robinson-3-blocks}
\end{figure}
\subsubsection{A combination of these variations}
Let $\mathcal R_4$ be the variation of $\mathcal R$ where we simultaneously add the decorations of $\mathcal R_1$,  $\mathcal R_2$,  and $\mathcal R_3$. $\mathfrak X_0$  is essentially the naive adaptation of $\mathcal R_4$ to a subshift on $\mathcal{H}_3$. The decorations from $\mathcal R_1$ and $\mathcal R_2$ will ensure alignment of the corresponding blocks, while the decorations from $\mathcal R_3$ will ensure that the bottom sides of the boundaries of the blocks are decorated uniformly. 

There is one further decoration that we cannot meaningfully represent in the Euclidean case. That is, in $\mathcal H_3$, new vertices appear when we move in the downwards direction. In order to fill this ``extra space'', we will use white arrows as in \Cref{fig:white-arrow}.
\begin{figure}[h]
    \centering
    \includegraphics[page=1,width=0.1\textwidth]{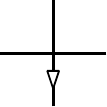}
    \caption{White arrow, which will be used to ``fill'' extra space when adapting $\mathcal R_4$ to $\mathcal H_3$.}
    \label{fig:white-arrow}
\end{figure}
\subsection{Partitions of models of $\mathcal H_3$ into cells}\label{subsec:cells}
Recall that we defined $2^n$-cells in \Cref{subsection:definition-of-X_1}. It is convenient to review some properties of partitions (up to boundaries) of the graph associated to a model of $\mathcal H_3$ into $2^n$-cells. 

Fix $n\geq 1$, a model $\varphi\in\mathcal M(\mathcal H_3)$, and the corresponding graph $\mathcal G(\mathcal H_3,\varphi)$. By a partition (up to boundaries) of $\varphi$ into $2^n$-cells, we mean a disjoint collection of $2^n$-cells such that every vertex in the associated graph belongs to some $2^n$-cell, or to the boundary of some $2^n$-cell. We also require that if a vertex is in the boundary of a $2^n$-cell, then it must belong to the boundary of at least two of them (in other words, these boundaries need to overlap). For such a partition, we define the following two properties:
\begin{itemize}
    \item \define{Horizontal alignment}: Every $2^n$-cell is the right (resp. left) horizontal neighbor of exactly one $2^n$-cell.
    \item \define{Vertical alignment}: Every $2^n$-cell is below exactly one $2^n$-cell.  
\end{itemize}
If both conditions are satisfied, we simply call it a partition (up to boundaries) of $\varphi$ into \define{well-aligned} $2^n$-cells. We illustrate these properties in \Cref{fig:2-cell-alignment-example}. 
\begin{figure}[h]
    \centering
    \includegraphics[page=1,width=1\textwidth]{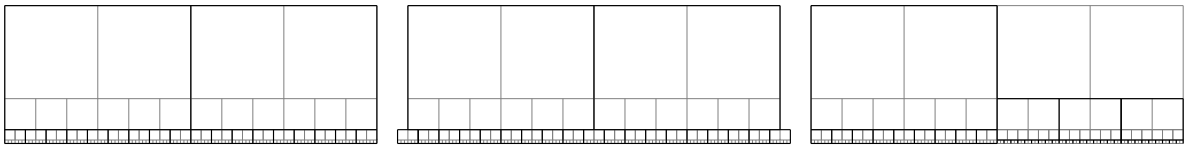}
    \caption{The leftmost picture shows some $2$-cells satisfying both vertical and horizontal alignment, the middle picture illustrates the failure of vertical alignment, and the rightmost picture illustrates the failure of horizontal alignment.  The boundaries of $2$-cells are highlighted in black.}
    \label{fig:2-cell-alignment-example}
\end{figure}
The following result will be key for us. 
\begin{proposition}\label{prop:sequence-of-partitions}%
    Let $\varphi\in \mathcal M(\mathcal H_3)$. For $n\geq 1$, let $P$ be a partition of $\varphi$ into well-aligned $2^n$-cells. There are exactly four choices of a partition $Q$ of $\varphi$ into well-aligned $2^{n+1}$-cells such that, if we decompose $Q$ into $2^n$-cells, we obtain $P$. Furthermore, given an arbitrary  $2^{n+1}$-cell $C$ formed by $2^n$-cells in $P$, there is exactly one such partition $Q$ having $C$ as one of its elements. 
\end{proposition}
\begin{proof}
    Let $n$, $P$, and $C$ be as in the statement. Because $2^n$-cells are horizontally well-aligned, it makes sense to speak of infinite horizontal rows of $2^n$-cells in $P$. We consider the two infinite horizontal rows of $2^n$-cells that intersect $C$. These two rows can be combined to form an infinite horizontal row of $2^{n+1}$-cells. There are exactly two ways to do this, and exactly one compatible with $C$ (the two choices can be parametrized by the following parity choice: take an arbitrary $2^n$-cell in the first row, and choose whether it will belong to the left, or right half, of a $2^{n+1}$-cell). Next, if we move down and consider the next two infinite horizontal rows of $2^n$-cells, they can also be combined to form a row of $2^{n+1}$-cells. As before, there are only two possibilities. Of these two possible choices, there is only one with the property that the $2^{n+1}$-cells we form are vertically well-aligned with respect to the $2^{n+1}$-cells above (the row where $C$ appears). The same argument applies to the row of $2^{n+1}$-cells above the one in which $C$ appears. Because the previous argument is valid for arbitrary $C$, we see that the  choice we made for the row of $2^{n+1}$-cells containing $C$ propagates arbitrarily both up and down, defining a partition $Q$ of $\varphi$ into $2^{n+1}$-cells as in the statement.
    
    The possible partitions $Q$ can be parametrized by four choices. First we choose which infinite horizontal rows of $2^n$-cells in $P$ will become upper/lower halves in $Q$. After this is done, we take some infinite horizontal row of $2^n$-cells in $P$ corresponding to upper halves, and then we choose which $2^n$-cells will be on the left or right half of its $2^{n+1}$-cell.
\end{proof}
An elementary observation that we have omitted until now is that for an arbitrary $\varphi\in\mathcal M(\mathcal H_3)$, a partition into well-aligned $2$-cells does exist. This follows from the argument in \Cref{prop:sequence-of-partitions}, but using the \textit{faces} of $\mathcal G(\mathcal H_3,\varphi)$ as the smaller cells (in the sense of planar graph). Applying \Cref{prop:sequence-of-partitions} $n-1$ times, we see that a fixed model in $\mathcal H_3$ admits exactly $4^n$ possible partitions into well-aligned $2^n$-cells, and these can be parametrized by a sequence of $n$ choices among four options. It also follows that they are all individually computable for computable $\varphi$. 

\subsection{Definition of \texorpdfstring{$\mathfrak{X}_0$}{X0} on \texorpdfstring{$\mathcal{H}_3$}{H3}}\label{subsec:definition-X-0}
We are now ready to define $\mathfrak{X}_0$, our hyperbolic analogue of $\mathcal R$. We start by defining the relevant alphabet. 
\begin{definition}\label{def:alphabet}
    The alphabet $A_0$ is the set of crosses and arms shown in \Cref{fig:decorated-crosses-left,fig:decorated-crosses-right,fig:decorated-vertical-arms,fig:decorated-horizontal-arms-left,fig:decorated-horizontal-arms-right,fig:decorated-horizontal-arms-degree-3-left,fig:decorated-horizontal-arms-degree-3-right,fig:decorated-special-arms}.
\end{definition}

In order to define $\mathfrak X_0$, we first define decorated $2$-blocks. 
\begin{definition}\label{def:decorated-blocks}
    A decorated $2$-block is a pattern whose support is a 2-cell, which in the center carries a decorated cross (\Cref{fig:decorated-crosses-left,fig:decorated-crosses-right}). Arrows that emanate from this cross propagate until they reach the boundary, through arms as in the first row of \Cref{fig:decorated-horizontal-arms-degree-3-left,fig:decorated-horizontal-arms-degree-3-right}. A decorated 2-block is shown in \Cref{fig:decorated-2-block-example}.
\end{definition}
\begin{definition}\label{def:stacking}\label{def:boundary}
Let $(\varphi,x)$ be a configuration with model $\varphi\in\mathcal M(\mathcal H_3)$ and alphabet $A_0$. Assume that $(\varphi,x)$ verifies our standing assumption that neighboring vertices assign the same decoration to a common edge (\Cref{rem:standing-convention}). We say that $(\varphi,x)$ satisfies the \define{stacking conditions} for decorated $2$-blocks if we can find a partition (up to boundaries) of $\varphi$ into well-aligned $2$-cells, in such a way that each of them supports a $2$-block. 
    
In addition, we say that $(\varphi,x)$ satisfies the \textbf{boundary conditions} for decorated $2$-blocks if the decorations on the edges in the boundaries of these $2$-blocks satisfy the following.
    \begin{itemize}
        \item \define{Degree 3 vertices}: every vertex with degree $3$ in the boundary of a decorated $2$-block carries a horizontal arm as in the second and third rows of \Cref{fig:decorated-horizontal-arms-degree-3-left,fig:decorated-horizontal-arms-degree-3-right} (with the white arrow going down). In other words, the symbols in the first row of \Cref{fig:decorated-horizontal-arms-degree-3-left,fig:decorated-horizontal-arms-degree-3-right} cannot appear in these boundaries.
        \item \define{Lower boundary:} the bottom side of the boundary of a decorated $2$-block is decorated uniformly. That is, all edges carry the same kind of arrow. 
    \end{itemize}
\end{definition}

\begin{definition}\label{def:X_0}
    We define $\mathfrak{X}_0$ as the set of all configurations satisfying the stacking and boundary conditions in \Cref{def:boundary}.
\end{definition}
The following can be easily verified from the definition. 
\begin{proposition}$\mathfrak{X}_0$ is an SFT.
\end{proposition}
We will now proceed to define decorated $2^n$-blocks for $\mathfrak X_0$, and prove that configurations are made by stacking well-aligned decorated $2^n$-blocks for all $n\geq 1$ (\Cref{thm:hierarchical-structure-X_0}).

\begin{figure}[]
    \centering
    \includegraphics[page=1,width=1\textwidth]{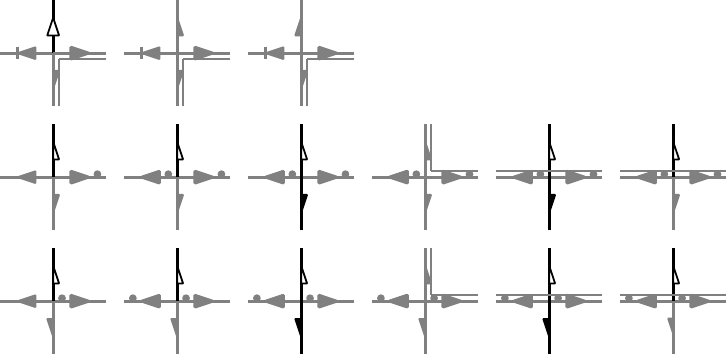}
    \caption{Decorated crosses. They will go on the left half of a decorated $2^n$-block. Those in the first row face down, and the rest face up.}
    \label{fig:decorated-crosses-left}
\end{figure}
\begin{figure}[]
    \centering
    \includegraphics[page=1,width=1\textwidth]{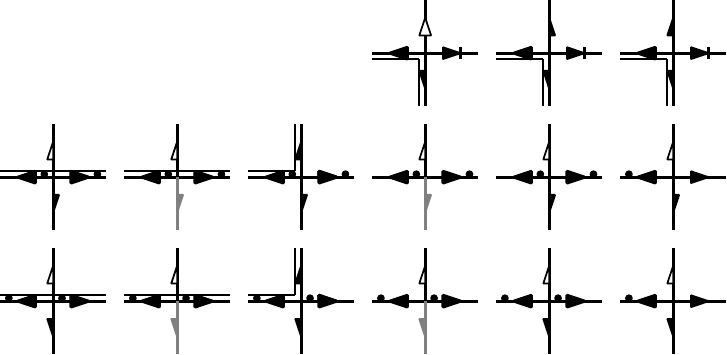}
    \caption{Decorated crosses. They will go in the right half of a decorated $2^n$-block. Those in the first row face down, and the rest face up.}
    \label{fig:decorated-crosses-right}
\end{figure}
\begin{figure}[]
    \centering
    \includegraphics[page=1,width=0.99\textwidth]{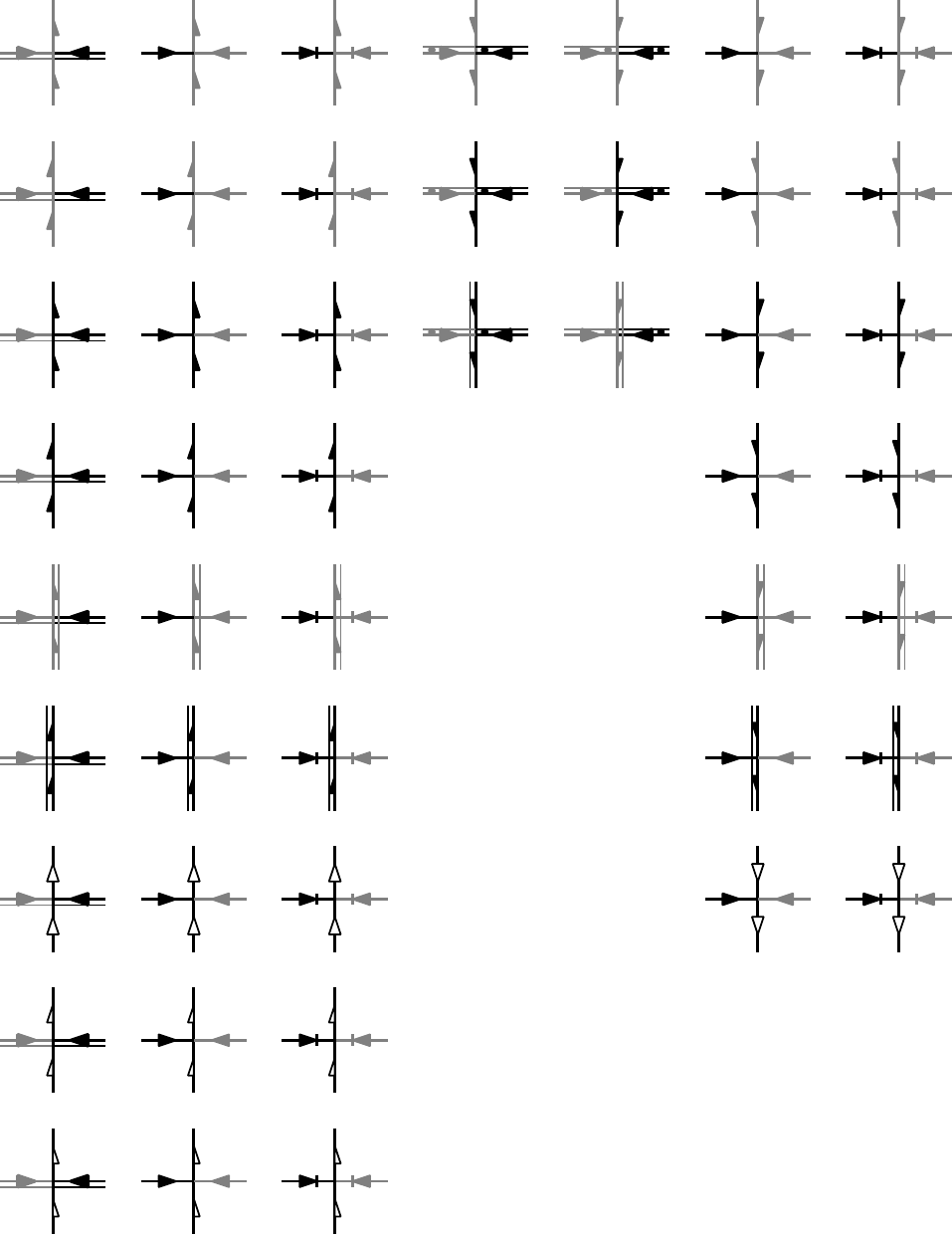}
    \caption{Decorated vertical arms.}
    \label{fig:decorated-vertical-arms}
\end{figure}
\begin{figure}[]
    \centering
    \includegraphics[page=1,width=0.8\textwidth]{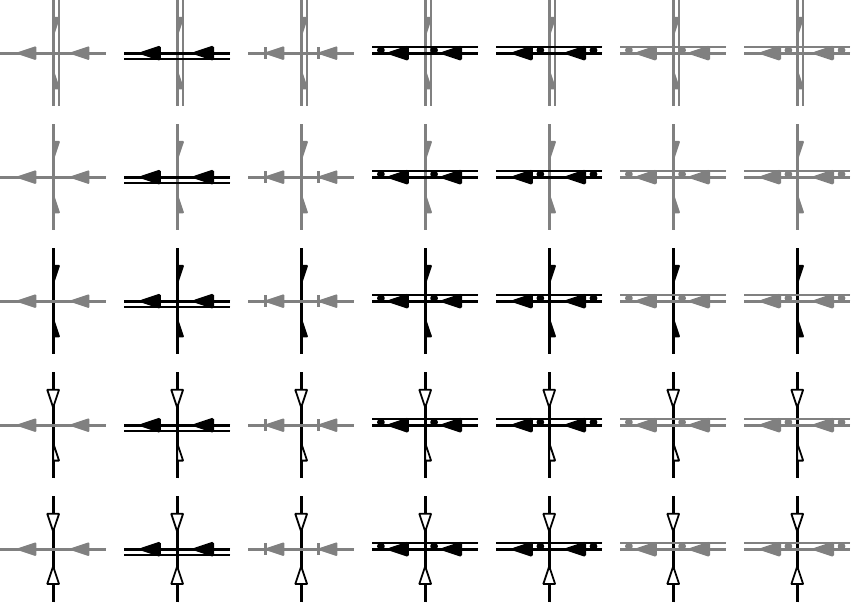}
    \caption{Decorated horizontal arms (degree 4 vertices).}
    \label{fig:decorated-horizontal-arms-left}
\end{figure}
\begin{figure}[]
    \centering
    \includegraphics[page=1,width=0.8\textwidth]{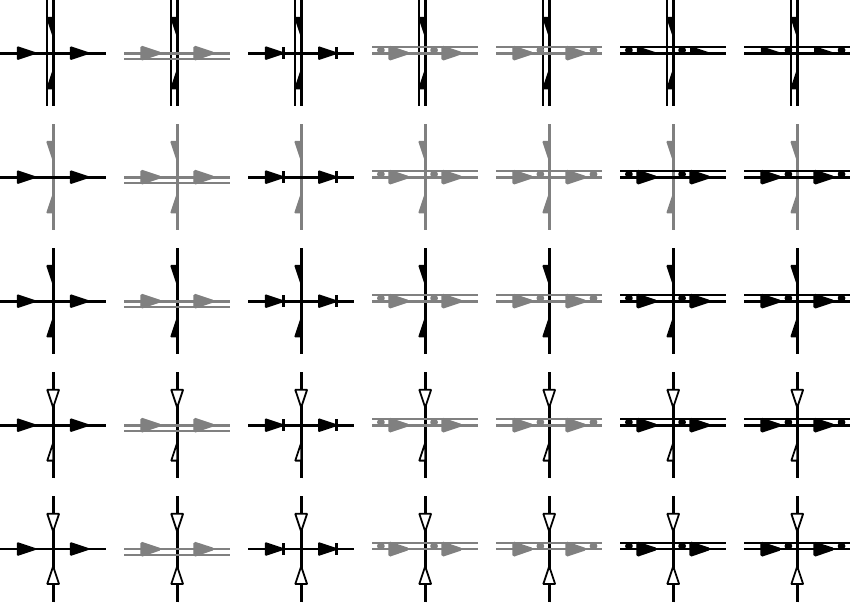}
    \caption{Decorated horizontal arms (degree 4 vertices).}
    \label{fig:decorated-horizontal-arms-right}
\end{figure}
\begin{figure}[]
    \centering
    \includegraphics[page=1,width=0.8\textwidth]{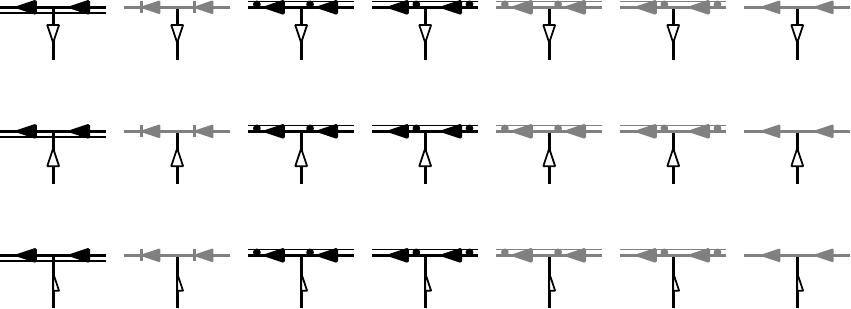}
    \caption{Decorated horizontal arms (degree 3 vertices).}
    \label{fig:decorated-horizontal-arms-degree-3-left}
\end{figure}
\begin{figure}[]
    \centering
    \includegraphics[page=1,width=0.8\textwidth]{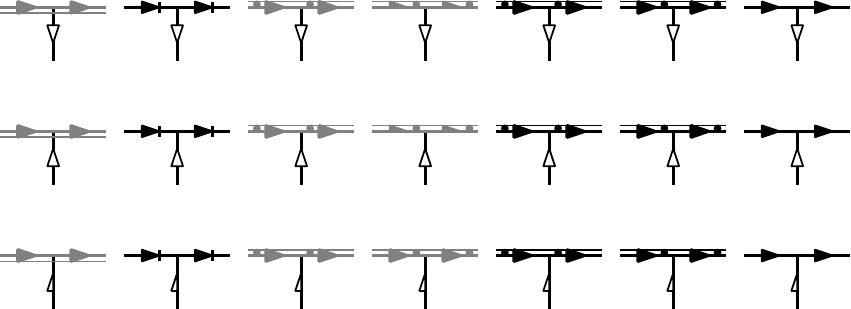}
    \caption{Decorated horizontal arms (degree 3 vertices).}
    \label{fig:decorated-horizontal-arms-degree-3-right}
\end{figure}
\begin{figure}
    \centering
    \includegraphics[page=1,width=0.8\textwidth]{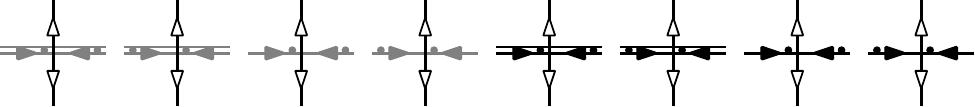}
    \caption{Decorated special arms. We use the word special because they are not analogous to arms in $\mathcal R$. Special arms will be created by neighboring decorated $2^n$-blocks in the lower half of a decorated $2^{n+1}$-block (defined below).}
    \label{fig:decorated-special-arms}
\end{figure}
\begin{figure}[]
    \centering
    \includegraphics[page=1,width=0.4\textwidth]{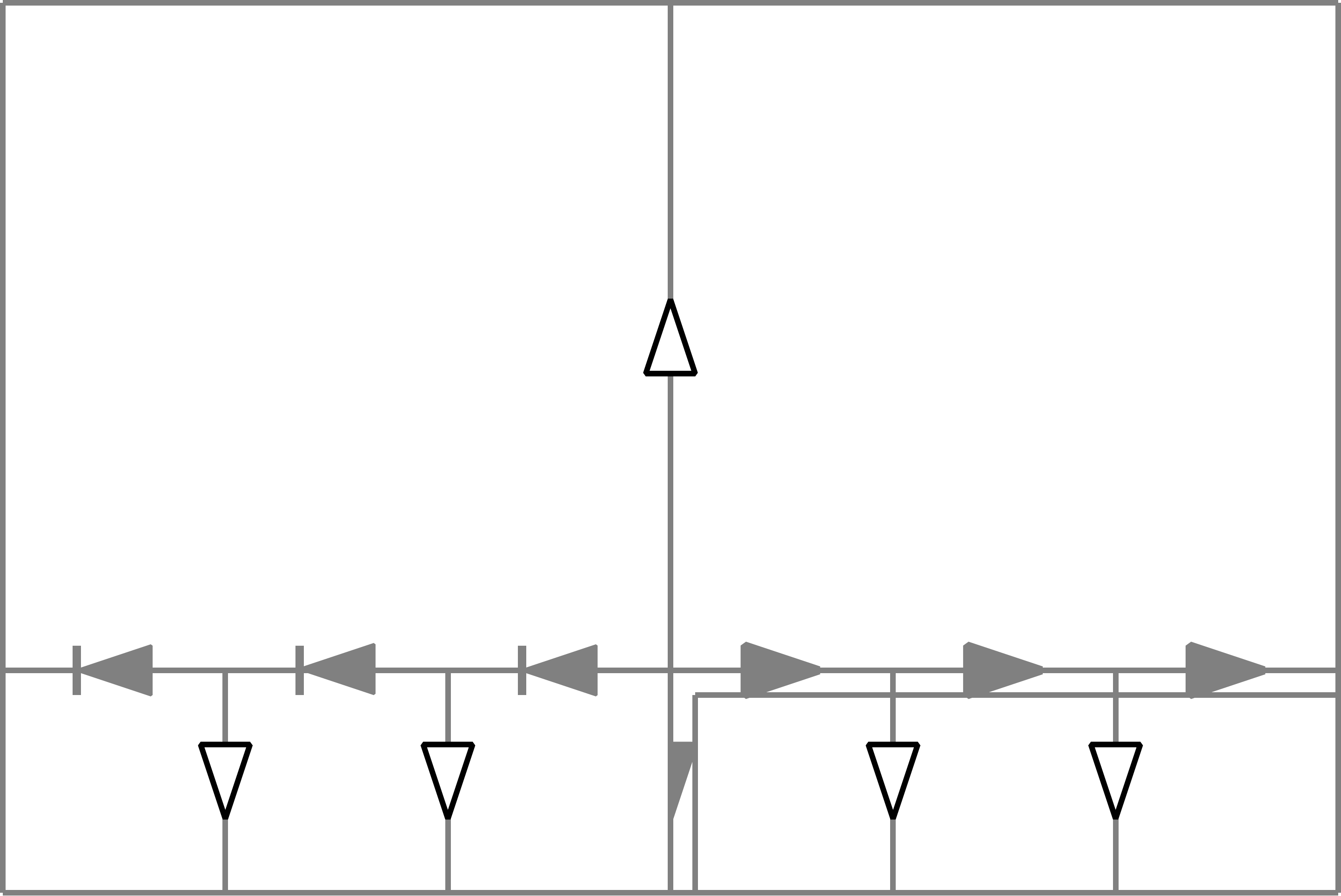}
    \caption{A decorated $2$-block. Its central cross is the leftmost in the first row of \Cref{fig:decorated-crosses-left}.}
    \label{fig:decorated-2-block-example}
\end{figure}
\subsection{Decorated blocks for \texorpdfstring{$\mathfrak{X}_0$}{X0}}\label{subsec:decorated-blocks} We will now define \define{decorated $2^n$-blocks} ($n\geq 1$), which are patterns with support a $2^n$-cell. We will also define a property of them called \define{upper characteristic}. We declare all $2$-blocks to have the same upper characteristic. For $n>1$, the upper characteristic of a decorated $2^n$-block is the ordered tuple of all vertical arrows reaching its upper boundary from below, excluding the one coming from the central cross. We remark that each such vertical arrow comes from the central cross of a smaller decorated block. Thus the upper characteristic of a decorated block depends on certain choices made for smaller blocks inside it.

We now proceed with a recursive definition of decorated $2^n$-blocks.  Fix $n\geq 1$ and suppose that we have already defined decorated $2^n$-blocks. We define a decorated $2^{n+1}$-block as any pattern with support a $2^{n+1}$-cell which can be obtained by the following procedure. 

We start by decomposing a $2^{n+1}$-cell into $2^n$-cells in the natural manner (see \Cref{fig:overlaps}). We obtain two rows of $2^n$-cells, the upper row having $2$ of them and the lower one having $2\cdot 3^{2^n}$ of them. We will now decorate the set of edges in our $2^{n+1}$-cell that belong to the boundary of some $2^n$-cell inside it. The shape of this collection is a sort of cross (with height $2^{n+1}$), plus $2\cdot 3^{2^{n}}-2$ vertical strips of edges of height $2^n$; see \Cref{fig:4-square-skeleton}. We choose a cross from \Cref{fig:decorated-crosses-left} or \Cref{fig:decorated-crosses-right}, we place it in the central vertex of our $2^{n+1}$-cell, and we propagate its arrows until they reach the boundary. The vertical strips will carry decorated special arms at their middle vertices (\Cref{fig:decorated-special-arms}). For now, we only place the corresponding vertical white arrow, and propagate it vertically. We illustrate these steps in Figures \ref{fig:4-square-skeleton}, \ref{fig:boundaries-2-cells-4-cell-white-arrows}, and \ref{fig:boundaries-2-cells-4-cell-white-arrows-zoom} for $n=1$.
\begin{figure}[]
    \centering
    \includegraphics[page=1,width=0.9\textwidth]{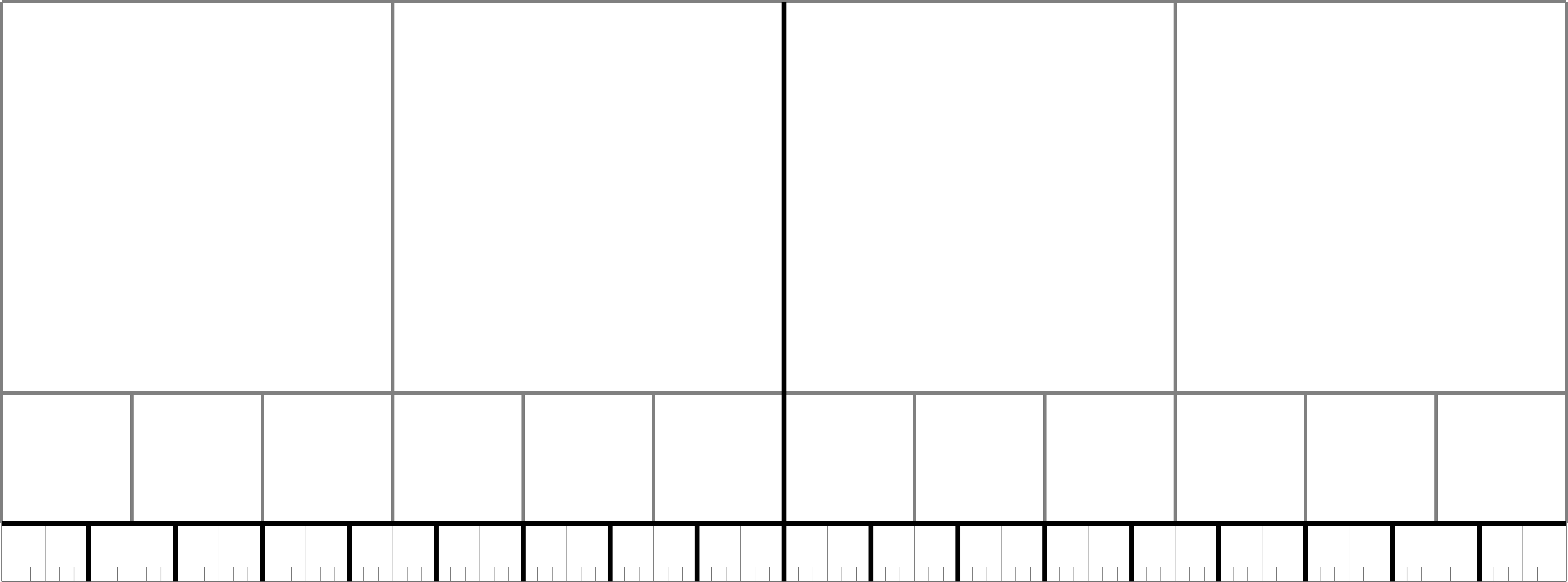}
    \caption{Edges in a $4$-cell that belong to the boundary of some $2$-cell in its natural decomposition.} 
    \label{fig:4-square-skeleton}
\end{figure}
\begin{figure}[]
    \centering
    \includegraphics[page=1,width=0.9\textwidth]{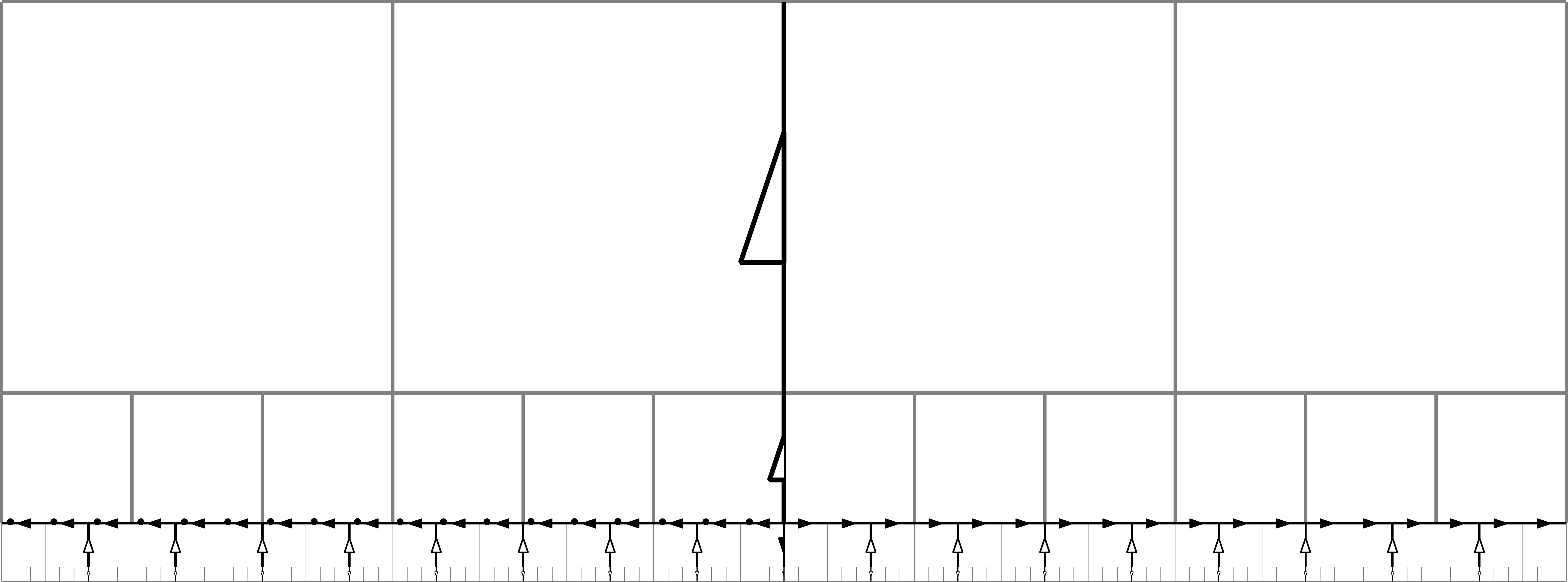}
    \caption{Decoration of the highlighted edges from \Cref{fig:4-square-skeleton}. Here we used the central cross at the bottom right of 
    \Cref{fig:decorated-crosses-right}. See \Cref{fig:boundaries-2-cells-4-cell-white-arrows-zoom} for a detail of the lower half.
    } 
    \label{fig:boundaries-2-cells-4-cell-white-arrows}
\end{figure}
\begin{figure}[]
    \centering
    \includegraphics[page=1,width=0.7\textwidth]{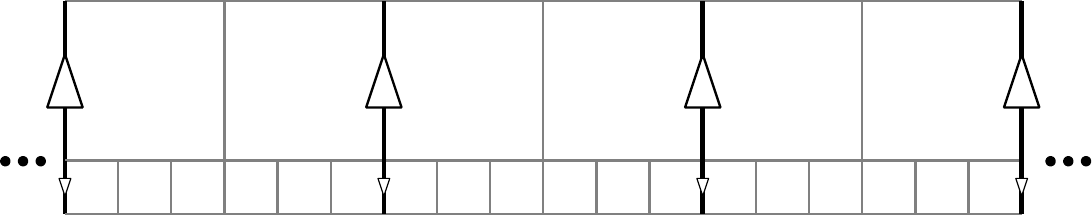}
    \caption{Detail of the lower half of \Cref{fig:boundaries-2-cells-4-cell-white-arrows}.} 
    \label{fig:boundaries-2-cells-4-cell-white-arrows-zoom}
\end{figure}

We decomposed our $2^{n+1}$-cell into $2^n$-cells, and we have chosen decorations for those edges in our $2^{n+1}$-cell that are not in the smaller $2^n$-cells. Next we decorate these $2^{n}$-cells by placing a $2^n$-block in each of them. Their central crosses are chosen as follows. The easy part is the upper half, which has just $2$ of them. On the left we choose a central cross facing down and right  (\Cref{fig:decorated-crosses-left}), and on the right we choose a central cross facing down and left (\Cref{fig:decorated-crosses-right}). In the lower half we need to choose $2\cdot 3^{2^n}$ such crosses among those facing up in \Cref{fig:decorated-crosses-left} and \Cref{fig:decorated-crosses-right}. For the first $3^{2^n}$ of them we require that they are gray (\Cref{fig:decorated-crosses-left}), and that the one in the central position is a cross carrying a corner of double arrows. For the next $3^{2^n}$ of them, we require that they are black (\Cref{fig:decorated-crosses-right}), and that the one in the central position is a cross carrying a corner of double arrows. For the remaining requirements it is convenient to use  finite state automata.  Omitting the dot decorations, a valid choice must be accepted by the finite state automaton in \Cref{fig:FSA-lower-half-simplified}. Including the dot decorations, a valid choice must be accepted either by the finite state automaton in \Cref{fig:FSA-lower-half-left}, or by the one in \Cref{fig:FSA-lower-half-right}. 
\begin{figure}[]
    \centering
    \includegraphics[page=1,width=0.96\textwidth]{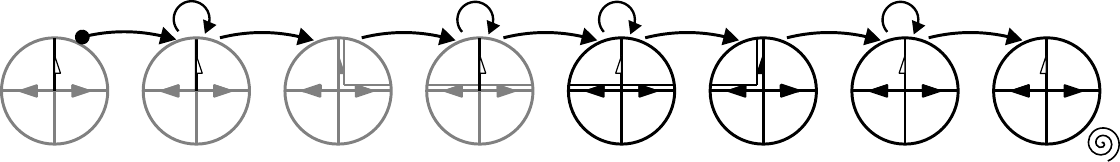}
    \caption{Omitting dot decorations, the sequence of $2\cdot 3^{2^n}$ crosses that we need to choose, understood as a word of length $2\cdot 3^{2^n}$, must be accepted by this finite state automaton. The initial state is the leftmost, and the final state is the rightmost (labeled with a spiral).} 
    \label{fig:FSA-lower-half-simplified}
\end{figure}
\begin{figure}[]
    \centering
    \includegraphics[page=1,width=0.9\textwidth]{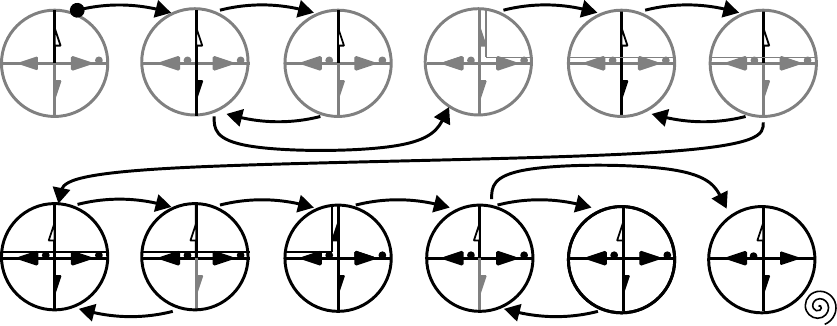}
    \caption{The sequence of $2\cdot 3^{2^n}$ crosses that we need to choose, understood as a word of length $2\cdot 3^{2^n}$, must be accepted by this finite state automaton, or the one in \Cref{fig:FSA-lower-half-right}. The initial state is the one at the top left, and the final state is labeled with a spiral.} 
    \label{fig:FSA-lower-half-left}
\end{figure}
\begin{figure}[]
    \centering
    \includegraphics[page=1,width=0.9\textwidth]{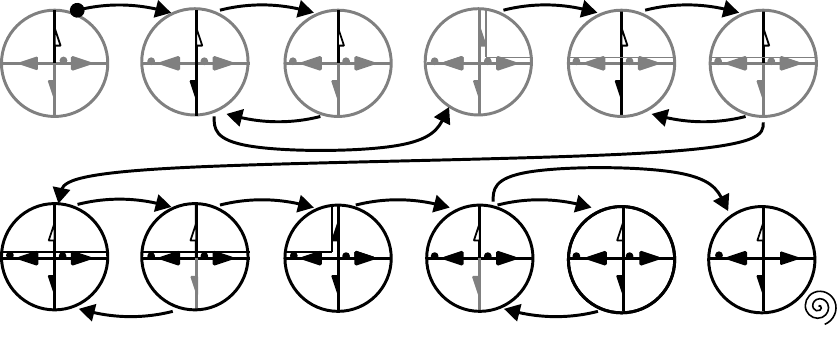}
    \caption{The sequence of $2\cdot 3^{2^n}$ crosses that we need to choose, understood as a word of length $2\cdot 3^{2^n}$, must be accepted by this finite state automaton, or the one in \Cref{fig:FSA-lower-half-left}. The initial state is the one at the top left, and the final state is labeled with a spiral.} 
    \label{fig:FSA-lower-half-right}
\end{figure}

In order to finish the definition of our $2^{n+1}$-block, it only remains to choose the upper characteristic of the $2^n$-blocks inside it. For those in the upper half, we declare any choice to be valid. For a $2^n$-block in the lower half, there is a unique choice that ensures that every vertex in its upper boundary carries an arm as in Figures \ref{fig:decorated-horizontal-arms-left}, \ref{fig:decorated-horizontal-arms-right}, \ref{fig:decorated-horizontal-arms-degree-3-left}, and \ref{fig:decorated-horizontal-arms-degree-3-right}. 

At this point we have finished the definition of decorated $2^{n+1}$-block. We display two decorated $4$-blocks  in  \Cref{fig:decorated-4-block} and \Cref{fig:decorated-4-block-alternative}.
\begin{remark}
Let us emphasize the following asymmetry between the upper and lower boundaries of a decorated $2^n$-block. One can verify that the pattern of arrows that we can see reaching the lower boundary is fully determined by the central cross. For this reason we do not bother to define ``lower characteristic''. In the case of the upper boundary, however, the pattern of arrows that we can see is not constrained by the central cross, and there are multiple possibilities. This flexibility is indispensable, because the upper characteristic of a decorated $2^n$-block is tied to the position it can occupy below another decorated $2^n$-block. Recall that the number of such positions grows with $n$, and it equals $3^{2^n}$. 
\end{remark}
\begin{figure}[]
    \centering
    \includegraphics[page=1,width=0.96\textwidth]{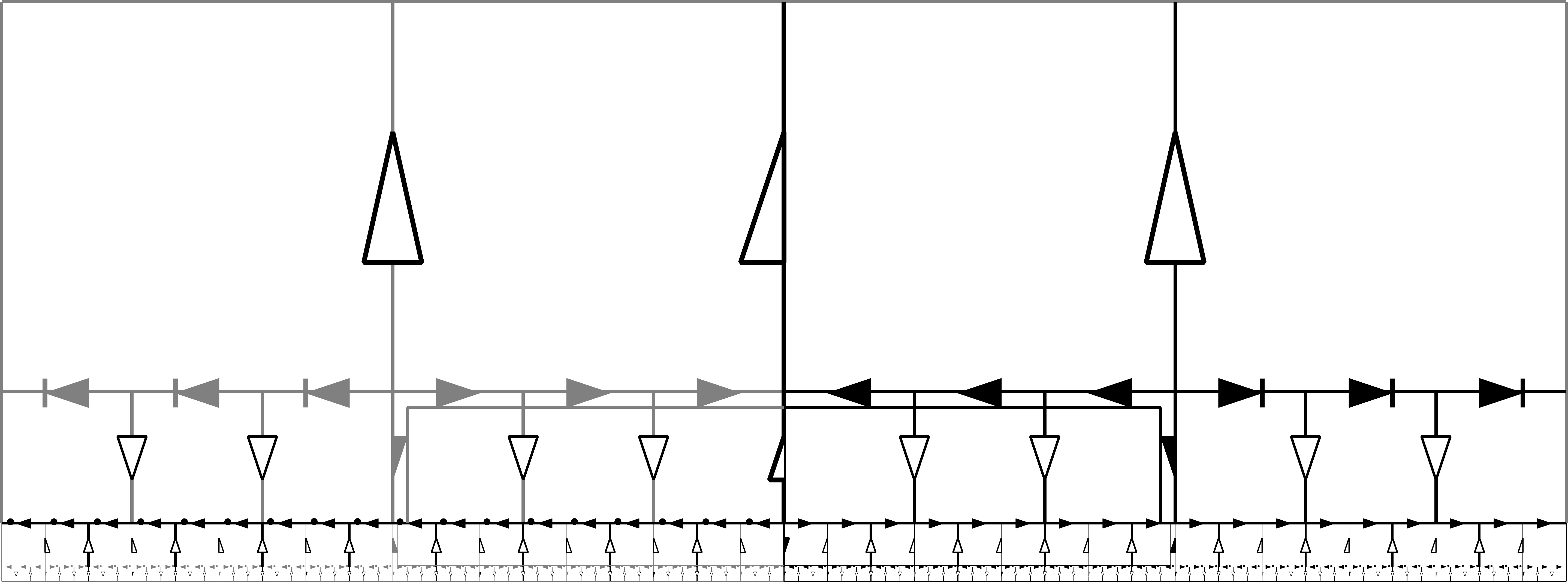}
    \caption{A decorated $4$-block in $\mathfrak X_0$, with some arrows or lines resized for visual purposes. Its central cross ensures that it lives inside the lower half of a decorated $8$-block.} 
    \label{fig:decorated-4-block}
\end{figure}
\begin{figure}[]
    \centering
    \includegraphics[page=1,width=0.96\textwidth]{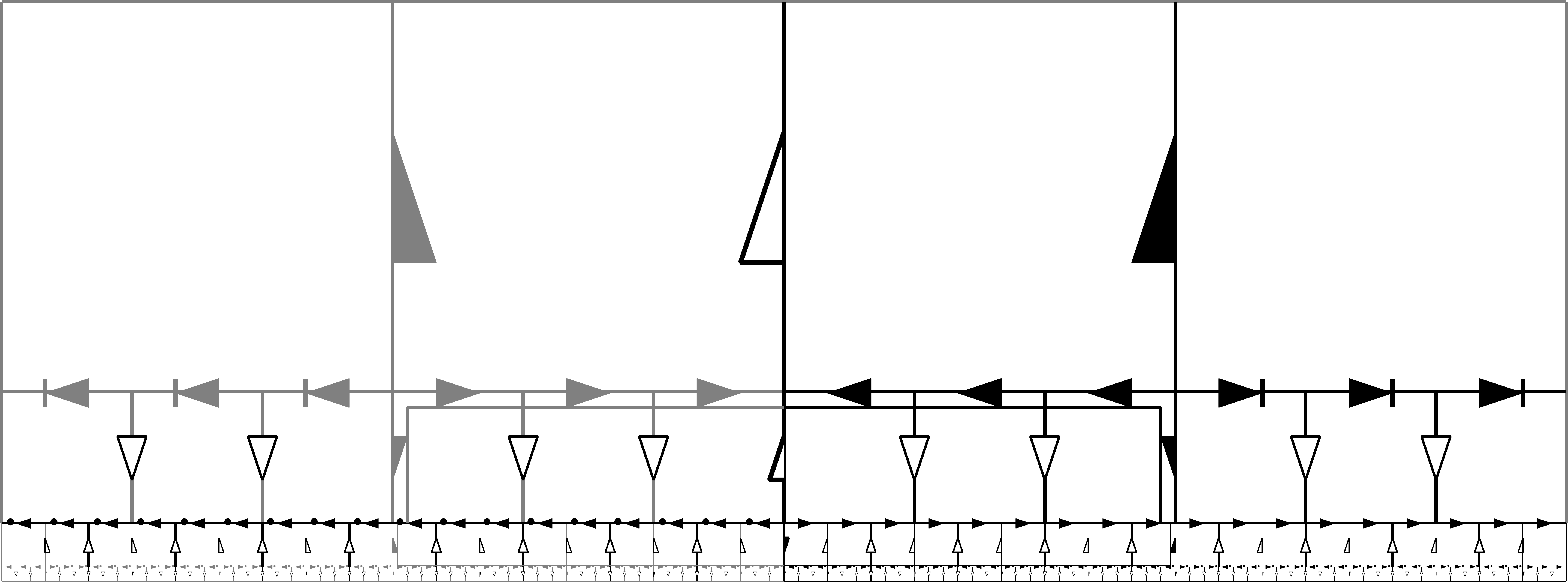}
    \caption{A decorated $4$-block in $\mathfrak X_0$, with some arrows or lines resized for visual purposes. It has the same central cross as the one in \Cref{fig:decorated-4-block}, but its upper characteristic is different. }
    \label{fig:decorated-4-block-alternative}
\end{figure}
\begin{figure}[]
    \centering
    \includegraphics[page=1,width=0.96\textwidth]{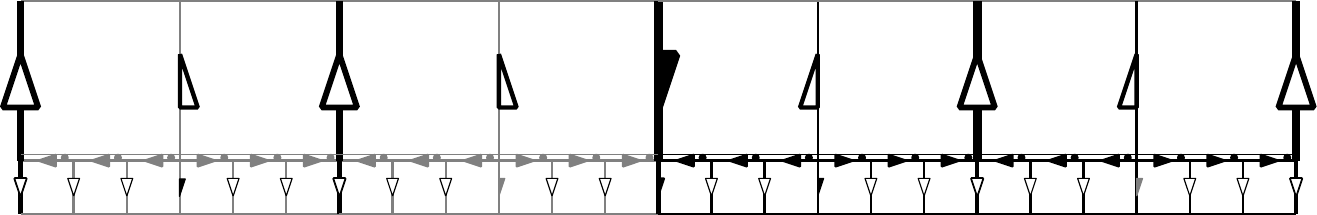}
    \caption{Detail of the central four $2$-blocks in the lower half of the decorated $4$-block from \Cref{fig:decorated-4-block}. The transition between gray and black indicates the center of this horizontal row, where we have a vertical arm as in the fifth column of \Cref{fig:decorated-vertical-arms}. Some arrows or lines have been resized for visual reasons.} 
    \label{fig:decorated-4-block-detail}
\end{figure}
Observe that by construction, decorated $2$-blocks inside a decorated $2^{n+1}$-block satisfy both the stacking and boundary properties. By a standard compactness argument (or K\"onig's infinity lemma), the next result follows. 
\begin{proposition}\label{prop:X_0-is-nonempty}
$\mathfrak{X}_0$ is nonempty.
\end{proposition}
\begin{remark} Not all decorated $2^n$-blocks in our definition can be extended to global configurations. This depends on the chosen upper characteristic. Filtering out the ``bad'' upper characteristics requires some care, and for this reason we do not make it part of the definition of decorated block. For a given central cross and $n\geq 1$, it is not difficult to (computably) characterize those upper characteristics for which the corresponding decorated $2^n$-block can be extended to a configuration in $\mathfrak X_0$. 
\end{remark}
\subsection{The hierarchical structure in \texorpdfstring{$\mathfrak{X}_0$}{X0}}
In order to better understand the  hierarchical structure of configurations in $\mathfrak{X}_0$, we need the following definition.
\begin{definition}\label{def:stacking-boundary-for-n}
Let $n>1$. The \define{stacking} and \define{boundary} conditions for decorated $2^n$-blocks are obtained by replacing $2$-blocks by $2^n$-blocks, and $2$-cells by $2^n$-cells, in Definition \ref{def:stacking}.
\end{definition}
We will now prove \Cref{thm:hierarchical-structure-X_0}, our main result about the hierarchical structure of $\mathfrak{X}_0$. It states that every configuration in $\mathfrak{X}_0$ satisfies the stacking and boundary conditions for decorated $2^n$-blocks, for all $n\geq 1$. 
\begin{theorem}
\label{thm:hierarchical-structure-X_0}%
    Let $(\varphi,x)\in\mathfrak{X}_0$. Then $2^n$-blocks in this configuration satisfy stacking and boundary conditions for all $n\geq 1$. 
\end{theorem}
\begin{proof}
Fix $(\varphi,x)$ in $\mathfrak{X}_0$. We prove by induction on $n$ that stacking and boundary conditions are satisfied for decorated $2^n$-blocks in $x$. The base case $n=1$ is just the definition of $\mathfrak{X}_0$. Now fix $n\geq1$, and assume that stacking and boundary conditions for decorated $2^n$-blocks are satisfied. Our goal is proving that stacking and boundary conditions for decorated $2^{n+1}$-blocks are also satisfied.

The general mechanism is as follows. First observe that thanks to stacking conditions, we can speak of infinite horizontal rows of decorated $2^n$-blocks in $x$ (from now on we may not mention $x$). Now, consider two decorated $2^n$-blocks that are horizontal neighbors. Then the arrows emitted by their central crosses must meet somewhere in their shared boundary. Next, boundary conditions ensure that an arm is created. This process is illustrated in \Cref{fig:dramatization}. In this manner, the list of possible arms creates adjacency rules for our decorated $2^n$-blocks. A similar remark applies to a decorated $2^n$-block and its vertical neighbor that lies beneath, right in the middle. 

\begin{figure}[]
    \centering
    \includegraphics[page=1,width=0.8\textwidth]{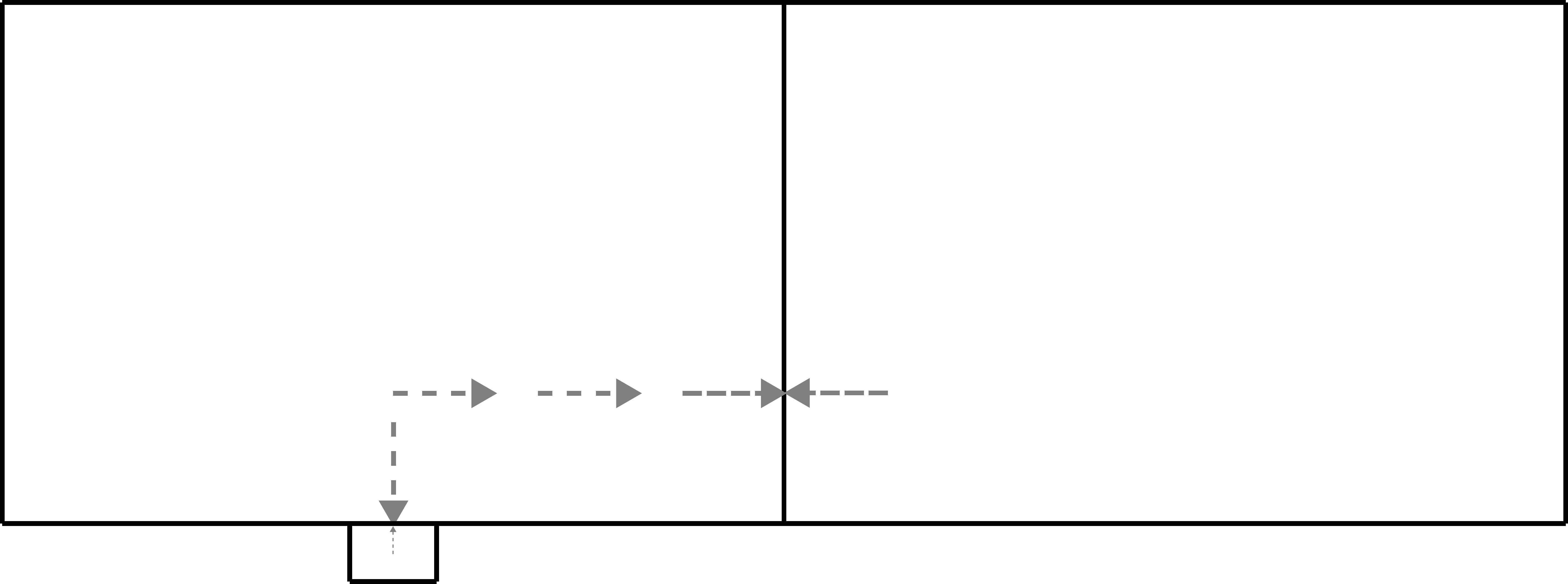}
    \caption{Simplified representation of a decorated $2^n$-block (top left), together with a horizontal and vertical neighbor. The key is that central crosses ``emit'' arrows that travel until they reach the boundaries. Because boundary conditions are satisfied by inductive hypothesis, the vertex where these arrows meet can only carry an arm. From the list of possible arms, we obtain constraints for the kind of central crosses that can be carried by neighboring decorated $2^n$-blocks.}
    \label{fig:dramatization}
\end{figure}

\begin{figure}[]
    \centering
    \includegraphics[page=1,width=0.2\textwidth]{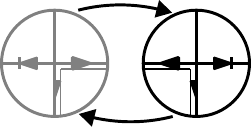}
    \caption{
    Vertex presentation of a $\ZZ$-subshift of finite type on 2 symbols. It describes the horizontal compatibility of decorated $2^n$-blocks whose central crosses face down. Decorations in the vertical arrow going up are omitted.}
    \label{fig:SFT-upper-half}
\end{figure}
\begin{figure}[]
    \centering
    \includegraphics[page=1,width=0.8\textwidth]{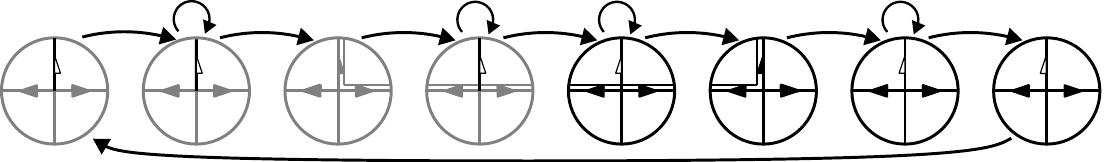}
    \caption{
    Vertex presentation of a $\ZZ$-subshift of finite type on 8 symbols. It describes the horizontal compatibility of decorated $2^n$-blocks whose central cross faces up. Decorations in the vertical arrow going down and dot decorations are omitted.}
    \label{fig:SFT-lower-half-simplified}
\end{figure}
\begin{figure}[]
    \centering
    \includegraphics[page=1,width=0.8\textwidth]{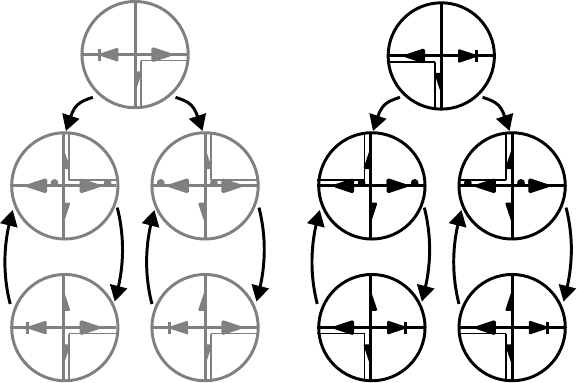}
    \caption{
    Vertex presentation of a $\mathbb{N}$-SFT on 10 symbols. It describes adjacency constraints for decorated $2^n$-blocks that are vertical neighbors (when we move to the vertical neighbor below, in the middle position) and carry a cross with a corner of double arrows. The top symbol has the decoration on its upwards arrow omitted, as the three possibilities are allowed (see \Cref{fig:decorated-crosses-left} and \Cref{fig:decorated-crosses-right}).} 
    \label{fig:SFT-vertical-column}
\end{figure}

\begin{figure}[]
    \centering
    \includegraphics[page=1,width=0.95\textwidth]{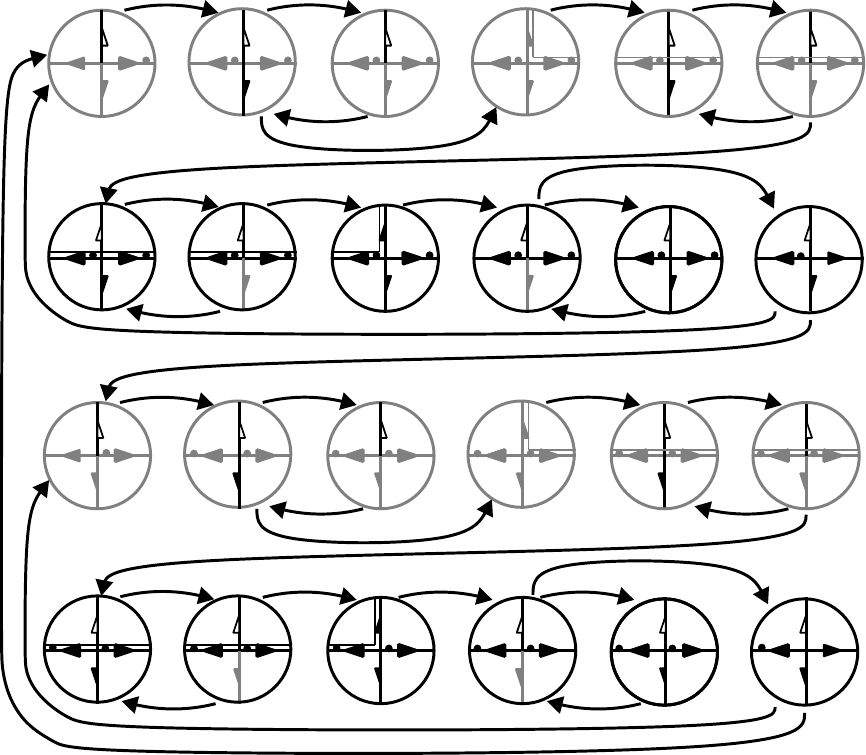}
    \caption{
    Vertex presentation of a $\ZZ$-subshift of finite type on 24 symbols. It describes the horizontal compatibility of decorated $2^n$-blocks whose central cross faces up.}
    \label{fig:SFT-lower-half-rich}
\end{figure}

By the general mechanism explained in the previous paragraph, from the list of crosses and arms we can make conclusions about the global arrangement of decorated $2^n$-blocks in our arbitrary configuration. For instance, observe that horizontal arrows emitted by crosses facing down are different from those emitted by crosses facing up. Since vertical arms do not mix up these kinds of arrows, we have proved the following. 
\begin{claim}\label{claim:horizontal-rows-have-constant-crosses}
     Consider an infinite horizontal row of decorated $2^n$-blocks. Then either all their central crosses face up, or all their central crosses face down. 
\end{claim}
Reading the list of vertical arms, we can obtain a more precise description of the possible central crosses that we can see in an infinite horizontal row of decorated $2^n$-blocks. It is convenient to represent these constraints by a $\ZZ$-subshift of finite type. 
\begin{claim}\label{claim:SFT-upper-half}
     Consider an infinite horizontal row of decorated $2^n$-blocks whose central crosses face down. This sequence of central crosses must belong to the $\ZZ$-SFT in \Cref{fig:SFT-upper-half}.
\end{claim}
\begin{claim}\label{claim:SFT-lower-half-simplified}
     Consider an infinite horizontal row of decorated $2^n$-blocks whose central crosses face up. This sequence of central crosses must belong to the $\ZZ$-SFT in \Cref{fig:SFT-lower-half-simplified}, provided we omit  decorations in the vertical arrows going down and dot decorations.
     \end{claim}
In order to make a more precise statement about an infinite horizontal row of decorated $2^n$-blocks facing up, we need to understand its interaction with the infinite horizontal row of decorated $2^n$-blocks above it (which will face down). Indeed, if we solely study the horizontal constraints for an infinite horizontal row of decorated $2^n$-blocks facing up, it might seem that a constant sequence is possible (for instance, crosses from \Cref{fig:decorated-crosses-left} and \Cref{fig:decorated-crosses-right} that carry a loop in \Cref{fig:SFT-lower-half-simplified}). As we shall see, rows of decorated $2^n$-blocks facing down will prevent this behavior. 
\begin{claim}\label{claim:SFT-vertical-column}
    Below a decorated $2^n$-block facing down (resp. up), right in the middle position, we have a decorated $2^n$-block facing up (resp. down). Furthermore, if we start at an arbitrary decorated $2^n$-block facing down, and we consider the infinite vertical column of decorated $2^n$-blocks obtained by always moving to the vertical neighbor below in the middle position, then the corresponding sequence must belong to the $\NN$-SFT in \Cref{fig:SFT-vertical-column}.  
\end{claim}
Observe that the infinite vertical column mentioned in this statement is well-defined by the vertical alignment condition. Both assertions in the statement follow by direct examination of our allowed horizontal arms (\Cref{fig:decorated-horizontal-arms-left} and \Cref{fig:decorated-horizontal-arms-right}).

\begin{claim}\label{claim:alternation-gray-black-lower-crosses}
     Consider an infinite horizontal row of decorated $2^n$-blocks with central crosses facing up. Then the vertical arrows that they emit downwards alternate periodically between black and gray. 
\end{claim}
To prove this claim, let $B_l$ be a decorated $2^n$-block whose central cross faces up, and let $B_r$ be its right horizontal neighbor. Let $B_l'$ be the decorated $2^n$-block below $B_l$ in the middle position, and let $B_r'$ be the decorated $2^n$-block below $B_r$ in the middle position. Then $B_l$ (resp. $B_r$) emits downwards an arrow which is either black or gray. Because of the possible horizontal arms (see the second and third rows in \Cref{fig:decorated-horizontal-arms-left} and \Cref{fig:decorated-horizontal-arms-right}), this arrow meets an upwards arrow of the same color, coming from the central cross of $B_l'$ (resp. $B_r'$). The color of this arrow only depends on whether $B_l'$ (resp. $B_r'$) faces  left or right (see the first row of \Cref{fig:decorated-crosses-left} and \Cref{fig:decorated-crosses-right}). As a consequence, in order to prove the statement, it suffices to show that $B_l'$ and $B_r'$ face in different horizontal directions. An elementary counting argument shows that $B_l'$ and $B_r'$ are separated by an even number of decorated $2^n$-blocks (exactly $3^{2^{n}}-1$), so our claim follows from \Cref{claim:SFT-upper-half}.

We can now upgrade \Cref{claim:SFT-lower-half-simplified} to a more detailed statement. Specifically, combining \Cref{claim:alternation-gray-black-lower-crosses} with \Cref{claim:SFT-lower-half-simplified}, we obtain the following. 
\begin{claim}\label{claim:SFT-lower-half-rich}
     Consider an infinite horizontal row of decorated $2^n$-blocks whose central crosses face up. Then the corresponding sequence of central crosses must belong to the $\ZZ$-SFT in \Cref{fig:SFT-lower-half-rich}.
\end{claim}
We are now ready to prove the following. 
\begin{claim}\label{claim:small-blocks-form-large-blocks}
    Every decorated $2^{n}$-block is part of a unique larger decorated $2^{n+1}$-block. 
\end{claim}
It is clear that two decorated $2^{n+1}$-blocks cannot share a decorated $2^n$-block, so the uniqueness part of the statement comes ``for free''. We only need to prove that every decorated $2^n$-block is part of a decorated $2^{n+1}$-block. Thanks to the previously established properties, it is sufficient to prove our claim for the special case of a decorated $2^n$-block whose central cross faces down and right. Let us pick such a block, and denote it by $B_l$. Let $B_r$ be the horizontal neighbor of $B_l$ to its right. Let $B_l'$ be the vertical neighbor below $B_l$ in the middle, and let $B_r'$ be the vertical neighbor of $B_r$ below in the middle. By \Cref{claim:SFT-vertical-column}, we know that $B_l'$ and $B_r'$ carry a central cross with a corner of double arrows, and the only information which is not yet determined is the kind of dot decoration they have (which determines the decoration of the arrow going downwards). By inductive hypothesis, we know that all arrows in the bottom boundary of $B_l$ (resp. $B_r$) go in the same direction. The horizontal arm created by the central crosses of $B_l$ and $B_l'$ (resp. $B_r$ and $B_r'$) ensures that they all go left (resp. right), see the first row in \Cref{fig:decorated-horizontal-arms-left} and \Cref{fig:decorated-horizontal-arms-right}. 

We claim that all vertical neighbors below  $B_l$ (resp. $B_r$) different from $B_l'$ (resp. $B_r'$) carry a gray central cross (resp. black). Indeed, let $B'$ be such a decorated $2^n$-block. The color of the central cross of $B'$ is in correspondence with the kind of white arrow that it emits upwards (see \Cref{fig:decorated-crosses-left} and \Cref{fig:decorated-crosses-right}). This white arrow will eventually meet the lower boundary of $B_l$ (resp. $B_r$) and create an arm, either at a vertex of degree 3 or 4 (see rows 4 of  \Cref{fig:decorated-horizontal-arms-left} and \Cref{fig:decorated-horizontal-arms-right}). It follows that the kind of white arrow is in correspondence with the direction of horizontal arrows in this arm, which by the observation in the previous paragraph, depends only on whether $B'$ is below $B_l$ or $B_r$. Thus our claim holds.

Next, observe that the properties we proved until now are also valid for the pair of decorated $2^n$-blocks at the left and right of $B_l$ and $B_r$. In particular, the left neighbor of the leftmost $2^n$-block below $B_l$ carries a black central cross, and the right neighbor of the rightmost $2^n$-block below $B_r$ carries a gray central cross. Combining this observation with \Cref{fig:SFT-lower-half-simplified}, we can already conclude that the only way to choose the central crosses of the decorated $2^n$-blocks below $B_l$ and $B_r$ is exactly as described in the simplified finite state automaton from \Cref{fig:FSA-lower-half-simplified}, including the initial and final states. Having established this, \Cref{claim:SFT-lower-half-rich} shows that the only way to complete the missing decorations matches the definition of decorated $2^{n+1}$-block (see \Cref{fig:FSA-lower-half-left} and \Cref{fig:FSA-lower-half-right}). 

The transition from gray to black decorated $2^n$-blocks in the middle of the lower half of our tentative decorated $2^{n+1}$-block creates a vertical arm that transmits arrows going down, as in the fourth or fifth column of \Cref{fig:decorated-vertical-arms}. At this point we can conclude from the boundary conditions that our tentative $2^{n+1}$-block indeed carries a cross in its center, whose arrows are propagated until they reach the boundary (this is just Robinson's  argument in \cite[p. 189]{robinson_undecidability_1971}). Next, consider the vertical strip of edges between two arbitrary neighboring decorated $2^n$-blocks of the same color, in the lower half of our tentative $2^{n+1}$-block. The vertex in the middle of this strip carries a special arm, created by the central crosses of the $2^n$-blocks at the left and right. The vertical arrows born in this special arm propagate both upwards and downwards (\Cref{fig:decorated-special-arms}). Thus our tentative decorated $2^{n+1}$-block satisfies all conditions in the definition of decorated $2^{n+1}$-block (\Cref{subsec:decorated-blocks}), and \Cref{claim:small-blocks-form-large-blocks} is proved. 

To finish the proof by induction, let us verify that decorated $2^{n+1}$-blocks in $x$ satisfy the stacking and boundary conditions (see \Cref{def:stacking}, \Cref{def:boundary}, and \Cref{def:stacking-boundary-for-n}). By inductive hypothesis, every edge is either in the support of a unique decorated $2^n$-block, or in the boundary of such a block. Then \Cref{claim:small-blocks-form-large-blocks} shows that the same is true for decorated $2^{n+1}$-blocks. The horizontal alignment property for $2^{n+1}$-blocks follows from \Cref{claim:SFT-upper-half}, and the vertical alignment property follows from \Cref{claim:SFT-vertical-column} plus an elementary parity argument. Next we observe that the bottom side of the boundary of a decorated $2^{n+1}$-block is decorated with arrows in the same direction. This follows from three facts: (i) that decorated $2^{n}$-blocks verify this property by inductive hypothesis, (ii) that for a decorated $2^{n}$-block facing up, the direction of the arrows in its bottom boundary corresponds bijectively with the dot decoration in its central cross (rows 2 and 3 in \Cref{fig:decorated-horizontal-arms-left} and \Cref{fig:decorated-horizontal-arms-right}), and finally, (iii) that all decorated $2^{n}$-blocks in the lower half of a decorated $2^{n+1}$-block carry the same kind of dot decoration (\Cref{fig:FSA-lower-half-left} and \Cref{fig:FSA-lower-half-right}). As a consequence, and since decorated $2^{n}$-blocks satisfy the boundary conditions, the same is true for decorated $2^{n+1}$-blocks.  
\end{proof}

\subsection{Assembling infinite configurations in \texorpdfstring{$\mathfrak{X}_0$}{X0}}\label{subsec:directions}
We will now analyze how configurations in $\mathfrak X_0$ are assembled from blocks starting from the ``origin'' (the empty word in a model). The next results are, in a sense, the $\mathcal H_3$ analogue of the analysis in \cite[\S 8]{robinson_undecidability_1971}.

It is convenient to define \define{directions} in the alphabet $\{\nearrow, \searrow, \swarrow, \nwarrow\}$. The direction of a cross symbol in the alphabet $A_0$ is defined as follows.
\begin{itemize}
    \item $\searrow$ is associated to a cross as in the first row of     \Cref{fig:decorated-crosses-left}. 
    \item $\nearrow$ is associated to a cross as in the second and third rows of     \Cref{fig:decorated-crosses-left}. 
    \item $\swarrow$ is associated to a cross as in the first row of     \Cref{fig:decorated-crosses-right}.
    \item $\nwarrow$ is associated to a cross as in the second and third rows of     \Cref{fig:decorated-crosses-right}.
\end{itemize}
The direction of a decorated $2^n$-block is defined as the direction of its central cross. 

The idea behind this terminology is as follows. Suppose that we have a decorated $2^n$-block with direction $\searrow$, and it appears in some configuration in $\mathfrak X_0$. By \Cref{thm:hierarchical-structure-X_0}, it is part of a larger decorated $2^{n+1}$-block. If we interpret this as having our decorated $2^n$-block \textit{grow} to a decorated $2^{n+1}$-block, then its support grows precisely in the bottom and right direction ($\searrow$). 

Pick a configuration $(\varphi,x)\in \mathfrak X_0$. We define its \define{sequence of directions} 
\[
(d_n)_{n\geq 1}
\]
as follows. We start by finding the minimal  $w\in\supp(\varphi)$ (in the shortlex order) that carries the central cross of a $2$-block, and we let $d_1$ be the direction of this cross. By \Cref{thm:hierarchical-structure-X_0}, this $2$-block appears within a $2^2$-block in the same configuration. We define $d_2$ as the direction of this $2^2$-block. Inductively, we always use a $2^n$-block to define $d_n$, and then we define $d_{n+1}$ as the direction  of the unique $2^{n+1}$-block containing the previously used $2^n$-block. 

In order to better understand what sequences of directions can arise in $\mathfrak X_0$, we need the following elementary observation. 
\begin{proposition}\label{prop:boundary-determination}
    Let $n\geq 1$ and let $B$ be a decorated $2^n$-block whose central cross faces down (resp. up). Let $k\geq 1$ and suppose that $B$ appears within a larger decorated $2^{n+k}$-block, in such a way that the direction of the horizontal arrows in the upper (resp. lower) horizontal  boundary of $B$ is already determined. This information, plus the relative position of $B$ within the larger block,  uniquely determines the central cross of $B$.
\end{proposition}
\begin{proof}
    Suppose first that the central cross of $B$ faces down. Its relative position within the larger block determines most decorations in its central cross. The only nontrivial case is when $B$ lies in the middle position below the $2^n$-block above it. If this is the case, then the upwards arrow in the central cross of $B$ is not white, but two options remain to be determined (\Cref{fig:decorated-crosses-left} and \Cref{fig:decorated-crosses-right}). Inspecting \Cref{fig:decorated-horizontal-arms-left} and \Cref{fig:decorated-horizontal-arms-right}, rows 2 and 3, we see that this choice is determined by the direction of the horizontal arrows in the upper boundary of $B$. 
    
    A similar argument applies if $B$ faces up. In this situation, its position within the larger block determines all decorations in its central cross, with the exception of the dot decoration. But again, this is determined by the direction of the horizontal arrows in its lower boundary, by the arms from \Cref{fig:decorated-horizontal-arms-left} and \Cref{fig:decorated-horizontal-arms-right}, rows 2 and 3. 
\end{proof}

\begin{proposition}\label{prop:prescribe-directions}
    Let $\varphi\in \mathcal M(\mathcal H_3)$ be an arbitrary model, and let $(d_n)_{n\geq 1}$ be an arbitrary sequence of elements in $\{\nearrow, \searrow, \swarrow, \nwarrow\}$ such that every symbol appears infinitely often. Then there exists $(\varphi, x)\in \mathfrak X_0$ whose sequence of directions is exactly $(d_n)_{n\geq 1}$. 
\end{proposition}
\begin{proof}
Fix $\varphi$ and $(d_n)_{n\geq 1}$ as in the statement. In order to find $(\varphi,x)$ whose sequence of directions equals $(d_n)_{n\geq 1}$, we will define three sequences.
\begin{itemize}
    \item A sequence $(p_n)_{n\geq 1}$ of decorated $2^{n}$-blocks. 
    \item A sequence $(E_n)$ of  $2^n$-cells. Each $p_n$ will be supported in $E_n$.
    \item A sequence $(P_n)_{n\geq 1}$ of partitions (up to boundaries) of $\varphi$ into well-aligned $2^n$-cells, see \Cref{subsec:cells}. Each $E_n$ will be a $2^n$-cell from $P_n$. 
\end{itemize}
Here $(E_n)_{n\geq 1}$ and $(P_n)_{n\geq 1}$ are auxiliary elements, but it is convenient to define them before choosing $(p_n)_{n\geq 1}$. 

We choose $(E_n)_{n\geq 1}$ and $(P_n)_{n\geq 1}$ inductively, and simultaneously. We start by choosing $P_1$ as an arbitrary partition of $\varphi$ into well-aligned $2$-cells (see \Cref{subsec:cells}). Inductively, let $n\geq 1$, and suppose that we have already chosen the partition $P_n$. We will simultaneously choose $E_n$ and $P_{n+1}$. 
    \begin{itemize}
        \item If $n=1$, then we let $E_n$ be the $2$-cell from $P_1$ whose central vertex is labeled by the minimal word (in the shortlex order).
        \item If $n>1$ then we let $E_n$ be the unique $2^n$-cell in $P_n$ containing $E_{n-1}$. 
    \end{itemize} 
We will now invoke \Cref{prop:sequence-of-partitions}, which says that there are exactly four ways to choose a partition (up to boundaries) $P_{n+1}$ of $\varphi$ into well-aligned $2^{n+1}$-cells, whose decomposition into $2^n$-cells yields $P_n$. Using this result, we choose $P_{n+1}$ depending on $d_n$. 
    \begin{enumerate}
        \item Case $d_{n}=\searrow$. We choose $P_{n+1}$ to ensure that $E_n$ lives in the upper left half of a $2^{n+1}$-cell from $P_{n+1}$.
        \item Case $d_n=\swarrow$.
        We choose $P_{n+1}$ to ensure that $E_n$ lives in the upper right half of a $2^{n+1}$-cell from $P_{n+1}$.
        \item Case $d_n=\nearrow$. Let $Q_n$ be the unique $2^{n}$-cell in $P_n$ sitting above $E_n$. Then we choose $P_{n+1}$ by the condition that $Q_n$ lives in the upper left half of a $2^{n+1}$-cell from $P_{n+1}$. 
        \item Case $d_n=\nwarrow$. Let $Q_n$ be the unique $2^{n}$-cell in $P_n$ sitting above $E_n$. Then we choose $P_{n+1}$ by the condition that $Q_n$ lives in the upper right half of a $2^{n+1}$-cell from $P_{n+1}$.
    \end{enumerate}
Thus we have defined $(E_n)_{n\geq 1}$ and $(P_n)_{n\geq 1}$. We remark that the supports $(E_n)_{n\geq 1}$ are nested, in the sense that when we decompose $E_{n+1}$ into $2^n$-cells, one of these $2^n$-cells is $E_n$ ($n\geq 1$). 

For each $n\geq 1$ we will define a decorated $2^n$-block $p_n$ supported in $E_n$. We will define them under the constraint that $p_n$ is one of the decorated $2^n$-blocks forming $p_{n+1}$. This constraint already determines most decorations in the central cross for $p_n$, for all $n\geq 1$. Specifically, if $E_n$ belongs to the upper half of $E_{n+1}$, then the central cross of $p_n$ is determined, with the exception of the arrow it emits upwards. If $E_n$ belongs to the lower half of $E_{n+1}$, then most of the central cross of $p_n$ is determined, including the color of its downwards arrow (gray/black), but its dot decoration remains to be determined. For every $n\geq 1$ we have already defined the \textit{direction} of all arrows in $p_n$, even while some decorations for the arrows remain undetermined. Given $n\geq 1$, we can find $m$ large enough so that, looking at $p_n$ within the partially defined pattern $p_{n+m}$, we can determine the direction of all arrows in the boundaries of $p_n$. By \Cref{prop:boundary-determination}, this information determines the central cross of $p_n$. Applying \Cref{prop:boundary-determination} to the smaller blocks within $p_n$, we can determine the central crosses of all the smaller blocks composing it, and in particular its upper characteristic. 

Thus we have defined the sequence $(p_{n})_{n\geq 1}$ of decorated $2^{n}$-blocks. It is nested, in the sense that $p_{n+1}(w)=p_n(w)$ whenever both $p_n$ and $p_{n+1}$ are defined at $w\in\supp(\varphi)$.  Furthermore, for every $w\in\supp(\varphi)$ we have that $p_n(w)$ is defined for sufficiently large $n$. This follows from the assumption that every symbol from $\{\nearrow, \searrow, \swarrow, \nwarrow\}$  appears infinitely often in $(d_n)_{n\geq 1}$. Thus we can define a configuration $x$ by 
\begin{equation}\label{eq:definition-limiting-configuration}
x(w)=p_n(w) \text{ for $n$ large enough, $w\in\supp(\varphi)$.}
\end{equation}
By construction, the configuration $x$ has the property in the statement.
\end{proof}
\begin{proposition}[Computable \Cref{prop:prescribe-directions}]\label{prop:m(X_0)-is-zero}  Let $\varphi\in \mathcal M(\mathcal H_3)$ be a computable model (for instance, the one from \Cref{ex:maingraph}), and let $(d_n)_{n\geq 1}$ be a computable sequence of elements in $\{\nearrow, \searrow, \swarrow, \nwarrow\}$ such that every symbol appears infinitely often. Then there exists a computable $(\varphi, x)\in \mathfrak X_0$ whose sequence of directions is exactly $(d_n)_{n\geq 1}$.  In particular, $\mathfrak X_0$ has zero Medvedev degree.
\end{proposition}
\begin{proof}
    We claim that for computable $\varphi$ and $(d_n)_{n\geq 1}$, the configuration $x$ defined in \Cref{prop:prescribe-directions} is computable. By definition of $x$, it is sufficient to argue that $(p_n)_{n\geq 1}$ is a computable sequence, see \Cref{eq:definition-limiting-configuration}. The key step in the proof of \Cref{prop:prescribe-directions} whose computability is nontrivial is the choice of the sequence $(P_n)_{n\geq 1}$. That this can be done computably follows by iterating \Cref{prop:sequence-of-partitions}. The remaining choices in the algorithm given in the proof of \Cref{prop:prescribe-directions}  are deterministic and computable, given by repeatedly using the definition of decorated $2^n$-block and \Cref{prop:boundary-determination}, and by inspecting the finite list of arms and crosses.
\end{proof}
\subsection{Proof of \Cref{thm:X_1-sofic}}
We are now ready to derive \Cref{thm:X_1-sofic} from the previous results. The extra technical property that we need will follow from our analysis in \Cref{subsec:directions}. Recall that $A_0$ and $A_1$ denote the alphabets for $\mathfrak X_0$ and $\mathfrak X_1$, respectively. 
\begin{definition}
    Let $\pi\colon A_0 \to A_1$ be the function which replaces crosses and arms by the diagrams defined in \Cref{fig:alfabetito_1}, in such a way that all double arrows (with any sort of decoration) are replaced by the corresponding thick lines, and the rest of the arrows are replaced by simple lines. 
\end{definition}
\begin{proof}[Proof of~\Cref{thm:X_1-sofic}]
Comparing \Cref{fig:alfabetito_1} with Figures \ref{fig:decorated-crosses-left}, \ref{fig:decorated-crosses-right}, \ref{fig:decorated-vertical-arms}, \ref{fig:decorated-horizontal-arms-left}, \ref{fig:decorated-horizontal-arms-right}, \ref{fig:decorated-horizontal-arms-degree-3-left}, \ref{fig:decorated-horizontal-arms-degree-3-right}, and \ref{fig:decorated-special-arms}, we see that $\pi$ is surjective. Furthermore, comparing \Cref{fig:alfabetito_1} with \Cref{fig:decorated-crosses-left} and  \Cref{fig:decorated-crosses-right}, we see that $\pi$ sends decorated crosses in $A_0$ to crosses in $A_1$, and that every cross in $A_1$ is the image of a cross in $A_0$ by $\pi$.

For $(\varphi,x)\in\mathfrak X_0$, we define $\overline x$ by 
\[\overline x(w)=\pi(x(w)), \ w\in\supp(\varphi).\]
It follows from our construction of $\mathfrak X_0$ that $(\varphi,\overline x)\in \mathfrak X_1$. Thus we have a well-defined map \[\Phi\colon \mathfrak X_0\to\mathfrak X_1, \
 (\varphi,x)\mapsto(\varphi,\overline x).\]
 We argue that $\Phi$ is surjective. Fix an arbitrary arbitrary model $\varphi\in\mathcal M(\mathcal H_3)$ and consider a configuration $(\varphi,x)\in\mathfrak X_0$. Let $n\geq 1$. By \Cref{thm:hierarchical-structure-X_0}, $x$ can be seen as an infinite arrangement of decorated $2^{n+1}$-blocks. Decomposing these into decorated $2^{n}$-blocks, we can see that for every cross from \Cref{fig:decorated-crosses-left} and \Cref{fig:decorated-crosses-right} is the central cross of some decorated $2^n$-block in $x$. Applying $\Phi$, it follows that the six kinds of $2^n$-blocks from $\mathfrak X_1$ appear in $\Phi(\varphi,x)$. Since this is true for all $n\geq1$, then it follows by a standard compactness argument that $\Phi$ is surjective.

Now, let $\varphi$ be the computable model from \Cref{ex:maingraph}. By \Cref{prop:m(X_0)-is-zero}, there exists a computable $x$ with $(\varphi,x)\in \mathfrak X_0$, and such that its sequence $(d_n)_{n\geq 1}$ of directions equals the periodic sequence
\[
(\nearrow \ \swarrow \ \nwarrow \ \searrow \ )^{\infty}=\nearrow \ \swarrow \ \nwarrow \ \searrow \ \nearrow \ \swarrow \ \nwarrow \ \searrow \ \nearrow \ \swarrow \ \nwarrow \ \searrow \ \dots
\]
With this choice we have that $\Phi(\varphi,x)$ satisfies the required condition in \Cref{thm:X_1-sofic}. That is, every $w\in\supp(\varphi)$ is inside a the $4^n$-border of a $2^{2n+1}$-block, for $n$ large enough (this terminology was introduced in \Cref{sec:sofic}).
\end{proof}

\Addresses
\bibliographystyle{abbrv}
\bibliography{alexandria}

\end{document}